\documentclass[11pt,reqno]{article}
\usepackage[utf8]{inputenc}
\usepackage{geometry}
\usepackage{color}
\usepackage{amsmath}
\usepackage{amsthm}
\usepackage{amssymb}
\PassOptionsToPackage{normalem}{ulem}
\usepackage{ulem}
\usepackage[unicode=true,pdfusetitle,
 bookmarks=true,bookmarksnumbered=false,bookmarksopen=false,
 breaklinks=false,pdfborder={0 0 1},backref=page,colorlinks=true]
 {hyperref}
\hypersetup{
 linkcolor=blue, citecolor=blue}

\makeatletter
\numberwithin{equation}{section}
\numberwithin{figure}{section}
\theoremstyle{plain}
\newtheorem{thm}{\protect\theoremname}[section]
\theoremstyle{plain}
\newtheorem{cor}[thm]{\protect\corollaryname}
\theoremstyle{plain}
\newtheorem{prop}[thm]{\protect\propositionname}
\theoremstyle{plain}
\newtheorem{lem}[thm]{\protect\lemmaname}
\theoremstyle{remark}
\newtheorem{rem}[thm]{\protect\remarkname}

\usepackage{graphicx}
\usepackage{appendix}

\usepackage{enumitem}
\usepackage[backref=page]{hyperref}
\usepackage{multicol}

\makeatother

\providecommand{\corollaryname}{Corollary}
\providecommand{\lemmaname}{Lemma}
\providecommand{\propositionname}{Proposition}
\providecommand{\remarkname}{Remark}
\providecommand{\theoremname}{Theorem}

\begin{document}
\global\long\def\F{\mathrm{\mathbf{F}} }%
\global\long\def\Aut{\mathrm{Aut}}%
\global\long\def\C{\mathbf{C}}%
\global\long\def\H{\mathcal{H}}%
\global\long\def\U{\mathsf{U}}%
\global\long\def\ext{\mathrm{ext}}%
\global\long\def\hull{\mathrm{hull}}%
\global\long\def\triv{\mathrm{triv}}%
\global\long\def\Hom{\mathrm{Hom}}%

\global\long\def\trace{\mathrm{tr}}%
\global\long\def\End{\mathrm{End}}%

\global\long\def\L{\mathcal{L}}%
\global\long\def\W{\mathcal{W}}%
\global\long\def\E{\mathbb{E}}%
\global\long\def\SL{\mathrm{SL}}%
\global\long\def\R{\mathbf{R}}%
\global\long\def\Pairs{\mathrm{PowerPairs}}%
\global\long\def\Z{\mathbf{Z}}%
\global\long\def\rs{\to}%
\global\long\def\A{\mathcal{A}}%
\global\long\def\a{\mathbf{a}}%

\global\long\def\b{\mathbf{b}}%
\global\long\def\df{\mathrm{def}}%
\global\long\def\eqdf{\stackrel{\df}{=}}%
\global\long\def\ZZ{\overline{Z}}%
\global\long\def\Tr{\mathrm{Tr}}%
\global\long\def\N{\mathbf{N}}%
\global\long\def\std{\mathrm{std}}%
\global\long\def\HS{\mathrm{H.S.}}%
\global\long\def\spec{\mathrm{spec}}%
\global\long\def\Ind{\mathrm{Ind}}%
\global\long\def\half{\frac{1}{2}}%
\global\long\def\Re{\mathrm{Re}}%
\global\long\def\Im{\mathrm{Im}}%
\global\long\def\Rect{\mathrm{Rect}}%
\global\long\def\Crit{\mathrm{Crit}}%
\global\long\def\Stab{\mathrm{Stab}}%
\global\long\def\SL{\mathrm{SL}}%
\global\long\def\Tab{\mathrm{Tab}}%
\global\long\def\Cont{\mathrm{Cont}}%
\global\long\def\I{\mathcal{I}}%
\global\long\def\J{\mathcal{J}}%
\global\long\def\short{\mathrm{short}}%
\global\long\def\Id{\mathrm{Id}}%
\global\long\def\B{\mathcal{B}}%
\global\long\def\ax{\mathrm{ax}}%
\global\long\def\cox{\mathrm{cox}}%
\global\long\def\row{\mathrm{row}}%
\global\long\def\col{\mathrm{col}}%
\global\long\def\X{\mathbb{X}}%

\global\long\def\V{\mathcal{V}}%
\global\long\def\P{\mathbb{P}}%
\global\long\def\Fill{\mathsf{Fill}}%
\global\long\def\fix{\mathsf{fix}}%
\global\long\def\reg{\mathrm{reg}}%
\global\long\def\edge{E}%
\global\long\def\id{\mathrm{id}}%
\global\long\def\emb{\mathrm{emb}}%

\global\long\def\Hom{\mathrm{Hom}}%
 
\global\long\def\F{\mathrm{\mathbf{F}} }%
  
\global\long\def\pr{\mathrm{Prob} }%
 
\global\long\def\tr{\mathrm{tr} }%
\global\long\def\gs{\mathsf{GS}}%
 
\global\long\def\Xcov{{\scriptscriptstyle \overset{\twoheadrightarrow}{\mathcal{\gs}}}}%
 
\global\long\def\covers{\leq_{\Xcov}}%
\global\long\def\core{\mathrm{Core}}%
\global\long\def\pcore{\mathrm{PCore}}%
\global\long\def\im{\mathrm{imm}}%
\global\long\def\br{\mathsf{BR}}%
\global\long\def\ebs{\mathsf{EBS}}%
\global\long\def\ev{\mathrm{ev}}%
\global\long\def\CC{\mathcal{C}}%
\global\long\def\sides{\mathrm{Sides}}%
\global\long\def\tp{\mathrm{top}}%
\global\long\def\lf{\mathrm{left}}%
\global\long\def\MCG{\mathrm{MCG}}%

\global\long\def\defect{\mathrm{Defect}}%
\global\long\def\M{\mathcal{M}}%
\global\long\def\O{\mathcal{O}}%
\global\long\def\EE{\mathcal{E}}%
\global\long\def\sign{\mathrm{sign}}%

\global\long\def\v{\mathfrak{v}}%
\global\long\def\e{\mathfrak{e}}%
\global\long\def\f{\mathfrak{f}}%
\global\long\def\D{\mathfrak{D}}%
\global\long\def\d{\mathfrak{d}}%
\global\long\def\he{\mathfrak{he}}%
\global\long\def\OVB{\mathsf{OvB}}%
\global\long\def\G{\mathcal{G}}%

\global\long\def\SU{\mathsf{SU}}%
\global\long\def\p{\mathfrak{p}}%
\global\long\def\NN{\mathcal{N}}%
\global\long\def\hN{\hat{\mathcal{N}}}%
\global\long\def\Wg{\mathrm{Wg}}%
\global\long\def\Res{\mathrm{Res}}%
\global\long\def\ABG{\mathrm{ABG}}%
\global\long\def\zd{\mathrm{zd}}%
\global\long\def\SSTab{\mathcal{SST}}%
\global\long\def\FF{\mathcal{F}}%
\global\long\def\q{\mathfrak{q}}%
\global\long\def\PSU{\mathsf{PSU}}%
\global\long\def\irr{\mathrm{irr}}%
\global\long\def\su{\mathfrak{su}}%
\global\long\def\Ad{\mathrm{Ad}}%
\global\long\def\m{\mathbf{m}}%
\global\long\def\x{\mathbf{x}}%
\global\long\def\LR{\mathsf{LR}}%
\global\long\def\Q{\mathbf{Q}}%
\global\long\def\Z{\mathbf{Z}}%
\global\long\def\T{\mathcal{T}}%
\global\long\def\tkld{\dot{\T}_{_{n}}^{k,\ell}}%
\global\long\def\tkl{\T_{_{n}}^{k,\ell}}%

\title{Random Unitary Representations of Surface Groups I:\linebreak{}
Asymptotic expansions}
\author{Michael Magee}
\maketitle
\begin{abstract}
In this paper we study random representations of fundamental groups
of surfaces into special unitary groups. The random model we use is
based on a symplectic form on moduli space due to Atiyah, Bott, and
Goldman. Let $\Sigma_{g}$ denote a topological surface of genus $g\geq2$.
We establish the existence of a large $n$ asymptotic expansion, to
any fixed order, for the expected value of the trace of any fixed
element of $\pi_{1}(\Sigma_{g})$ under a random representation of
$\pi_{1}(\Sigma_{g})$ into $\SU(n)$. Each such expected value involves
a contribution from all irreducible representations of $\SU(n)$.
The main technical contribution of the paper is effective analytic
control of the entire contribution from irreducible representations
outside finite sets of carefully chosen rational families of representations.
\end{abstract}
\tableofcontents{}

\section{Introduction}

Let $g\in\N$ with $g\geq2$ and let $\Sigma_{g}$ denote a closed
topological surface of genus $g$. If $x_{0}$ is a point in $\Sigma_{g}$
then we have
\[
\pi_{1}(\Sigma_{g},x_{0})\cong\Gamma_{g}\eqdf\langle a_{1},b_{1},\ldots,a_{g},b_{g}\,|\,[a_{1},b_{1}]\cdots[a_{g},b_{g}]\rangle.
\]
The group $\Gamma_{g}$ is called the \emph{surface group} of genus
$g$. For $n\in\N$ the group $\U(n)$ is the group of $n\times n$
complex unitary matrices with respect to the standard Hermitian inner
product on $\C^{n}$. Then $\SU(n)$ is the subgroup of $\U(n)$ consisting
of matrices with unit determinant.

The space of homomorphisms $\Hom(\Gamma_{g},\SU(n))$ is given the
topology coming from the embedding
\begin{equation}
\Hom(\Gamma_{g},\SU(n))\hookrightarrow\SU(n)^{2g},\quad\phi\mapsto(\phi(a_{1}),\phi(b_{1}),\ldots,\phi(a_{g}),\phi(b_{g})).\label{eq:product-identification-hom-space}
\end{equation}
This embedding shows that $\Hom(\Gamma_{g},\SU(n))$ is an algebraic
variety, but it is a variety with singularities \cite[pg. 204 Prop.]{GOLDMANSYMPLECTIC}.
We let $\Hom(\Gamma_{g},\SU(n))^{\irr}$ denote the collection of
homomorphisms $\phi$ such that $\phi$ is irreducible as a linear
representation of $\Gamma_{g}$. The space $\Hom(\Gamma_{g},\SU(n))^{\irr}$
then inherits the structure of a smooth non-complete manifold from
(\ref{eq:product-identification-hom-space}) (\emph{ibid.)}.

There is an action of $\SU(n)$ on $\Hom(\Gamma_{g},\SU(n))$ by postcomposition
with inner automorphisms; from the point of view of (\ref{eq:product-identification-hom-space})
this is just the diagonal action of $\SU(n)$ by conjugation. This
action factors through an action of $\PSU(n)$, that is, $\SU(n)$
modulo its finite center. The quotient by this action is denoted by
$\Hom(\Gamma_{g},\SU(n))/\PSU(n)$. It is shown by Goldman in \emph{(ibid.)
}that the action of $\PSU(n)$ on $\Hom(\Gamma_{g},\SU(n))^{\irr}$
is free and the \emph{moduli space
\[
\M_{g,n}\eqdf\Hom(\Gamma_{g},\SU(n))^{\irr}/\PSU(n)
\]
}is a smooth real manifold. This moduli space is the underlying set
of \emph{random representations }of $\Gamma_{g}$ discussed in this
paper. By a theorem of Narasimhan and Seshadri \cite{NarasimhanSeshadri},
if a complex structure on $\Sigma_{g}$ is fixed, $\M_{g,n}$ corresponds
via a natural map to a moduli space of stable holomorphic rank-$n$
vector bundles on $\Sigma_{g}$.

To describe the law of the random representation we need to recall
some further results of Goldman. In \emph{(ibid}.), Goldman shows
that there is a natural symplectic form $\omega_{g,n}$ on $\M_{g,n}$,
defined up to a scalar normalization that we fix in $\S\S$\ref{subsec:The-Atiyah-Bott-Goldman-measure}.
This symplectic form arose previously in the work of Atiyah and Bott
\cite{AB}. It is analogous to the Weil-Petersson form on the Teichmüller
space of complex structures on $\Sigma_{g}$ and is defined precisely
in $\S\S$\ref{subsec:The-Atiyah-Bott-Goldman-measure} of this paper.
The symplectic form $\omega_{g,n}$ yields a volume form
\[
d\mathrm{Vol}_{\M_{g,n}}\eqdf\frac{\wedge^{\frac{1}{2}\dim\M_{g,n}}(\omega_{g,n}^{\ABG})}{(\dim\M_{g,n})!}.
\]
The random representations in this paper are sampled according to
this volume form, normalized to be a probability measure. We call
the normalized measure the \emph{Atiyah-Bott-Goldman }measure.

The statistics of random representations we are interested in come
from functions on \\
$\Hom(\Gamma_{g},\SU(n))$ that are invariant under conjugation by
$\SU(n)$. The natural functions to integrate on moduli spaces like
$\M_{g,n}$ are \emph{geometric functions} (as studied e.g. by Mirzakhani
\cite{Mirzakhani} in the Weil-Petersson context). These functions
are also called \emph{Wilson loops} in the theoretical physics literature.

We now fix a concrete instance of a family of geometric functions.
For $g\in\U(n)$ let $\tr(g)$ denote the trace of $g$ as an $n\times n$
matrix. Given any element $\gamma\in\Gamma_{g}$, we obtain a continuous
function
\[
\tr_{\gamma}:\Hom(\Gamma_{g},\SU(n))\to\C,\quad\tr_{\gamma}(\phi)\eqdf\tr(\phi(\gamma)).
\]
Clearly $\tr_{\gamma}$ is invariant under conjugation by $\SU(n)$,
and hence yields a continuous bounded function that we give the same
name
\[
\tr_{\gamma}:\M_{g,n}\to\C.
\]
In this paper we instigate a study of the expected value of $\tr_{\gamma}$,
that is, 
\begin{equation}
\E_{g,n}[\tr_{\gamma}]\eqdf\frac{\int_{\M_{g,n}}\tr_{\gamma}\,d\mathrm{Vol}_{\M_{g,n}}}{\int_{\M_{g,n}}\,d\mathrm{Vol}_{\M_{g,n}}}.\label{eq:expectation-ratio}
\end{equation}
For fixed $\gamma$, we are interested in the large $n$ behavior
of this expected value. One has the simple bound
\[
|\E_{g,n}[\tr_{\gamma}]|\leq n
\]
and this bound is attained if $\gamma=\id_{\Gamma_{g}}$. On the other
hand, if $\gamma\in\Gamma_{g}$ is not the identity, then a basic
prediction is 
\begin{equation}
\lim_{n\to\infty}\frac{|\E_{g,n}[\tr_{\gamma}]|}{n}=0.\label{eq:convergence-to-regular-rep}
\end{equation}
The significance of this prediction is that it extends a celebrated
result of Voiculescu \cite[Theorem 3.8]{Voiculescu1991} on the asymptotic
$*$-freeness of Haar unitary matrices, suitably interpreted, from
free groups to surface groups. The current paper lays the groundwork
for the proof of (\ref{eq:convergence-to-regular-rep}) in the next
paper in the series \cite{MageeRURSG2}.

Not only that, but here we will expose a separate phenomenon for the
values $\E_{g,n}[\tr_{\gamma}]$: they can be approximated to any
order $O(n^{-M})$ as $n\to\infty$ by a Laurent polynomial in $n$
depending on $\gamma$. The formal theorem is the following:
\begin{thm}
\label{thm:Main-theorem}For any $g\ge2$ and $\gamma\in\Gamma_{g}$
there is an infinite sequence of rational numbers
\[
a_{-1}(\gamma),a_{0}(\gamma),a_{1}(\gamma),a_{2}(\gamma),\ldots
\]
 such that for any $M\in\N$, as $n\to\infty$
\begin{equation}
\E_{g,n}[\tr_{\gamma}]=a_{-1}(\gamma)n+a_{0}(\gamma)+\frac{a_{1}(\gamma)}{n}+\cdots+\frac{a_{M-1}(\gamma)}{n^{M-1}}+O(n^{-M}).\label{eq:main-theorem-eq}
\end{equation}
\end{thm}

Theorem \ref{thm:Main-theorem} has the following direct corollary.
\begin{cor}
For any $\gamma\in\Gamma_{g}$ the limit 
\[
\lim_{n\to\infty}\frac{\E_{g,n}[\tr_{\gamma}]}{n}
\]
 exists.
\end{cor}

The main technical result we prove in order to establish Theorem \ref{thm:Main-theorem}
is interesting in its own right so we discuss this now. The quantity
in the denominator of (\ref{eq:expectation-ratio}), i.e., the symplectic
volume of $\M_{g,n}$, was calculated non-rigorously by Witten in
\cite{Witten1991}. Witten's result is in terms of \emph{Witten zeta
functions}, so named by Zagier in \cite{Zagier}. The \emph{Witten
zeta function} of $\SU(n)$ is defined by the series
\begin{equation}
\zeta(2g-2;n)\eqdf\sum_{\substack{\text{(\ensuremath{\rho,W)\in\widehat{\SU(n)}}}}
}\frac{1}{(\dim W)^{s}}.\label{eq:witten-def}
\end{equation}
Here $\widehat{\SU(n)}$ is the set of equivalence class of irreducible
representations of $\SU(n)$. The sum in (\ref{eq:witten-def}) converges
absolutely for $\Re(s)>\frac{2}{n}$ by a result of Larsen and Lubotzky
\cite[Thm. 5.1]{LL} (see also \cite[\S 2]{HS} for an alternative
proof of this fact). The following formula that Witten obtained was
rigorously established to hold by Sengupta in \cite{Sengupta}.
\begin{thm}[Witten's formula]
\label{thm:Witten}With the normalization of $\mathrm{Vol}_{\M_{g,n}}$
fixed as in $\S\S\ref{subsec:The-Atiyah-Bott-Goldman-measure}$, we
have 
\[
\int_{\M_{g,n}}d\mathrm{Vol}_{\M_{g,n}}=n\zeta(2g-2;n).
\]
\end{thm}

In fact Sengupta also provided a method to compute the integral of
any continuous function on $\M_{g,n}$ with respect to $\mathrm{Vol}_{\M_{g,n}}$
and this is the starting point of our work. We let $\F_{2g}\eqdf\langle a_{1},b_{1},\ldots,a_{g},b_{g}\rangle$
be the free group on the generators $a_{1},b_{1},\ldots,a_{g},b_{g}$.
Let $R_{g}\eqdf[a_{1},b_{1}]\cdots[a_{g},b_{g}]\in\F_{2g}$. Therefore
we have a surjective homomorphism $\F_{2g}\xrightarrow{q_{g}}\Gamma_{g}$
obtained from quotient by the normal subgroup generated by $R_{g}$.
We say that $w\in\F_{2g}$ represents the conjugacy class of $\gamma\in\Gamma_{g}$
if $q_{g}(w)$ is conjugate to $\gamma$ in $\Gamma_{g}$. For any
$w\in\F_{2g}$ there is a \emph{word map} $w:\SU(n)^{2g}\to\SU(n)$
obtained by substituting elements of $\SU(n)$ into the letters of
$w$. We write $d\mu_{\SU(n)^{2g}}^{\mathrm{Haar}}(x)$ for the probability
Haar measure on $\SU(n)^{2g}$, this is the product of the probability
Haar measures on the $2g$ factors.

One has the following corollary of Sengupta's main result \cite[Thm.1]{Sengupta}.
\begin{cor}
\label{cor:Fourier-expansion-of-main-integral}Let $g\geq2$ and $\gamma\in\Gamma_{g}$.
Suppose that $w\in\F_{2g}$ is an element representing the conjugacy
class of $\gamma$. Then 
\begin{equation}
\E_{g,n}[\tr_{\gamma}]=\zeta(2g-2;n)^{-1}\sum_{\substack{\text{\text{(\ensuremath{\rho,W)\in\widehat{\SU(n)}}}}}
}(\dim W)\mathcal{I}(w,\rho),\label{eq:integral-formula}
\end{equation}
where 
\begin{equation}
\mathcal{I}(w,\rho)\eqdf\int\tr(w(x))\overline{\tr(\rho(R_{g}(x)))}d\mu_{\SU(n)^{2g}}^{\mathrm{Haar}}(x)\label{eq:I-def}
\end{equation}
\uline{if} the sum on the right hand side of (\ref{eq:integral-formula})
is absolutely convergent.
\end{cor}

We explain how to obtain Corollary \ref{cor:Fourier-expansion-of-main-integral}
from \cite[Thm.1]{Sengupta} in $\S\S$\ref{subsec:The-Atiyah-Bott-Goldman-measure}
using ideas already presented in \cite{Sengupta}. Let $[\Gamma_{g},\Gamma_{g}]$
denote the commutator subgroup of $\Gamma_{g}$. Using Corollary \ref{cor:Fourier-expansion-of-main-integral}
it is not hard to show:
\begin{prop}
\label{prop:commutator-prop}If $\gamma\notin[\Gamma_{g},\Gamma_{g}]$
then there exists $n_{0}=n_{0}(\gamma)$ such that for $n\geq n_{0}$
\[
\E_{g,n}[\tr_{\gamma}]=0.
\]
\end{prop}

Proposition \ref{prop:commutator-prop} is proved in $\S\S$\ref{subsec:Integrating-over-SU(n)2g}.
This proves Theorem \ref{thm:Main-theorem} in the case that $\gamma\notin[\Gamma_{g},\Gamma_{g}]$. 

We now explain how we prove Theorem \ref{thm:Main-theorem} in general
by using Corollary \ref{cor:Fourier-expansion-of-main-integral}.
We first discuss the zeta function factor in (\ref{eq:integral-formula}).
One has the following theorem due to Guralnick, Larsen and Manack.
\begin{thm}[{\cite[Thm. 2]{GLM}}]
\label{thm:GLM}For any $s>0$, $\lim_{n\to\infty}\zeta(s;n)=1$.
\end{thm}

The limiting value arises from the trivial representation in (\ref{eq:witten-def});
it is possible to boost the methods of \cite{GLM} to show that $\zeta(2g-2;n)$
can be approximated to any order $O(n^{-M})$ by a Laurent polynomial
in $n$. In fact, the results of this paper can be viewed as a far
generalization of this result and accordingly, the just-mentioned
result is established as a byproduct of our proofs in $\S\S$\ref{subsec:Proof-of-Theorem-main}.

Thus the proof of Theorem \ref{thm:Main-theorem} amounts to showing
that 
\begin{equation}
\sum_{\substack{\text{\text{(\ensuremath{\rho,W)\in\widehat{\SU(n)}}}}}
}(\dim W)\mathcal{I}(w,\rho)\label{eq:main-sum}
\end{equation}
 is absolutely convergent and can be approximated to any order by
some Laurent polynomial. The obvious bad feature of the sum (\ref{eq:main-sum})
is that it runs over infinitely many representations of $\SU(n)$
and moreover, there are more and more of these as $n$ increases.
We aim to approximate (\ref{eq:main-sum}) by finitely many of its
terms and this requires an ordering of the representations of $\widehat{\SU(n)}$. 

The correct way to do this is as follows. For every $k,\ell\in\N$
and pair of Young diagrams $\mu\vdash k$, $\nu\vdash\ell$, with
number of rows given by $\ell(\mu)$, $\ell(\nu)$, for every $n\geq\ell(\mu)+\ell(\nu)$
there is a rational family of irreducible representations denoted
by $(\rho_{n}^{[\mu,\nu]},W_{n}^{[\mu,\nu]})\in\widehat{\SU(n)}$
defined in $\S\S$\ref{subsec:Representation-theory-of-U(n)-SU(n)}.
If we define for $B\in\N$ 
\begin{equation}
\Omega(B;n)\eqdf\{\,(\rho_{n}^{[\mu,\nu]},W_{n}^{[\mu,\nu]}):\,\ell(\mu),\ell(\nu)\leq B,\mu_{1},\nu_{1}\leq B^{2}\,\}\label{eq:omega-def}
\end{equation}
then $\Omega(B;n)$ is in one-to-one correspondence with the $(\mu,\nu)$
such that $\ell(\mu),\ell(\nu)\leq B,\mu_{1},\nu_{1}\leq B^{2}$ when
$n$ is sufficiently large. This specific choice of $\Omega(B;n)$
is for technical convenience, becoming useful in $\S\S$\ref{subsec:Proof-of-Theorem-main}.

We prove bounds on the $\mathcal{I}(w,\rho)$ in Theorem \ref{thm:single-dimension-bound}.
The main challenges are that not only that estimates for $\mathcal{I}(w,\rho)$
must overcome the weights $\dim W$ in (\ref{eq:main-sum}) but also
that these bounds must remain effective for $\dim W$ much larger
than $n$. It is quite well understood that matrix integrals such
as $\mathcal{I}(w,\rho)$ are challenging in this regime as the main
method of performing such integrals, known as the \emph{Weingarten
calculus}, often fail to produce understandable answers there. This
is because the Weingarten function $\Wg_{n,k}$ defined in (\ref{eq:Wg-def})
becomes increasingly complicated for $k\gg n$, drawing on more and
more different representations of large symmetric groups. We overcome
this inherent difficulty as follows. 

Firstly there is a minor observation that in all cases of interest,
$\SU(n)$ can be replaced by $\U(n)$ in (\ref{eq:I-def}) (Proposition
\ref{prop:using-U(n)-for-fourier-coefficients}). The main idea is
then that after some splitting up, parts of $\mathcal{I}(w,\rho)$
can be evaluated by integrating first $\overline{\tr(\rho(R_{g}(x)))}$
over all double cosets for a very large subgroup $\U(n-\D)\leq\U(n)$
where $\D$ is bounded depending only on $w$, which is fixed. During
this first integration, the structure of the word $R_{g}$ can be
exploited to produce a lot of cancellation. 

After this initial integral, we then apply the Weingarten calculus
through a novel strategy (cf. $\S\S$\ref{subsec:Second-integration:-overview})
making heavy use of representation theory of both symmetric groups
and $\U(n)$. What we achieve is the following technical result.
\begin{thm}
\label{thm:high-dim-sum}Suppose that $g\geq2$, $w\in\F_{2g}$ and
$B\in\N$.
\begin{enumerate}
\item For $n\geq n_{0}(w,g)$, the sum in (\ref{eq:integral-formula}) is
absolutely convergent.
\item As $n\to\infty$ 
\begin{equation}
\sum_{\substack{\text{\text{(\ensuremath{\rho,W)\in\widehat{\SU(n)}}\ensuremath{\backslash}\ensuremath{\Omega(B;n)}}}}
}(\dim W)\mathcal{I}(w,\rho)\ll_{B,w,g}n^{|w|}n^{-2\log B}\label{eq:tail-bound}
\end{equation}
where $|w|$ denotes the length of $w$ as a reduced word.
\end{enumerate}
\end{thm}

The point of (\ref{eq:tail-bound}) is not the exact form of the right
hand side, but rather, that it gives effective control of the tail.
Theorem \ref{thm:high-dim-sum} shows that by taking $B$ sufficiently
large and fixed depending on $w$, the contribution to $\E_{g,n}[\tr_{\gamma}]$
from $(\ensuremath{\rho,W)\in\widehat{\SU(n)}}\backslash\ensuremath{\Omega(B;n)}$
can be made to decay as fast as any $n^{-M}$, for $M\in\N$. We have
the following direct corollary of Theorem \ref{thm:high-dim-sum},
Theorem \ref{thm:GLM}, and Corollary \ref{cor:Fourier-expansion-of-main-integral}.
\begin{cor}
Suppose that $g\geq2$, $\gamma\in\Gamma_{g}$, and $w\in\F_{2g}$
represents the conjugacy class of $\gamma$. For any $B\in\N$ we
have as $n\to\infty$
\[
\E_{g,n}[\tr_{\gamma}]=\zeta(2g-2;n)^{-1}\sum_{\substack{\text{\text{(\ensuremath{\rho,W)\in}\ensuremath{\Omega(B;n)}}}}
}(\dim W)\mathcal{I}(w,\rho)+O_{B,w,g}\left(n^{|w|}n^{-2\log B}\right).
\]
 
\end{cor}

\subsection{Related works I: Spaces of representations}

The existence of an asymptotic expansion of $\E_{g,n}[\tr_{\gamma}]$
as in Theorem \ref{thm:Main-theorem} follows a long line of related
results. The most closely related of these is the analog of Theorem
\ref{thm:Main-theorem} when $\SU(n)$ is replaced by the family of
symmetric groups $S_{n}$. For $\pi\in S_{n}$ let $\fix(\rho)$ denote
the number of fixed points of $\pi$, and for $\gamma\in\Gamma_{g}$
let $\fix_{\gamma}:\Hom(\Gamma_{g},S_{n})\to\N$ be the function $\fix_{\gamma}(\phi)\eqdf\fix(\phi(\gamma))$.
The representation space $\Hom(\Gamma_{g},S_{n})$ is finite and we
let $\E_{g,S_{n}}[\fix_{\gamma}]$ denote the expected value of $\fix_{\gamma}$
with respect to the uniform probability measure on $\Hom(\Gamma_{g},S_{n})$.
An exactly analogous result to Theorem \ref{thm:Main-theorem} for
$\E_{g,S_{n}}[\fix_{\gamma}]$ was established by the author and Puder
in \cite[Thm. 1.1]{MPasympcover}.

Similarly, if instead of using a surface group $\Gamma_{g}$, we consider
a free group $\F_{r}$ with $r\geq2$, for any compact Lie group $G$
the representation space $\Hom(\F_{r},G)$ can be identified with
$G^{r}$ and hence can be given the corresponding probability Haar
measure. If $G$ is finite, this is simply the uniform probability
measure. For any character $\chi$ of $G$ and $w\in\F_{r}$ we obtain
a function
\[
\chi_{w}:\Hom(\F_{r},G)\to\C,\quad\chi_{w}(\phi)\eqdf\chi(\phi(w)).
\]
Then we can define $\E_{\F_{r},G}[\chi_{w}]$ to be the expected value
of $\chi_{w}$ with respect to the Haar probability measure. Not only
is the analog of Theorem \ref{thm:Main-theorem} true for many natural
families of $(G(n),\chi(n))$, but actually, in the case of free groups,
$\E_{\F_{r},G(n)}[\chi(n)_{\gamma}]$ is a \emph{rational} function
of $n$ for $n$ sufficiently large. Indeed, for fixed $w\in\F_{r}$,
$\E_{\F_{r},G(n)}[\chi(n)_{\gamma}]$ agrees with a rational function
of $n$ for $n\gg_{w}1$ when 
\begin{itemize}
\item $G(n)=S_{n}$ and $\chi(n)=\fix$ \cite{nica1994number,Linial2010}
\item $G(n)$ is a family of generalized symmetric groups, e.g. hyperoctahedral
groups, and $\chi(n)$ is the trace in a natural defining representation
\cite{mp2019surface}
\item $G(n)=\U(n)$ and $\chi(n)=\tr$ \cite{MPunitary}
\item $G(n)=\mathsf{O}(n)$ or $\mathsf{Sp}(n)$ and $\chi(n)=\tr$ \cite{mp2019orthogonal}.
\end{itemize}

\subsection{Related works II: 2D Yang-Mills theory}

The expected values $\E_{g,n}[\tr_{\gamma}]$ are very closely connected
with the expected value of Wilson loops in 2D Yang-Mills theory. We
briefly explain these connections and some prior work done by theoretical
physicists in the area. It is explained by Sengupta in \cite[Appendix A]{Sengupta}
that if $\Sigma_{g}$ is endowed with a Riemannian metric with associated
area $A=A(\Sigma_{g})$ then
\begin{equation}
\sum_{\substack{\text{\ensuremath{(\rho,W)\in\widehat{\SU(n)}}}}
}e^{-\frac{C(\rho)tA}{2n}}(\dim W)\,\I(w,\rho)\label{eq:Yang-Mills}
\end{equation}
is, in the context of a quantum $\SU(n)$ Yang-Mills theory on $\Sigma_{g}$,
a heuristic definition of the expected value of the Wilson loop measuring
the trace of the holonomy around the loop in $\Sigma_{g}$ that $w\in\F_{2g}$
represents. Here $C(\rho)$ is the Casimir eigenvalue of the representation
$(\rho,W)$ and $t$ is a coupling constant. (We have inserted the
factor $\frac{1}{n}$ in the exponent above so that it matches with
e.g. \cite{Gross_1993}.) It is worth pointing out that the methods
of the current paper should also allow one to effectively and rigorously
approximate (\ref{eq:Yang-Mills}) although this is not pursued here.

The emphasis in the physics literature is not on the rigorous analytic
approximation of integrals such as (\ref{eq:Yang-Mills}), but rather,
on the interpretation of (\ref{eq:Yang-Mills}) as a formal power
series and then reordering the terms and truncating in a formal way.
These manipulations are not intended as having rigorous mathematical
consequences. For example, as we understand, none of the `$\frac{1}{n}$-expansions'
of (\ref{eq:Yang-Mills}) obtained before in any sense rigorously
approximate (\ref{eq:Yang-Mills}); nonetheless, they are significant
to physicists. Since the values $\E_{g,n}[\tr_{\gamma}]$ that we
focus on here correspond to the $t=0$ case of (\ref{eq:Yang-Mills})
via Corollary \ref{cor:Fourier-expansion-of-main-integral}, we briefly
survey what is known to physicists for general $t$, with the disclaimer
that the author is by no means an expert in the concepts of theoretical
physics.

In physics literature, the \emph{chiral} expansion means that the
$(\rho,W)$ are parameterized by $(\rho_{n}^{\lambda},W_{n}^{\lambda})$
where $\lambda$ runs over Young diagrams. For this parameterization
to work (cf. $\S\S\ref{subsec:Representation-theory-of-U(n)-SU(n)}$)
one should restrict to Young diagrams with less than $n$ rows. This
is referred to as the \emph{finite-n} expansion. However, in some
cases this restriction is lifted and (\ref{eq:Yang-Mills}) is interpreted
as a sum over \emph{all} Young diagrams. This is called the \emph{large-n}
expansion. In the chiral expansion, the Young diagrams are ordered
by the number of boxes that they contain.

The \emph{partition function }of the quantum Yang-Mills theory corresponds
to (\ref{eq:Yang-Mills}) in the case $w=\id$, up to a factor $\frac{1}{n}$,
and is given by \cite[2.4]{Gross_1993}
\[
Z(G,tA,n)\eqdf\sum_{\substack{\text{\ensuremath{(\rho,W)\in\widehat{\SU(n)}}}}
}e^{-\frac{C(\rho)tA}{2n}}\frac{1}{(\dim W)^{2g-2}}.
\]

A large-$n$ chiral expansion of the partition function was obtained
by Gross and Taylor in \cite{Gross_1993}. The coefficients of this
expansion are interpreted in terms of branched covers of surfaces
and from this Gross and Taylor deduce their titular statement that
`Two dimensional QCD is a String Theory'. A finite-$n$ chiral expansion
of the partition function in terms of branched covers with some extra
data was obtained by Baez and Taylor in \cite{BT}. In \cite{RAMGOOLAM_1996},
Ramgoolam gives a large-$n$ chiral expansion of (\ref{eq:Yang-Mills})
in terms of branched covers of surfaces. 

In the language of these papers, the expansion we obtain in Theorem
\ref{thm:Main-theorem} is a finite-$n$, fully non-chiral expansion
of the expected value of a Wilson loop, when the coupling constant
is set to zero. The main point is that this asymptotic expansion is
established rigorously through Theorem \ref{thm:high-dim-sum}. 

\subsubsection*{Notation}

We write $\N$ for the natural numbers, $\N_{0}\eqdf\N\cup\{0\}$,
and $\Q$ denotes the rationals. We write $[n]\eqdf\{1,\ldots,n\}$
for $n\in\N$ and $[k,\ell]\eqdf\{k,k+1,\ldots,\ell\}$ for $k,\ell\in\N$,
$k\leq\ell$. If $A$ and $B$ are sets $A\backslash B$ is the set
of elements of $A$ that are not in $B$. We write 
\[
(n)_{\ell}\eqdf n(n-1)\cdots(n-\ell).
\]
We let $e(\theta)\eqdf\exp(2\pi i\theta)$ for $\theta\in\R$. If
$G$ is a group, and $g_{1},g_{2}\in G$, we let $[g_{1},g_{2}]\eqdf g_{1}g_{2}g_{1}^{-1}g_{2}^{-1}$.
We write $[G,G]$ for the subgroup of $G$ generated by elements of
the form $[g_{1},g_{2}]$ for $g_{1},g_{2}\in G$. If $V$ is a complex
vector space and $q\in\N$ we let 
\[
V^{\otimes q}\eqdf\underbrace{V\otimes_{\C}\cdots\otimes_{\C}V}_{q};
\]
 in general if we write a tensor product without explicit subscript
it is over $\C$. We write $\Q(t)$ for the ring of rational functions
in an indeterminate $t$; i.e. ratios of polynomials.

We use Vinogradov notation as follows. If $f$ and $h$ are functions
of $n\in\N$, we write $f\ll h$ to mean that there are constants
$n_{0}\geq0$ and $C_{0}\geq0$ such that for $n\geq n_{0}$, $f(n)\leq C_{0}h(n)$.
We write $f=O(h)$ to mean $f\ll|h|$. We write $f\asymp h$ to mean
both $f\ll h$ and $h\ll f$. If in any of these statements the implied
constants depend on additional parameters we add these parameters
as subscript to $\ll,O,$ or $\asymp$. Throughout the paper we view
the genus $g$ as fixed and so any implied constant may depend on
$g$. 

\subsubsection*{Acknowledgments}

We thank Benoît Collins, Doron Puder, Sanjaye Ramgoolam, and Calum
Shearer for discussions about this work. This project has received
funding from the European Research Council (ERC) under the European
Union’s Horizon 2020 research and innovation programme (grant agreement
No 949143).

\section{Background}

\subsection{Young diagrams and tableaux\label{subsec:Young-diagrams-and}}

\subsubsection*{Young diagrams}

A \emph{Young diagram }(YD) is a collection of left-aligned rows of
identical square boxes in the plane, where the number of boxes in
each row is non-increasing from top to bottom. Any YD $\lambda$ also
gives a non-increasing sequence of natural numbers $(\lambda_{1},\lambda_{2},\ldots,\lambda_{\ell(\lambda)})$
where $\lambda_{i}$ is the number of boxes in the $i$\textsuperscript{th}
(from top to bottom) row of $\lambda$, and $\ell(\lambda)$ is the
number of rows of $\lambda$. A finite non-increasing sequence of
natural numbers is called a \emph{partition}. We think of partitions
and YDs interchangeably in this paper via the above correspondence.
For example, the partition $(k)$ corresponds to the Young diagram
with one row consisting of $k$ boxes. The empty YD with no boxes
is denoted by $\emptyset$. The \emph{size }of a YD $\lambda$ is
the number of boxes that it contains, or $\sum_{i=1}^{\ell(\lambda)}\lambda_{i}$.
The size of $\lambda$ is denoted by $|\lambda|$ and the statement
$|\lambda|=k$ is sometimes written $\lambda\vdash k$.

Given two Young diagrams $\lambda$ and $\mu$, we say $\mu\subset\lambda$
if every box of $\mu$ is a box of $\lambda$. A \emph{skew Young
diagram }(SYD)\emph{ }is a pair $\lambda,\mu$ of Young diagram such
that $\mu\subset\lambda$. This is usually written as $\lambda/\mu$
and $\lambda/\mu$ is thought of as the collection of boxes of $\lambda$
that are not boxes of $\mu$. A Young diagram $\lambda$ is identified
with $\lambda/\emptyset$ and in this way YDs are special cases of
SYDs.

We will have use for the following relations between Young diagrams.
We say $\mu\subset_{k}\lambda$ if $\mu\subset\lambda$ and $\lambda/\mu$
contains $k$ boxes. We say $\mu\subset^{1}\lambda$ if $\mu\subset\lambda$
and no two boxes of $\lambda/\mu$ are in the same column. We say
$\mu\subset^{r}\lambda$ if there is a sequence $\mu_{1},\ldots,\mu_{r-1}$
of YDs such that
\[
\mu\subset^{1}\mu_{1}\subset^{1}\cdots\subset^{1}\mu_{r-1}\subset^{1}\lambda.
\]
Finally, we write $\mu\subset_{k}^{r}\lambda$ if both $\mu\subset_{k}\lambda$
and $\mu\subset^{r}\lambda$.

\subsubsection*{Semistandard Young tableaux}

In the rest of the paper we use the notation $[n]\eqdf\{1,\ldots,n\}$
for $n\in\N$ and $[k,\ell]\eqdf\{k,k+1,\ldots,\ell\}$ for $k,\ell\in\N$.

Given a SYD $\lambda/\mu$ (which may in fact be a YD) and a subset
$S\subset\N$, a \emph{semistandard tableau} of shape $\lambda/\mu$
with entries in $S$ is a filling of the boxes of $\lambda/\mu$ with
the numbers of $S$ such that the numbers in the boxes are strictly
increasing along columns from top to bottom and non-strictly increasing
along rows from left to right. If $\lambda/\mu$ is a SYD we write
\[
\SSTab_{[k,\ell]}(\lambda/\mu)
\]
 for the semistandard tableaux of shape $\lambda/\mu$ with entries
in $[k,\ell]$. We also use all obvious variants of this notation,
e.g. for YD $\lambda$ and $n\in\N$, $\SSTab_{[n]}(\lambda)$ is
the collection of semistandard tableaux of shape $\lambda$ with entries
in $[n]$. There is by convention a unique semistandard tableau (with
any entry set) of shape the empty YD (or an empty SYD).

If $\lambda/\mu$ is a SYD, $n>m$, $T_{1}\in\SSTab_{[m]}(\mu)$ and
$T_{2}\in\SSTab_{[m+1,n]}(\lambda/\mu)$, we write $T_{1}\sqcup T_{2}$
for the semistandard tableau obtained by adjoining the numbers-in-boxes
of $T_{2}$ to those of $T_{1}$. It is easy to see this is always
a valid semistandard tableau of shape $\lambda$.

Given $T\in\SSTab_{[k,m]}(\lambda/\mu)$, the \emph{weight }of $T$
is the function
\[
\omega_{T}:[k,m]\to\N_{0},
\]
where $\omega_{T}(j)$ is the number of occurrences of $j$ in $T$.

Given $T\in\SSTab_{[n]}(\lambda)$, $\lambda$ a YD, and $m\in\N_{0}$
with $m\leq n$, we write $T\lvert_{>m}$ for the skew semistandard
tableaux formed by the numbers-in-boxes of $T$ with numbers $>m$,
and similarly write $T\lvert_{\leq m}$ for the numbers-in-boxes of
$T$ with numbers $\leq m$. 

\subsection{General representation theory\label{subsec:General-representation-theory}}

Here we clarify the language of representation theory that will appear
throughout the paper.

A unitary representation of a compact Lie group\footnote{Note that compact Lie groups include all finite groups.}
$G$ consists of a Hilbert space\footnote{In this paper, all Hilbert spaces appearing are finite dimensional.}
$\H$ and a homomorphism $\rho_{1},G\to\U(\H)$ where $\U(\H)$ is
the group of unitary operators on $\H$. Any compact Lie group has
a trivial unitary representation $(\triv_{G},\C)$ with inner product
on $\C$ given by $\langle z_{1},z_{2}\rangle=z_{1}\overline{z_{2}}$
and $\triv_{G}(g)=1$ for all $g\in G$. 

If $(\rho_{1},\H_{1})$ and $(\rho_{2},\H_{2})$ are two finite-dimensional
representations of a compact Lie group $G$, an \emph{intertwiner}
between $(\rho_{1},\H_{1})$ and $(\rho_{2},\H_{2})$ is a \uline{linear}
map $I:\H_{1}\to\H_{2}$ such that for all $g\in G$, $I\rho_{1}(g)=\rho_{2}(g)I$.
If there is an invertible intertwiner between $(\rho_{1},\H_{1})$
and $(\rho_{2},\H_{2})$ we say $(\rho_{1},\H_{1})$ and $(\rho_{2},\H_{2})$
are \emph{linearly isomorphic. }If there is an invertible isometric
intertwiner between the two we say they are unitarily equivalent or
isomorphic as unitary representations.

In many cases, the maps $\rho_{1}$ and $\rho_{2}$ will be tacitly
inferred from the group $G$ and the spaces $\H_{1}$ and $\H_{2}$.
When this is the case, we write
\[
\Hom_{G}(\H_{1},\H_{2})
\]
 for the vector space of linear intertwiners between $(\rho_{1},\H_{1})$
and $(\rho_{2},\H_{2})$. 

For $(\rho,W)$ a unitary representation of a compact Lie group $G$,
and $H\le G$ a compact Lie subgroup with unitary representation $(\pi,V)$,
we define the \emph{$(\pi,V)$-isotypic subspace of $W$ for the subgroup
$H$} to be the subspace spanned by the images of all elements of
\[
\Hom_{H}(V,W)\eqdf\Hom_{H}(V,\Res_{H}^{G}W)
\]
where $\Res_{H}^{G}W$ is the \emph{restriction} of $\rho$ to $H$.
Any such isotypic space is itself a unitary sub-representation of
$\Res_{H}^{G}W$ (for the group $H$). If $V$ is an irreducible representation
of $H$, the dimension of $\Hom_{H}(V,W)$ is the \emph{multiplicity
}with which $V$ appears in $\Res_{H}^{G}W$.

Let $(\pi,W)$, $(\pi_{1},W_{1})$ and $(\pi_{2},W_{2})$ be finite-dimensional
unitary representations of a compact Lie group. We explain some basic
constructions.

There is a dual unitary representation on the space of complex linear
functionals on $W$ with inner product induced by that on $W$; this
representation is also irreducible if $(\pi,W)$ is. The action of
$g$ on a linear functional $\varphi$ on $W$ is by $g[\varphi](w)=\varphi(g^{-1}w)$.
If $w\in W$ we write $\check{w}\in\check{W}$ for the linear functional
\[
\check{w}:v\mapsto\langle v,w\rangle.
\]
We denote the dual representation by $(\check{\pi},\check{W})$ or
simply $\check{W}$. Given two finite dimensional unitary representations
$(\pi_{1},W_{1})$ and $(\pi_{2},W_{2})$ of a compact Lie group $G$,
the tensor product $W_{1}\otimes W_{2}\eqdf W_{1}\otimes_{\C}W_{2}$
has an action of $G$ that is `diagonal' by $\pi_{1}$ on the first
factor and $\pi_{2}$ on the second factor. The tensor product inherits
a Hermitian inner product from that on $W_{1}$ and $W_{2}$ where
$\langle w_{1}\otimes w_{2},w'_{1}\otimes w'_{2}\rangle\eqdf\langle w_{1},w'_{1}\rangle\langle w_{2},w'_{2}\rangle$
that makes $W_{1}\otimes W_{2}$ a unitary representation of $G$
under the diagonal action. This extends to tensor powers of $W$;
any $W^{\otimes k}$ for $k\in\N$ is in this way a unitary representation
of $G$ under the diagonal action.

There is a canonical isomorphism 
\begin{equation}
W_{1}\otimes\check{W_{2}}\cong\Hom(W_{2},W_{1})\label{eq:hom}
\end{equation}
of linear representations, where the right hand side is linear maps
from $W_{2}$ to $W_{1}$, where $G$ acts `diagonally' on the left
hand side, and by conjugation ($g:A\mapsto\pi_{1}(g)A\pi_{2}(g)^{-1}$)
on the right hand side. If $W_{1}',W_{2}'$ are subrepresentations
of $W_{1},W_{2}$ then we view $\Hom(W'_{2},W'_{1})$ as a subrepresentation
of $\Hom(W_{2},W_{1})$ via (\ref{eq:hom}). This corresponds to extending
linear maps by 0 on the orthogonal complement of $W'_{2}$ in $W_{2}$.

In the case $W_{1}=W_{2}=W$ (\ref{eq:hom}) is moreover an isomorphism
of unitary representations $W\otimes\check{W}\cong\End(W)$ if $\End(W)$
is given the Hilbert-Schmidt inner product $\langle A,B\rangle\eqdf\tr(AB^{*})$,
$B^{*}$ standing for the Hermitian transpose of $B$. 

\subsection{Representation theory of symmetric groups\label{subsec:Representation-theory-of-sym-groups}}

Although the problems of this paper are not initially posed in a way
that involves symmetric groups, the representation theory of unitary
groups is intimately connected via Schur-Weyl duality (see $\S$\ref{subsec:Representation-theory-of-U(n)-SU(n)})
to the representation theory of symmetric groups, so it plays a large
part in this paper.

We write $S_{k}$ for the symmetric group of permutations of the set
$[k]$, and $\C[S_{k}]$ for the group algebra of $S_{k}$. As a technicality,
the group $S_{0}$ is the group with one element.

The equivalence classes of irreducible representations of $S_{k}$
are in one-to-one correspondence with YDs $\lambda\vdash k$ \cite[\S 4.2]{FH}.
The irreducible unitary representation of $S_{k}$ corresponding to
$\lambda\vdash k$ will be denoted by $(\pi_{\lambda},V^{\lambda})$
and simply referred to as $V^{\lambda}$. We write $\chi_{\lambda}$
for the character of $V^{\lambda}$, i.e.
\[
\chi_{\lambda}(\sigma)\eqdf\tr(\pi_{\lambda}(\sigma)),\quad\sigma\in S_{k}.
\]
All characters of irreducible representations of symmetric groups
are real-valued (e.g. by \cite[Frobenius Formula 4.10]{FH}).

The dimension of $V^{\lambda}$ is given by the Frame-Robinson-Thrall
hook-length formula as follows. The \emph{hook }of a box $\square$
in $\lambda\vdash k$ is the collection of boxes either to the right
of, or below, $\lambda$, including the box itself. We write $h_{\lambda}(\square)$
for the number of boxes in the hook of $\square$. The hook-length
formula \cite{FRT} states
\begin{equation}
d_{\lambda}\eqdf\dim V^{\lambda}=\frac{k!}{\prod_{\square\in\lambda}h_{\lambda}(\square)}.\label{eq:hook-length-formula}
\end{equation}

Before proceeding we fix some notation. If we refer to $S_{\ell}\leq S_{k}$
with $\ell\leq k$ we always view $S_{\ell}$ as the subgroup of permutations
that fix every element of $[\ell+1,k]$. As a consequence we obtain
fixed inclusions $\C[S_{\ell}]\subset\C[S_{k}]$ for $\ell$ and $k$
as above. When we write $S_{\ell}\times S_{k-\ell}\leq S_{k}$, the
first factor is $S_{\ell}$ as defined above and the second factor
is $S'_{k-\ell}$ which is our notation for the subgroup of permutations
that fix every element of $[\ell]$.

Given $\lambda\vdash\ell$, the element
\[
\p_{\lambda}\eqdf\frac{d_{\lambda}}{\ell!}\sum_{\sigma\in S_{\ell}}\chi_{\lambda}(\sigma)\sigma\in\C[S_{\ell}]
\]
 is a central idempotent in $\C[S_{\ell}]$ with the following important
property. If $(\pi,V)$ is any unitary irreducible representation
of $S_{k}$ with $k\geq\ell$, then by linear extension $\pi:\C[S_{k}]\to\End(V)$.
Under the fixed inclusion $\C[S_{\ell}]\subset\C[S_{k}]$, $\pi(\p_{\lambda})$
is the orthogonal projection onto the $V^{\lambda}$-isotypic subspace
of $V$ for the subgroup $S_{k}$.

Suppose that $\ell\leq k$, $\ell,k\in\N_{0}$, with $\lambda\vdash k$
and $\mu\vdash\ell$. We write 
\[
d_{\lambda/\mu}\eqdf\dim\Hom_{S_{\ell}}(V^{\mu},V^{\lambda}).
\]
The branching rules for $S_{k}$ imply that $d_{\lambda/\mu}=0$ unless
$\mu\subset\lambda$. Applying Frobenius reciprocity to the pair $(V^{\mu},V^{\lambda})$
then taking the dimension of the induced representation $\Ind_{S_{\ell}}^{S_{k}}V^{\mu}$
gives the formula, for $\mu\vdash\ell$ fixed and $b\in\N_{0}$,
\begin{equation}
\sum_{\mu\subset_{b}\lambda}d_{\lambda/\mu}d_{\lambda}=\frac{(\ell+b)!}{\ell!}d_{\mu}.\label{eq:induced-rep-formula}
\end{equation}

Suppose that $\ell_{1},\ell_{2}\in\N_{0}$ and $\ell_{1}+\ell_{2}=k$.
The irreducible representations of $S_{\ell_{1}}\times S_{\ell_{2}}$
are of the form $V^{\mu_{1}}\otimes V^{\mu_{2}}$ with $\mu_{i}\vdash\ell_{i}$
for $i=1,2$; $S_{\ell_{1}}$ acts on the first factor and $S'_{\ell_{2}}$
acts on the second factor. The numbers
\[
\LR_{\mu_{1},\mu_{2}}^{\lambda}\eqdf\dim\Hom_{S_{\ell_{1}}\times S_{\ell_{2}}}(V^{\mu_{1}}\otimes V^{\mu_{2}},V^{\lambda})\in\N_{0}
\]
are called \emph{Littlewood-Richardson coefficients.} They are notoriously
difficult to work with, but thankfully, in this paper the most detail
we need about them is the following:
\begin{lem}[Recast of Pieri's formula]
\label{lem:pieri-formula}Suppose that $\ell_{1},\ell_{2}\in\N_{0}$
and $\ell_{1}+\ell_{2}=k$. Let $\mu\vdash\ell_{1}$ and $\lambda\vdash k$
with $\mu\subset^{1}\lambda$. We have 
\[
\LR_{\mu,(\ell_{2})}^{\lambda}=\dim\Hom_{S_{\ell_{1}}\times S_{\ell_{2}}}(V^{\mu}\otimes\triv_{S_{\ell_{2}}},V^{\lambda})=1.
\]
\end{lem}

\begin{proof}
This is a standard fact but normally presented slightly differently.
To obtain this version, one needs to know that $\LR_{\mu,(\ell_{2})}^{\lambda}$
is also the coefficient of the Schur polynomial $s_{\mu}$ in the
expansion of $s_{\lambda}s_{(\ell_{2})}$ in Schur polynomials (see
\cite[eq. (A.8), Ex. 4.43]{FH} for the equivalence of these definitions
of Littlewood-Richardson coefficients). Then the lemma follows from
the version of Pieri's formula for $s_{\lambda}s_{(\ell_{2})}$ given
in \cite[eq. (A.7)]{FH}.
\end{proof}

\subsection{Representation theory of $\protect\U(n)$ and $\protect\SU(n)$ \label{subsec:Representation-theory-of-U(n)-SU(n)}}

Here we give a brief account of the representation theory of $\U(n)$
and $\SU(n)$. The details as well as more background can be found
in \cite{FH}. 

The equivalence classes of irreducible representations of $\U(n)$
are in one-to-one correspondence with their highest weights by the
theorem of the highest weight. These highest weights are given by
linear combinations

\[
\lambda_{1}\hat{\omega}_{1}+\cdots+\lambda_{n}\hat{\omega_{n}}
\]
where $(\lambda_{1},\ldots,\lambda_{n})\in\Z^{n}$ is a non-increasing
sequence of integers and $\hat{\omega}_{1},\ldots,\hat{\omega}_{n}$
is a system of fundamental weights for $\U(n)$. The sequence $(\lambda_{1},\ldots,\lambda_{n})$
is called the \emph{signature} of the representation. If all $\lambda_{i}\in\Z_{\geq0}$
then the signature corresponds to a Young diagram $\lambda$ with
$\ell(\lambda)\leq n$, and conversely, any YD $\lambda$ with $\ell(\lambda)\leq n$
gives rise to an equivalence class of irreducible representation of
$\U(n)$ denoted by $(\tilde{\rho}_{n}^{\lambda},W_{n}^{\lambda})$.

Every irreducible representation of $\U(n)$ restricts to an irreducible
representation of $\SU(n)$, and every irreducible representation
of $\SU(n)$ arises by restriction from $\U(n)$. Two irreducible
representations of $\U(n)$ restrict to equivalent representations
of $\SU(n)$ if and only if their signatures differ by an integer
multiple of $(1,1,\ldots,1)$. This immediately shows that the equivalence
classes of irreducible representations of $\SU(n)$ are in one to
one correspondence with YDs $\lambda$ such that $\ell(\lambda)\leq n-1$,
since there is a unique signature $(\lambda_{1},\ldots,\lambda_{n})$
whose representation restricts to the given one of $\SU(n)$ with
$\lambda_{n}=0$ and this corresponds to a YD by deleting trailing
zeros. For such a $\lambda,$ we write $(\rho_{n}^{\lambda},W_{n}^{\lambda})$
for the restriction of this representation of $\SU(n)$. 

\emph{To be less verbose, in the rest of the paper we will tend to
refer to representations simply by the vector space whenever the module
structure can be inferred.}

The trace of $g\in\U(n)$ on $W_{n}^{\lambda}$ is given by $s_{\lambda}(e(\theta_{1}),\ldots,e(\theta_{n}))$
where $s_{\lambda}$ is the \emph{Schur polynomial} associated to
the YD $\lambda$ and $e(\theta_{1}),\ldots,e(\theta_{n})$ are the
eigenvalues of $g$. As such, the dimension of $W_{n}^{\lambda}$
is given by the specialization of the Schur polynomial, letting $1^{n}\eqdf(\underbrace{1,1,\ldots,1}_{n})$
\begin{equation}
D_{\lambda}(n)\eqdf\dim W_{n}^{\lambda}=s_{\lambda}(1^{n}).\label{eq:dimension-given-by-semistandard-tableaux}
\end{equation}
There is a formula for $s_{\lambda}(1^{n})$ due to Stanley called
the \emph{hook-content }formula. Given a YD $\lambda$, for any box
$\square$ of $\lambda$, we define the \emph{content} of the box
to be
\[
c(\square)\eqdf j(\square)-i(\square)
\]
where $i(\square)$ is the row number (starting at 1, counting from
top to bottom) of the box and $j(\square)$ is the column number (starting
at 1, counting from left to right). Recall the quantities $h_{\lambda}(\square)$
from $\S\S$\ref{subsec:Representation-theory-of-sym-groups}. The
hook-content formula \cite[Cor. 7.21.4]{StanleyEC2} says
\begin{equation}
D_{\lambda}(n)=s_{\lambda}(1^{n})=\frac{\prod_{\square\in\lambda}(n+c(\square))}{\prod_{\square\in\lambda}h_{\lambda}(\square)}.\label{eq:hook-content-formula}
\end{equation}
By a slight abuse of notation, if $g$ is any element of $\U(n)$,
we also write $s_{\lambda}(g)$ for $\tr_{W_{n}^{\lambda}}(\tilde{\rho}_{\lambda}(g))$;
strictly speaking this is the Schur polynomial evaluated at the eigenvalues
of $g$. The hook-content formula (\ref{eq:hook-content-formula})
implies that for a fixed YD $\lambda$, $D_{\lambda}(n)$ is a polynomial
function of $n$ with coefficients in $\Q$.

There is an important formula expressing $s_{\lambda}(g)$ in terms
of power sum symmetric functions. Given a partition $\lambda=(\lambda_{1},\ldots,\lambda_{r})$,
and $g\in\U(n)$, we define
\[
p_{\lambda}(g)\eqdf\tr(g^{\lambda_{1}})\tr(g^{\lambda_{2}})\cdots\tr(g^{\lambda_{r}}).
\]
 Given $\lambda\vdash k$, the change of base formula is \cite[Ex. A.29]{FH}
\begin{equation}
s_{\lambda}(g)=\frac{1}{k!}\sum_{\pi\in S_{k}}\chi_{\lambda}(\pi)p_{\mu(\pi)}(g)\label{eq:base-change-formula}
\end{equation}
where $\mu(\pi)\vdash k$ is the partition given by the cycle type
of $\pi\in S_{k}$.

Besides the \emph{polynomial }families of representations of $\U(n)$,
as $n$ varies, obtained by fixing $\lambda$ and varying $n$, there
are similar \emph{rational }families that we explain now\emph{. }Let
$\mu\vdash k$ and $\nu\vdash\ell$ be fixed Young diagrams. Then
for every $n\geq\ell(\mu)+\ell(\nu)$, there is an irreducible unitary
representation of $\U(n)$ with signature
\[
(\mu_{1},\mu_{2},\ldots,\mu_{\ell(\mu)},\underbrace{0,\ldots,0}_{n-\ell(\mu)-\ell(\nu)},-\nu_{\ell(\nu)},-\nu_{\ell(\nu)-1},\ldots,-\nu_{1}).
\]
We write $W_{n}^{[\mu,\nu]}$ for this representation of $\U(n)$,
$D_{[\mu,\nu]}(n)$ for its dimension, and $s_{[\mu,\nu]}(g)$ for
the character of this representation at $g\in\U(n)$. The representations
$W_{n}^{[\mu,\nu]}$ directly generalize the case of $W_{n}^{\lambda}$
with $\lambda$ fixed by taking $\mu=\lambda$ and $\nu$ to be the
empty Young diagram. The character $s_{[\lambda,\mu]}(g)$ can be
written in terms of Schur polynomials by a result of Koike.
\begin{thm}[{\cite[eq. (0.2)]{Koike}}]
\label{thm:rational-char-to-schur}For $\mu\vdash k$ and $\nu\vdash\ell$,
$n\geq\ell(\mu)+\ell(\nu)$, $g\in\U(n)$, we have
\[
s_{[\mu,\nu]}(g)=\sum_{\substack{p_{1},p_{2},p_{3}\in\N_{0}\\
p_{1}+p_{2}=k,\quad p_{1}+p_{3}=\ell
}
}\sum_{\substack{\nu_{1}\vdash p_{1},\nu_{2}\vdash p_{2},\nu_{3}\vdash p_{3}\\
\nu_{2}\subset\mu,\nu_{3}\subset\nu
}
}\LR_{\nu_{1},\nu_{2}}^{\mu}\LR_{\check{\nu_{1}},\nu_{3}}^{\nu}s_{\nu_{2}}(g)s_{\nu_{3}}(g^{-1})
\]
where $\check{\nu_{1}}$ denotes the transposed YD of $\nu_{1}$ (obtained
from $\nu_{1}$ by switching rows and columns).
\end{thm}

Inspection of Theorem \ref{thm:rational-char-to-schur} shows that
$D_{[\mu,\nu]}(n)=s_{[\mu,\nu]}(1^{n})$ is a finite linear combination,
with integer coefficents, of 
\begin{equation}
s_{\nu_{2}}(1^{n})s_{\nu_{3}}(1^{n}),\quad\nu_{2}\vdash p_{2}\leq k,\nu_{3}\vdash p_{3}\leq\ell.\label{eq:schur-function-product}
\end{equation}
The linear combination itself does not depend on $n$, and by the
hook-content formula (\ref{eq:hook-content-formula}), each of the
terms in (\ref{eq:schur-function-product}) is a polynomial function
of $n$ with coefficients in $\Q$ for $n\geq\ell(\mu)+\ell(\nu)$.
Moreover, the term in (\ref{eq:schur-function-product}) is $\asymp n^{p_{2}+p_{3}}$
so the contribution to $s_{[\lambda,\mu]}(1^{n})$ of maximal growth
corresponds to $p_{2}=k$ and $p_{3}=\ell$ hence is a unique term
with non-zero (in fact, unity) coefficient. These arguments yield
the following corollary to Theorem \ref{thm:rational-char-to-schur}.
\begin{cor}
\label{cor:dim-of-rational-representation}Let $\mu\vdash k$ and
$\nu\vdash\ell$. For $n\geq\ell(\mu)+\ell(\nu)$, $D_{[\mu,\nu]}(n)$
is given by a polynomial function of $n$ with coefficients in $\Q$
and 
\[
D_{[\mu,\nu]}(n)\asymp n^{k+\ell}
\]
as $n\to\infty$.
\end{cor}

Corollary \ref{cor:dim-of-rational-representation} shows that in
any organization of the irreducible representations of $\U(n)$ or
$\SU(n)$ into families of `small' or `large' dimensional representations,
the representations $W_{n}^{[\mu,\nu]}$ with $\mu,\nu$ fixed should
be considered alongside the representations $W_{n}^{\lambda}$ with
$\lambda$ fixed.

\subsubsection*{Schur-Weyl duality}

We always understand that $\C^{n}$ has the standard Hermitian inner
product and $\U(n)$ acts unitarily on $\C^{n}$ in its defining representation.
This representation coincides with $(\tilde{\rho}_{\lambda},W_{n}^{\lambda})$
where $\lambda$ consists of one box in the formalism of $\S$\ref{subsec:Representation-theory-of-U(n)-SU(n)}.
We write $\{e_{1},\ldots,e_{n}\}$ for the standard basis of $\C^{n}$.The
tensor power $(\C^{n})^{\otimes k}\eqdf\underbrace{\C^{n}\otimes\cdots\otimes\C^{n}}_{k}$
has an induced inner product that makes it into a unitary representation
of $\U(n)$, where $\U(n)$ acts diagonally. The space $(\C^{n})^{\otimes k}$
is also a unitary representation of $S_{k}$ where $S_{k}$ acts by
permuting indices. The actions of $\U(n)$ and $S_{k}$ commute and
hence $(\C^{n})^{\otimes k}$ is a unitary representation of $\U(n)\times S_{k}$.
We write $\pi_{n}^{k}:\U(n)\to\End((\C^{n})^{\otimes k})$ for the
diagonal action of $\U(n)$ and $\rho_{n}^{k}:S_{k}\to\End((\C^{n})^{\otimes k})$
for the action that permutes coordinates. Schur-Weyl duality gives
the following full description of the decomposition of $(\C^{n})^{\otimes k}$
into irreducible representations of $\U(n)\times S_{k}$, which are
a priori tensor products $W\otimes V$ where $W$ (resp. $V$) is
an irreducible representation of $\U(n)$ (resp. $S_{k}$).
\begin{prop}[{Schur-Weyl duality \cite[Chapt. IV]{Weyl}}]
\label{prop:Schur-weyl-duality}There is an isomorphism $\FF$ of
linear representations of $\U(n)\times S_{k}$
\[
\FF:(\C^{n})^{\otimes k}\xrightarrow{\sim}\bigoplus_{\substack{\lambda\vdash k\\
\ell(\lambda)\leq n
}
}W_{n}^{\lambda}\otimes V^{\lambda}.
\]
\end{prop}

\subsubsection*{Branching rules}

For any $r\ge0$ we view $\U(n-r)$ as the subgroup of $\U(n)$ of
elements that fix pointwise the standard basis elements $e_{n-r+1},\ldots,e_{n}$
of $\C^{n}$ in the defining representation of $\U(n)$. We have \cite[Ex. 6.12]{FH}
for $\ell(\lambda)\leq n$ 
\begin{equation}
\Res_{\U(n-1)}^{\U(n)}W_{n}^{\lambda}\cong\bigoplus_{\substack{\mu\,:\,\mu\subset^{1}\lambda\\
\ell(\mu)\leq n-1
}
}W_{n-1}^{\mu}.\label{eq:branching-rule}
\end{equation}
By iterating (\ref{eq:branching-rule}), we obtain a orthogonal direct
sum decomposition 
\begin{align*}
W_{n}^{\lambda} & =\bigoplus_{\substack{\emptyset\subset^{1}\mu_{1}\subset^{1}\ldots\subset^{1}\mu_{n-1}\subset^{1}\lambda\\
\ell(\mu_{i})\leq i
}
}W_{n,(\mu_{1},\ldots,\mu_{n-1})}^{\lambda}
\end{align*}
where
\[
W_{n,(\mu_{1},\ldots,\mu_{n-1})}^{\lambda}\eqdf\bigcap_{i=1}^{n-1}(W_{n}^{\lambda})_{\mu_{i}}
\]
is one-dimensional and $(W_{n}^{\lambda})_{\mu_{i}}$ denotes the
$W_{i}^{\mu_{i}}$-isotypic subspace of $W_{n}^{\lambda}$ for $\U(i)$.
We make the observation that a sequence
\begin{equation}
\emptyset\eqdf\mu_{0}\subset^{1}\mu_{1}\subset^{1}\ldots\subset^{1}\mu_{n-1}\subset^{1}\mu_{n}\eqdf\lambda\label{eq:seuqence-of-young-daigrams}
\end{equation}
gives a semistandard tableau of shape $\lambda$ by filling in the
boxes of $\mu_{i+1}/\mu_{i}$ with the number $i$. This gives a one
to one correspondence $(\mu_{_{1}},\ldots,\mu_{n-1})\mapsto T(\mu_{1},\ldots,\mu_{n-1})$
between such sequences of YDs as in (\ref{eq:seuqence-of-young-daigrams})
and $\SSTab_{[n]}(\lambda)$; for $T=T(\mu_{1},\ldots,\mu_{n-1})\in\SSTab_{[n]}(\lambda)$
we define
\[
W_{n,T}^{\lambda}\eqdf W_{n,(\mu_{1},\ldots,\mu_{n-1})}^{\lambda}
\]
and pick a unit-norm vector $w_{T}$ in each $W_{n,T}^{\lambda}$
(this is unique up to a unit complex number). We do this now once
and for all for all $\lambda$ and $T\in\SSTab_{[n]}(\lambda)$. The
resulting orthonormal basis 
\[
\{\,w_{T}\,:\,T\in\SSTab_{[n]}(\lambda)\,\}
\]
 of $W_{n}^{\lambda}$ is called a \emph{Gelfand-Tsetlin basis. }

Given $m,n\in\N_{0}$ with $m\leq n$, YDs $\lambda,\mu$ with $\mu\subset^{n-m}\lambda$,
$\ell(\mu)\leq m$, $\ell(\lambda)\leq n$, and $R_{1},R_{2}\in\SSTab_{[m+1,n]}(\lambda/\mu)$
we define
\begin{equation}
\EE_{R_{1},R_{2}}^{\lambda,\mu,m}\eqdf\frac{1}{\sqrt{D_{\mu}(m)}}\sum_{T\in\SSTab_{[m]}(\mu)}w_{T\sqcup R_{1}}\otimes\check{w}_{T\sqcup R_{2}}\in\End(W_{n}^{\lambda}).\label{eq:E-def}
\end{equation}

We have use for the following fact, analogous to \cite[Lemma 2.4]{MPasympcover}. 
\begin{lem}
\label{lem:commutant}Let $\lambda$ be a Young diagram with $\ell(\lambda)\leq n$.
Suppose $m\in[n]$. Let $Z(\lambda,m)\subset\End(W_{n}^{\lambda})$
denote the algebra generated by $\U(m)$ acting on $W_{n}^{\lambda}$.
The commutant of $Z(\lambda,m)$ in $\End(W_{n}^{\lambda})$ has an
orthonormal basis given by
\[
\left\{ \EE_{R_{1},R_{2}}^{\lambda,\mu,m}\,:\,\mu\subset^{n-m}\lambda,\,\ell(\mu)\leq m,\,R_{1},R_{2}\in\SSTab_{[m+1,n]}(\lambda/\mu)\right\} .
\]
\end{lem}

\begin{proof}
The proof is essentially the same as that of \cite[Lemma 2.4]{MPasympcover}.
\end{proof}

\subsection{The Weingarten calculus\label{subsec:The-Weingarten-calculus}}

The \emph{Weingarten calculus} is a method of calculating integrals
of the form
\begin{equation}
\int_{u\in\U(n)}u_{i_{1}j_{1}}\cdots u_{i_{k}j_{k}}\bar{u}_{i'_{1}j'_{1}}\cdots\bar{u}_{i'_{k}j'_{k}}d\mu_{\U(n)}^{\mathrm{Haar}}(u)\label{eq:wg-integral}
\end{equation}
 where $d\mu_{\U(n)}^{\mathrm{Haar}}$ is the probability Haar measure
on $\U(n)$. This calculus was developed by Weingarten \cite{weingarten1978asymptotic}
and Xue \cite{xu1997random} before being put on a rigorous footing
by Collins \cite{collins2003moments} and Collins and Śniady \cite{CS}.
One possible development of the Weingarten calculus involves reinterpreting
(\ref{eq:wg-integral}) as a problem in calculating the orthogonal
projection onto the $\U(n)$-invariant vectors in $\End((\C^{n})^{\otimes k})$,
where $u\in\U(n)$ acts on $A\in\End((\C^{n})^{\otimes k})$ by $A\mapsto\pi_{n}^{k}(u)A\pi_{n}^{k}(u^{-1})$,
$\pi_{n}^{k}:\U(n)\to\End((\C^{n})^{\otimes k})$ the diagonal action.
This is actually the point of view that will be relevant to this paper. 

The \emph{Weingarten function, }depending on parameters $n,k$ is
the following element of $\C[S_{k}]$ 
\begin{equation}
\Wg_{n,k}\eqdf\frac{1}{(k!)^{2}}\sum_{\substack{\lambda\vdash k\\
\ell(\lambda)\leq n
}
}\frac{d_{\lambda}^{2}}{D_{\lambda}(n)}\sum_{\sigma\in S_{k}}\chi_{\lambda}(\sigma)\sigma=\frac{1}{k!}\sum_{\substack{\lambda\vdash k\\
\ell(\lambda)\leq n
}
}\frac{d_{\lambda}}{D_{\lambda}(n)}\p_{\lambda}.\label{eq:Wg-def}
\end{equation}
Write $P_{n,k}$ for the orthogonal projection in $\End((\C^{n})^{\otimes k})$
onto the $\U(n)$-invariant vectors. We have the following proposition
of Collins and Śniady \cite[Prop. 2.3]{CS}.
\begin{prop}[Collins-Śniady]
\label{prop:Weingarten}Let $n,k\in\N$. Suppose $A\in\End((\C^{n})^{\otimes k})$.
Then
\[
P_{n,k}[A]=\rho_{n}^{k}\left(\Phi[A]\cdot\Wg_{n,k}\right)
\]
where 
\[
\Phi[A]\eqdf\sum_{\sigma\in S_{k}}\tr(A\rho_{n}^{k}(\sigma^{-1}))\sigma.
\]
\end{prop}

\subsection{Free groups and surface groups\label{subsec:Free-groups-and-surface-groups}}

Recall that $\Sigma_{g}$ is a closed topological surface of genus
$g$ with base point $x_{0}$ and we have identified $\pi_{1}(\Sigma_{g},x_{0})\cong\Gamma_{g}=\langle a_{1},b_{1},\ldots,a_{g},b_{g}|[a_{1},b_{1}]\cdots[a_{g},b_{g}]\rangle$. 

Given $w\in\F_{2g}$ we view $w$ as a combinatorial word in $a_{1},a_{1}^{-1},b_{1},b_{1}^{-1},\ldots,a_{g},a_{g}^{-1},b_{g},b_{g}^{-1}$
by writing it in reduced (shortest) form; i.e., $a_{1}$ does not
follow $a_{1}^{-1}$ etc. 

The commutator subgroup $[\F_{2g},\F_{2g}]$ (resp. $[\Gamma_{g},\Gamma_{g}]$)
is the group generated by all elements of the form $[h_{1},h_{2}]\eqdf h_{1}h_{2}h_{1}^{-1}h_{2}^{-1}$
with $h_{1},h_{2}\in\F_{2g}$ (resp. $\Gamma_{g}$). A simple fact
is that $w\in[\F_{2g},\F_{2g}]$ if and only if $a_{i}$ (resp. $b_{i})$
appears the same number of times as $a_{i}^{-1}$ (resp. $b_{i}^{-1}$)
in the reduced word of $w$, for each $1\leq i\leq g$. The abelianization
of $\Gamma_{g}$ coincides with the first singular homology group
of $\Sigma_{g}$ and is isomorphic to $\Z^{2g}$, generated by the
images of $a_{1},b_{1},\ldots,a_{g},b_{g}$. As such, the map induced
by $q_{w}$ from $\F_{2g}$ to the abelianization of $\Gamma_{g}$
has kernel $[\F_{2g},\F_{2g}]$; hence $w\in[\F_{2g},\F_{2g}]$ if
and only if $q_{g}(w)\in[\Gamma_{g},\Gamma_{g}]$. Therefore if $\gamma\in[\Gamma_{g},\Gamma_{g}]$,
any $w$ representing $\gamma$, or the conjugacy class of $\gamma,$must
be in $[\F_{2g},\F_{2g}]$.

\subsection{The Atiyah-Bott-Goldman measure\label{subsec:The-Atiyah-Bott-Goldman-measure}}

We first follow Goldman \cite{GOLDMANSYMPLECTIC} to define the Atiyah-Bott-Goldman
measure on $\M_{g,n}$. Consider the word map\emph{
\begin{align*}
R_{g} & :\SU(n)^{2g}\to\SU(n)\\
R_{g} & :(A_{1},B_{1},\ldots,A_{g},B_{g})\mapsto[A_{1},B_{1}]\cdots[A_{g},B_{g}].
\end{align*}
}We view $\Hom(\Gamma_{g},\SU(n))$ as $R_{g}^{-1}(\id)\subset\SU(n)^{2g}$
via the embedding (\ref{eq:product-identification-hom-space}). The
diagonal action of $\SU(n)$ on $\SU(n)^{2g}$ by conjugation factors
through an action of $\PSU(n)$. This action preserves $\Hom(\Gamma_{g},\SU(n))$.

Let $\Hom(\Gamma_{g},\SU(n))^{\irr}$ denote the set of $\phi\in\Hom(\Gamma_{g},\SU(n))$
that are \emph{irreducible }as linear representations\emph{. }The
conjugation action of $\PSU(n)$ on $\Hom(\Gamma_{g},\SU(n))^{\irr}$
is clearly free and the action is also proper \cite[pg. 205]{GOLDMANSYMPLECTIC}.
Hence
\[
\M_{g,n}\eqdf\Hom(\Gamma_{g},\SU(n))^{\irr}/\PSU(n)
\]
 is a smooth real manifold of dimension 
\[
\dim\M_{g,n}=(2g-2)\dim\SU(n)
\]
\cite[\S 1.3]{GOLDMANSYMPLECTIC}. The complement of $\M_{g,n}$ in
$\Hom(\Gamma_{g},\SU(n))/\PSU(n)$ is the union of finitely many manifolds
of strictly lower dimension than $\M_{g,n}$.

By \cite[\S 1.4]{GOLDMANSYMPLECTIC} for $[\phi]\in\M_{g,n}$ there
is a identification of the tangent fiber
\begin{equation}
T_{[\phi]}(\M_{g,n})\cong H^{1}(\Gamma_{g},\Ad[\phi])\label{eq:tangent-space-cohomology}
\end{equation}
where $\Ad[\phi]$ is $\su(n)$ with the action of $\Gamma_{g}$ given
by composing $\phi$ with the Adjoint action of $\SU(n)$ on $\su(n)$,
and $H^{1}(\Gamma_{g},\Ad[\phi])$ denotes group cohomology with coefficients
in a module (see e.g. \cite[Chapt. 3 \S 1]{BROWN} for a definition
of group cohomology with coefficients). A key property of group cohomology
of $\Gamma_{g}$ is that it identifies naturally with the singular
cohomology of $\Gamma_{g}$ with coefficients in the local system
corresponding to $\Ad[\phi]$. 

We normalize the Killing form on $\su(n)$ so that the induced Riemannian
volume on $\SU(n)$ has unit total mass, i.e. it gives the probability
Haar measure on $\SU(n)$. Cup product together with the Killing form
on $\su(n)$ and Poincaré duality induces an alternating nondegenerate
bilinear form
\[
H^{1}(\Gamma_{g},\Ad[\phi])\times H^{1}(\Gamma_{g},\Ad[\phi])\to H^{0}(\Gamma_{g},\Ad[\phi])\cong\R
\]
where the last isomorphism used that $\phi$ is irreducible (see \cite[\S 1.4]{GOLDMANSYMPLECTIC}
for more details on this argument). This yields via (\ref{eq:tangent-space-cohomology})
a differential two-form $\omega_{g,n}^{\ABG}$ on $\M_{g,n}$ that
turns out to be closed, hence symplectic (\cite[pg. 208]{GOLDMANSYMPLECTIC},
the closedness of the form was originally proved by Atiyah and Bott
\cite{AB}).

The Atiyah-Bott-Goldman (ABG) measure $\mu_{g,n}^{\ABG}$ is the probability
measure on $\M_{g,n}$ induced by the symplectic volume form:
\[
d\mathrm{Vol}_{\M_{g,n}}\eqdf\frac{\wedge^{\frac{1}{2}\dim\M_{g,n}}(\omega_{g,n}^{\ABG})}{(\dim\M_{g,n})!}.
\]
If $f$ is a continuous function on $\M_{g,n}$ then the expected
value of $f$ with respect to $\mu_{g,n}^{\ABG}$ is given by
\begin{equation}
\E_{g,n}[f]\eqdf\int fd\mu_{g,n}^{\ABG}\eqdf\frac{\int_{\M_{g,n}}fd\mathrm{Vol}_{\M_{g,n}}}{\int_{\M_{g,n}}d\mathrm{Vol}_{\M_{g,n}}}.\label{eq:ABGmeasure-def}
\end{equation}

Having defined the ABG measure we now explain how to deduce Corollary
\ref{cor:Fourier-expansion-of-main-integral} from Sengupta's work
\cite{Sengupta}. For mostly technical reasons this involves the introduction
of the \emph{heat kernel}\footnote{The heat kernel will only be involved for a short time before being
dispensed of.} on $\SU(n)$. The heat kernel is a one parameter family of functions
$Q_{t}:\SU(n)\to\R$ for $t\in(0,\infty)$ that, as a function on
$(0,\infty)\times\SU(n)$, is a fundamental solution to the heat equation
on $\SU(n)$ (cf. \cite[\S\S 1.2]{Sengupta} for the precise definition).
What is important here is that the heat kernel has the expansion
\begin{equation}
Q_{t}(h)=\sum_{\substack{\text{\ensuremath{(\rho,W)\in\widehat{\SU(n)}}}}
}e^{-\frac{C(\rho)t}{2}}\dim W\tr(\rho(h))\label{eq:heat-kernel-expansion}
\end{equation}
where $C(\rho)\geq0$ is the eigenvalue of the Casimir operator of
$\SU(n)$ on the representation $(\rho,W)$. This expansion is uniformly
convergent on $\SU(n)$ for each fixed $t>0$ \cite[Thm. 7.2.6]{Buser}.
The following theorem was proved by Sengupta \cite[Thm. 1]{Sengupta}.
\begin{thm}[Sengupta]
\label{thm:Sengupta}Let $g\geq2$ and suppose $f$ is a continuous
$\SU(n)$-conjugation-invariant function on $\SU(n)^{2g}$, and $\tilde{f}$
the function induced by $f$ on $\M_{g,n}$. Then 
\[
\lim_{t\to0_{+}}\int_{\SU(n)^{2g}}f(x)\overline{Q_{t}(R_{g}(x))}d\mu_{\SU(n)^{2g}}^{\mathrm{Haar}}(x)=\frac{1}{n}\int_{\M_{g,n}}\tilde{f}d\mathrm{Vol}_{\M_{g,n}}.
\]
The notation $\lim_{t\to0_{+}}$ means the limit is taken along positive
values of $t$.
\end{thm}

\begin{proof}[Deduction of Corollary \ref{cor:Fourier-expansion-of-main-integral}
from Theorem \ref{thm:Sengupta}]

Under the same assumptions as Theorem \ref{thm:Sengupta}, using (\ref{eq:heat-kernel-expansion})
one may write for fixed $t>0$
\begin{align}
 & \int_{\SU(n)^{2g}}f(x)\overline{Q_{t}(R_{g}(x))}d\mu_{\SU(n)^{2g}}^{\mathrm{Haar}}(x)\nonumber \\
= & \int_{\SU(n)^{2g}}f(x)\left(\sum_{\substack{\text{\ensuremath{(\rho,W)\in\widehat{\SU(n)}}}}
}e^{-\frac{C(\rho)t}{2}}\dim W\overline{\tr(\rho(R_{g}(x)))}\right)d\mu_{\SU(n)^{2g}}^{\mathrm{Haar}}\nonumber \\
 & =\sum_{\substack{\text{\ensuremath{(\rho,W)\in\widehat{\SU(n)}}}}
}e^{-\frac{C_{\lambda}(\rho)t}{2}}\dim W\int f(x)\overline{\tr(\rho(R_{g}(x)))}d\mu_{\SU(n)^{2g}}^{\mathrm{Haar}}(x)\label{eq:sengupta-temp}
\end{align}
where the interchange of sum and integral is valid by uniform convergence
of the heat kernel expansion. We are given $\gamma\in\Gamma_{g}$
and $w\in\F_{2g}$ representing the conjugacy class of $\gamma$.
We take $f=\tr\circ w$ ($w$ denoting the word map of $w$) so that
$\tilde{f}=\tr_{\gamma}$. Recall the notation $\I(w,\rho)$ from
\ref{eq:I-def}, in this case, Theorem \ref{thm:Sengupta} together
with (\ref{eq:sengupta-temp}) gives
\begin{align}
\frac{1}{n}\int_{\M_{g,n}}\tr_{\gamma}\,d\mathrm{Vol}_{\M_{g,n}} & =\lim_{t\to0_{+}}\sum_{\substack{\text{\ensuremath{(\rho,W)\in\widehat{\SU(n)}}}}
}e^{-\frac{C(\rho)t}{2}}\dim W\,\I(w,\rho).\label{eq:pre-dc-1}
\end{align}
Since we assume the hypothesis that $\sum_{\substack{\text{\ensuremath{(\rho,W)\in\widehat{\SU(n)}}}}
}\dim W\,\I(w,\rho)$ is absolutely convergent, using dominated convergence in (\ref{eq:pre-dc-1})
proves
\[
\int_{\M_{g,n}}\tr_{\gamma}\,d\mathrm{Vol}_{\M_{g,n}}=n\sum_{\substack{\text{\ensuremath{(\rho,W)\in\widehat{\SU(n)}}}}
}\dim W\,\I(w,\rho).
\]
 Combining this with Witten's formula (Theorem \ref{thm:Witten})
we obtain Corollary \ref{cor:Fourier-expansion-of-main-integral}.
\end{proof}

\section{Organization of representations}

\subsection{Models of representations\label{subsec:Models-of-representations}}

In light of Corollary \ref{cor:Fourier-expansion-of-main-integral},
we must evaluate $\I(w,\rho)$ for $(\rho,W)\in\widehat{\SU(n)}$.
We use different models of $(\rho,W)$ throughout the paper: recalling
the definition of $\Omega(B;n)$ from (\ref{eq:omega-def}), viewing
$B$ as fixed and assuming $n$ is large enough depending on $B$
\begin{itemize}
\item If $(\rho,W)\in\Omega(B;n)$ we identify $(\rho,W)\cong(\rho_{n}^{[\mu,\nu]},W_{n}^{[\mu,\nu]})$
for some $\mu,\nu$ with $|\mu|,|\nu|\leq B^{3}$ uniquely determined
by $(\rho,W)$. Hence 
\[
\I(w,\rho)=\I(w,[\mu,\nu])\eqdf\int\tr(w(x))\overline{s_{[\mu,\nu]}(R_{g}(x))}d\mu_{\SU(n)^{2g}}^{\mathrm{Haar}}(x).
\]
\item If $(\rho,W)\in\widehat{\SU(n)}\backslash\Omega(B;n)$ then we identify
$(\rho,W)\cong(\rho_{n}^{\lambda},W_{n}^{\lambda})$ for some $\lambda$
with $\ell(\lambda)\leq n-1$ uniquely determined by $(\rho,W)$.
Hence 
\begin{equation}
\I(w,\rho)=\I(w,\lambda)\eqdf\int\tr(w(x))\overline{s_{\lambda}(R_{g}(x))}d\mu_{\SU(n)^{2g}}^{\mathrm{Haar}}(x).\label{eq:I-lambda-def}
\end{equation}
\end{itemize}

\subsection{Integrating over $\protect\SU(n)^{2g}$ vs $\protect\U(n)^{2g}$\label{subsec:Integrating-over-SU(n)2g}}

At various points it is convenient to execute the integral defining
$\mathcal{I}(w,\lambda)$ or $\mathcal{I}(w,[\mu,\nu])$ as an integral
over $\U(n)^{2g}$ rather than $\SU(n)^{2g}$. This is indeed possible
in all cases that the large-$n$ behavior of $\E_{g,n}[\tr_{\gamma}]$
is interesting. Recall that $[\F_{2g},\F_{2g}]$ is the commutator
subgroup of $\F_{2g}$.
\begin{prop}
~\label{prop:using-U(n)-for-fourier-coefficients}
\begin{enumerate}
\item \label{enu:Fourier-using-U(n)}If $w\in[\F_{2g},\F_{2g}]$ then 
\begin{align*}
\mathcal{I}(w,\lambda) & =\int_{x\in\U(n)^{2g}}\tr(w(x))\overline{s_{\lambda}(R_{g}(x))}d\mu_{\U(n)^{2g}}^{\mathrm{Haar}}(x),\\
\I(w,[\mu,\nu]) & =\int_{x\in\U(n)^{2g}}\tr(w(x))\overline{s_{[\mu,\nu]}(R_{g}(x))}d\mu_{\U(n)^{2g}}^{\mathrm{Haar}}(x),
\end{align*}
In other words, the integrals can be computed using $\U(n)$ instead
of $\SU(n)$.
\item \label{enu:vanishing-expectation-commutator--subgroup}On the other
hand, if $w\in\F_{2g}$ and $w\notin[\F_{2g},\F_{2g}]$, then there
is $n_{0}=n_{0}(w)$ such that for $n\geq n_{0}$, for any $(\rho,W)\in\widehat{\SU(n)}$,
\begin{align*}
\mathcal{I}(w,\rho) & =0
\end{align*}
\end{enumerate}
\end{prop}

\begin{proof}
By uniqueness of the Haar measure on $\U(n)^{2g}$, we have 
\begin{equation}
\mu_{\U(n)^{2g}}^{\mathrm{Haar}}=\mu_{Z(\U(n)^{2g})}^{\mathrm{Haar}}*\mu_{\SU(n)^{2g}}^{\mathrm{Haar}}\label{eq:deconvolution}
\end{equation}
 where $*$ denotes convolution of measures. Therefore
\begin{align*}
 & \int_{x\in\U(n)^{2g}}\tr(w(x))\overline{s_{\lambda}(R_{g}(x))}d\mu_{\U(n)^{2g}}^{\mathrm{Haar}}(x)\\
= & \int_{z\in Z(\U(n)^{2g})}\int_{x\in\SU(n)^{2g}}\tr(w(zx))\overline{s_{\lambda}(R_{g}(zx))}d\mu_{\SU(n)^{2g}}^{\mathrm{Haar}}(x)d\mu_{Z(\U(n)^{2g})}^{\mathrm{Haar}}(z)\\
= & \int_{z\in Z(\U(n)^{2g})}\int_{x\in\SU(n)^{2g}}\tr(w(zx))\overline{s_{\lambda}(R_{g}(x))}d\mu_{\SU(n)^{2g}}^{\mathrm{Haar}}(x)d\mu_{Z(\U(n)^{2g})}^{\mathrm{Haar}}(z)\\
= & \int_{x\in\SU(n)^{2g}}\left(\int_{z\in Z(\U(n)^{2g})}\tr(w(zx))d\mu_{Z(\U(n)^{2g})}^{\mathrm{Haar}}(z)\right)\overline{s_{\lambda}(R_{g}(x))}d\mu_{\SU(n)^{2g}}^{\mathrm{Haar}}(x).
\end{align*}
The first equality used (\ref{eq:deconvolution}), the second used
$R_{g}(zx)=R_{g}(x)$ for all $x\in\U(n)^{2g}$ and $z\in Z(\U(n)^{2g})$,
and the last used Fubini's theorem. If $w\in[\F_{2g},\F_{2g}]$ then
by the discussion in $\S\S\ref{subsec:Free-groups-and-surface-groups}$
the letters of $w$ are `balanced' and the inner integrand is $\tr(w(zx))=\tr(w(x))$.
Thus we obtain
\[
\mathcal{I}(w,\lambda)=\int_{x\in\U(n)^{2g}}\tr(w(x))\overline{s_{\lambda}(R_{g}(x))}d\mu_{\U(n)^{2g}}^{\mathrm{Haar}}(x).
\]
In other words, the Fourier coefficient can be computed using $\U(n)$
instead of $\SU(n)$. The proof for $\I(w,[\mu,\nu])$ is exactly
the same. \emph{This proves Part \ref{enu:Fourier-using-U(n)}.}

For the second part, suppose $w$ is not in $[\F_{2g},\F_{2g}]$.
Then the letters of $w$ are not balanced and for $z=\left(e\left(\frac{k_{1}}{n}\right)\id,\ldots,e\left(\frac{k_{2g}}{n}\right)\id\right)\in Z(\SU(n)^{2g})$,
we have 
\[
\tr(w(zx))=e\left(m_{1}\frac{k_{1}}{n}+\cdots+m_{2g}\frac{k_{2g}}{n}\right)\tr(w(x))
\]
where all $m_{i}\in\Z$ and \emph{at least one $m_{i}\neq0$. }On
the other hand $R_{g}(zx)=R_{g}(x)$ as before. Suppose without loss
of generality in the following that $m_{1}\neq0$. Let $z_{0}=\left(e\left(\frac{1}{n}\right)\id,\ldots,\id\right)$.
We have by left-invariance of Haar measure
\begin{align*}
\mathcal{I}(w,\rho) & =\int_{x\in\SU(n)^{2g}}\tr(w(x))\overline{\tr(\rho(R_{g}(x)))}d\mu_{\SU(n)^{2g}}^{\mathrm{Haar}}(x)\\
 & =\int_{x\in\SU(n)^{2g}}\tr(w(z_{0}x))\overline{\tr(\rho(R_{g}(z_{0}x)))}d\mu_{\SU(n)^{2g}}^{\mathrm{Haar}}(x)\\
 & =e\left(\frac{m_{1}}{n}\right)\int_{x\in\SU(n)^{2g}}\tr(w(x))\overline{\tr(\rho(R_{g}(z_{0}x)))}d\mu_{\SU(n)^{2g}}^{\mathrm{Haar}}(x)\\
 & =e\left(\frac{m_{1}}{n}\right)\int_{x\in\SU(n)^{2g}}\tr(w(x))\overline{\tr(\rho(R_{g}(x)))}d\mu_{\SU(n)^{2g}}^{\mathrm{Haar}}(x)=e\left(\frac{m_{1}}{n}\right)\mathcal{I}(w,\rho).
\end{align*}
Hence for $n>m_{1}(w)$ we obtain $\mathcal{I}(w,\rho)=0$ for any
$(\rho,W)\in\widehat{\SU(n)}$. 
\end{proof}
This means the main theorem is proved when $\gamma\notin[\Gamma_{g},\Gamma_{g}]$.
\begin{proof}[Proof of Proposition \ref{prop:commutator-prop}]
Assume $\gamma\notin[\Gamma_{g},\Gamma_{g}]$. Pick $w$ representing
the conjugacy class of $\gamma$ as above. Since $\gamma\notin[\Gamma_{g},\Gamma_{g}]$
this implies $w\notin[\F_{2g},\F_{2g}]$ by the discussion in $\S\S$\ref{subsec:Free-groups-and-surface-groups}.
Hence Corollary \ref{cor:Fourier-expansion-of-main-integral} together
with Proposition \ref{prop:using-U(n)-for-fourier-coefficients} Part
\ref{enu:vanishing-expectation-commutator--subgroup} shows
\[
\E_{g,n}[\tr_{\gamma}]=0
\]
when $n\geq n_{0}(w)$ and the series defining $\zeta(2g-2;n)$ converges,
i.e. $n\geq2$. 
\end{proof}
\emph{As Theorem \ref{thm:Main-theorem} has now been proved for $\gamma\notin[\Gamma_{g},\Gamma_{g}]$,
in the rest of the paper, we may assume $\gamma\in[\Gamma_{g},\Gamma_{g}]$
and hence any $w$ representing the conjugacy class of $\gamma$ is
in $[\F_{2g},\F_{2g}]$.}

\subsection{Rationality of the contribution from $\Omega(B;n)$\label{subsec:Rationality-of-the-contirbution}}

Here we prove the following theorem. Let $\Q(t)$ denote the ring
of rational functions in an indeterminate $t$ with coefficients in
$\Q$. For $f\in\Q(t)$ and $t_{0}\in\Q$ we write $f(t_{0})$ for
the evaluation of the rational function at $t_{0}$, provided that
$t_{0}$ is not a pole of $f$. Let $w\in[\F_{2g},\F_{2g}]$ represent
the conjugacy class of $\gamma\in[\Gamma_{g},\Gamma_{g}]$. 
\begin{thm}
\label{thm:rational-small-dim}There is a rational function $Q_{B,w}\in\Q(t)$
such that for $n\geq|w|+2B^{3}$ 
\[
\sum_{\substack{(\rho,W)\in\Omega(B;n)}
}(\dim W)\mathcal{I}(w,\rho)=Q_{B,w}(n).
\]
\end{thm}

\begin{proof}
Fix $B\geq0$. Following the discussion in $\S\S$\ref{subsec:Models-of-representations}
we have
\[
\sum_{\substack{(\rho,W)\in\Omega(B;n)}
}(\dim W)\mathcal{I}(w,\rho)=\sum_{\mu,\nu\,:\,\ell(\mu),\ell(\nu)\leq B,\mu_{1},\nu_{1}\leq B^{2}}D_{[\mu,\nu]}(n)\I(w,[\mu,\nu]).
\]
Since the index set of the sum is finite and does not depend on $n$
it now suffices to prove that each $D_{[\mu,\nu]}(n)\I(w,[\mu,\nu])$
agrees with a rational function of $n$ for fixed $\mu,\nu$ as above.
By Corollary \ref{cor:dim-of-rational-representation}, each $D_{[\mu,\nu]}(n)$
agrees with a polynomial function of $n$ with coefficients in $\Q$
when $n\geq2B.$ Therefore the proof of the theorem is reduced to
the following:

\emph{Claim: For each $(\mu,\nu)\in\Omega(B)$, when $n\geq|w|+2B^{3}$,
$\I(w,[\mu,\nu])$ agrees with a rational function of $n$ with coefficients
in $\Q$.}

\emph{Proof of claim. }We first use Proposition \ref{prop:using-U(n)-for-fourier-coefficients}
to write
\[
\I(w,[\mu,\nu])=\int_{x\in\U(n)^{2g}}\tr(w(x))\overline{s_{[\mu,\nu]}(R_{g}(x))}d\mu_{\U(n)^{2g}}^{\mathrm{Haar}}(x).
\]
For $x\in\U(n)^{2g}$ we use Theorem \ref{thm:rational-char-to-schur}
followed by the formula (\ref{eq:base-change-formula}) to obtain
with $k=|\mu|,\ell=|\nu|$
\[
s_{[\mu,\nu]}(R_{g}(x))=\sum_{\text{\ensuremath{\substack{\text{YDs }\ensuremath{\mu^{1},\mu^{2}}\\
|\mu^{1}|\leq\ell,|\mu^{2}|\leq k
}
}}}\alpha_{\mu^{1},\mu^{2}}^{\mu,\nu}p_{\mu^{1}}(R_{g}(x))p_{\mu^{2}}(R_{g}(x)^{-1})
\]
 where the $\alpha_{\mu^{1},\mu^{2}}^{\mu,\nu}\in\Q$. Therefore
\[
\I(w,[\mu,\nu])=\sum_{\text{\ensuremath{\substack{\text{YDs}\ensuremath{,\mu^{2}}\\
|\mu^{1}|\leq\ell,|\mu^{2}|\leq k
}
}}}\alpha_{\mu^{1},\mu^{2}}^{\mu,\nu}\int_{x\in\U(n)^{2g}}\tr(w(x))p_{\mu^{1}}(R_{g}(x))p_{\mu^{2}}(R_{g}(x)^{-1})d\mu_{\U(n)^{2g}}^{\mathrm{Haar}}(x).
\]
For every fixed $\mu^{1},\mu^{2}$ appearing in the finite sum above,
\begin{equation}
\int_{x\in\U(n)^{2g}}\tr(w(x))p_{\mu^{1}}(R_{g}(x))p_{\mu^{2}}(R_{g}(x)^{-1})d\mu_{\U(n)^{2g}}^{\mathrm{Haar}}(x)\label{eq:power-sum-integral-naive}
\end{equation}
agrees with a rational function of $n$ by\footnote{This is a straightforward application of the Weingarten calculus.}
\cite[Prop. 1.1]{MPunitary}, for $n\geq|w|+2B^{3}\geq|w|+k+\ell$.
This proves the claim and hence the theorem.
\end{proof}

\section{The contribution from a single family of representations\label{sec:The-contribution-from-single-large-dim}}

\subsection{Statement of main sectional result and setup}

The main theorem of this $\S$\ref{sec:The-contribution-from-single-large-dim}
is the following key estimate that will be used for $(\rho_{n}^{\lambda},W_{n}^{\lambda})\in\widehat{\SU(n)}\backslash\Omega(B;n)$.
\begin{thm}
\label{thm:single-dimension-bound}For $w\in[\F_{2g},\F_{2g}]$ 
\begin{align*}
|\I(w,\lambda)|\leq & \sum_{[I]\in S_{_{n}}\backslash[n]^{|w|}}(n)_{\D(I)}\frac{1}{D_{\lambda}(n)^{2g}}\sum_{\substack{\mu\subset^{\D(I)}\lambda\\
\ell(\mu)\leq n-\D(I)
}
}\\
 & D_{\mu}(n-\D(I))\left(|\lambda/\mu|+|w|\right)^{4g|w|}|\SSTab_{[n-\D(I)+1,n]}(\lambda/\mu)|^{4g}
\end{align*}
where for $I=(i_{1},i_{2},\ldots,i_{|w|})\in[n]^{|w|}$, $\D(I)$
denotes the number of distinct entries of $I$.
\end{thm}

In the rest of this $\S$\ref{sec:The-contribution-from-single-large-dim},
we assume $g=2$ for simplicity of exposition. The proofs extend in
a straightforward way to $g\geq3$. We write $\{a,b,c,d\}$ for the
generators of $\F_{4}$ and $R=[a,b][c,d]$. We write $w$ in reduced
form:
\begin{equation}
w=f_{1}^{\varepsilon_{1}}f_{2}^{\varepsilon_{2}}\ldots f_{|w|}^{\varepsilon_{|w|}},\quad\varepsilon_{u}\in\{\pm1\},\,f_{u}\in\{a,b,c,d\},\label{eq:combinatorial-word}
\end{equation}
The expression (\ref{eq:combinatorial-word}) implies that for $h\eqdf(h_{a},h_{b},h_{c},h_{d})\in\U(n)^{4}$
\[
\tr(w(h))=\sum_{i_{j}\in[n]}(h_{f_{1}}^{\epsilon_{1}})_{i_{1}i_{2}}(h_{f_{2}}^{\epsilon_{2}})_{i_{2}i_{3}}\cdots(h_{f_{|w|}}^{\epsilon_{|w|}})_{i_{|w|}i_{1}}.
\]
Define for $I=(i_{1},i_{2},\ldots,i_{|w|})\in[n]^{|w|}$
\[
w_{I}(h)\eqdf(h_{f_{1}}^{\epsilon_{1}})_{i_{1}i_{2}}(h_{f_{2}}^{\epsilon_{2}})_{i_{2}i_{3}}\cdots(h_{f_{|w|}}^{\epsilon_{|w|}})_{i_{|w|}i_{1}}.
\]
After interchanging summation and integration in (\ref{eq:I-lambda-def})
we obtain
\begin{align}
\mathcal{I}(w,\lambda) & =\sum_{I\in[n]^{|w|}}\mathcal{I}^{*}(w_{I},\lambda).\label{eq:index-splitting}
\end{align}
where 
\[
\I^{*}(w_{I},\lambda)\eqdf\int_{h\in\U(n)^{4}}w_{I}(h)\overline{s_{\lambda}(R(h))}d\mu_{\U(n)^{4}}^{\mathrm{Haar}}
\]
Since we have an inclusion $S_{n}\subset\U(n)$ via 0-1 matrices,
for $\sigma\in S_{n}$ we can change variables
\[
(h_{a},h_{b},h_{c},h_{d})\mapsto(h'_{a},h'_{b},h'_{c},h'_{d})\eqdf(\sigma h_{a}\sigma^{-1},\sigma h_{b}\sigma^{-1},\sigma h_{c}\sigma^{-1},\sigma h_{d}\sigma^{-1}).
\]
The measure $d\mu_{\U(n)^{4}}^{\mathrm{Haar}}$ is invariant by this
change of variables and 
\begin{align*}
\overline{s_{\lambda}(R(h'_{a},h'_{b},h'_{c},h'_{d}))} & =\overline{s_{\lambda}(\sigma R(h{}_{a},h{}_{b},h{}_{c},h{}_{d})\sigma^{-1})}\\
 & =\overline{s_{\lambda}(R(h_{a},h_{b},h_{c},h_{d}))}.
\end{align*}
For $I=(i_{1},i_{2},\ldots,i_{|w|})$ and $\sigma\in S_{n}$ we define
$\sigma(I)=(\sigma(i_{1}),\sigma(i_{2}),\ldots,\sigma(i_{|w|}))$
and we have
\[
w_{I}(h')=w_{\sigma(I)}(h)
\]
so we obtain in total
\begin{align*}
\mathcal{I}^{*}(w_{I},\lambda)= & \int_{h\in\U(n)^{4}}w_{_{I}}(h)\overline{s_{\lambda}(R(h))}d\mu_{\U(n)^{4}}^{\mathrm{Haar}}\\
= & \int_{h\in\U(n)^{4}}w_{_{\sigma(I)}}(h)\overline{s_{\lambda}(R(h))}d\mu_{\U(n)^{4}}^{\mathrm{Haar}}\\
= & \mathcal{I}^{*}(w_{\sigma(I)},\lambda).
\end{align*}
We can therefore rewrite (\ref{eq:index-splitting}) as
\begin{align}
\mathcal{I}(w,\lambda) & =\sum_{[I]\in S_{_{n}}\backslash[n]^{|w|}}|S_{n}.I|\mathcal{I}^{*}(w_{I},\lambda)\nonumber \\
 & =\sum_{[I]\in S_{_{n}}\backslash[n]^{|w|}}(n)_{\D(I)}\mathcal{I}^{*}(w_{I},\lambda)\label{eq:index-splitting-two}
\end{align}
where $\D(I)$ is the number of distinct entries in $I$. Most of
the rest of the section will be devoted to estimating $\mathcal{I}^{*}(w_{I},\lambda)$;
the point of the previous calculations is that we can assume $I\in[n-\D(I)+1,n]^{|w|}$
and this will be exploited in $\S\S\ref{subsec:First-integrating-over}$.

\subsection{First integrating over a large subgroup\label{subsec:First-integrating-over}}

We keep the assumptions of the previous section and also assume 
\[
I\in[n-\D(I)+1,n]^{|w|}.
\]
We fix $I$ and hence write $\D=\D(I)$. This assumption means that
the function $w_{I}:\U(n)^{4}\to\C$ is binvariant for $\U(m)^{4}$
where
\[
m\eqdf n-\D.
\]
To simplify notation, \emph{all integrals over groups are done with
respect to the probability Haar measure} and this will be denoted
by $dg$ where $g$ is the group element. 

Our goal is to calculate $\mathcal{I}^{*}(w_{I},\lambda)$.The bi-invariance
of $d\mu_{\U(n)^{4}}^{\mathrm{Haar}}$ and the $\U(m)^{4}$-bi-invariance
of $w_{I}$ means we can write
\begin{align}
\mathcal{I}^{*}(w_{I},\lambda) & =\int_{h\in\U(n)^{4}}w_{_{I}}(h)\overline{s_{\lambda}(R(h))}dh\nonumber \\
 & =\int_{h_{1},h_{2}\in\U(m)^{4}}\left(\int_{h\in\U(n)^{4}}w_{_{I}}(h_{1}hh_{2})\overline{s_{\lambda}(R(h_{1}hh_{2}))}dh\right)dh_{1}dh_{2}\nonumber \\
 & =\int_{h_{1},h_{2}\in\U(m)^{4}}\int_{h\in\U(n)^{4}}w_{_{I}}(h)\overline{s_{\lambda}(R(h_{1}hh_{2}))}dh\,dh_{1}\,dh_{2}\nonumber \\
 & =\int_{h\in\U(n)^{4}}w_{I}(h)\left(\int_{h_{1},h_{2}\in\U(m)^{4}}\overline{s_{\lambda}(R(h_{1}hh_{2}))}dh_{1}dh_{2}\right)dh.\label{eq:Istar-first-form}
\end{align}
The first cancellation we will obtain comes from the integral
\begin{equation}
\int_{h_{1},h_{2}\in\U(m)^{4}}\overline{s_{\lambda}(R(h_{1}hh_{2}))}dh_{1}dh_{2},\quad h\in\U(n)^{4}.\label{eq:s_lambda_int}
\end{equation}
Our approach to this integral follows the same lines as \cite[\S 4.4]{MPasympcover}.
We consider the vector space
\begin{equation}
\W_{R}^{\lambda}\eqdf W_{a}^{\lambda}\otimes\check{W_{a}^{\lambda}}\otimes W_{b}^{\lambda}\otimes\check{W_{b}^{\lambda}}\otimes W_{c}^{\lambda}\otimes\check{W_{c}^{\lambda}}\otimes W_{d}^{\lambda}\otimes\check{W_{d}^{\lambda}}\label{eq:def-W-lambda}
\end{equation}
which is a unitary representation of $\U(n)^{4}$, the subscripts
indicating which elements of $(h_{a},h_{b},h_{c},h_{d})$ act on which
factor, so each $h_{f}$ acts diagonally on two factors.

Let $B_{\lambda}\in\End(\W_{R}^{\lambda})$ be defined via matrix
coefficients by the formula
\begin{align}
 & \left\langle B_{\lambda}\left(v_{1}\otimes\check{v}_{2}\otimes v_{3}\otimes\check{v_{4}}\otimes v_{5}\otimes\check{v}_{6}\otimes v_{7}\otimes\check{v_{8}}\right),w_{1}\otimes\check{w}_{2}\otimes w_{3}\otimes\check{w_{4}}\otimes w_{5}\otimes\check{w}_{6}\otimes w_{7}\otimes\check{w_{8}}\right\rangle \eqdf\nonumber \\
 & ~~~~~~~~~\left\langle v_{5},w_{7}\right\rangle \left\langle w_{5},w_{8}\right\rangle \left\langle w_{6},v_{8}\right\rangle \left\langle v_{3},v_{6}\right\rangle \left\langle v_{1},w_{3}\right\rangle \left\langle w_{1},w_{4}\right\rangle \left\langle w_{2},v_{4}\right\rangle \left\langle v_{7},v_{2}\right\rangle .\label{eq:B-lambda-def}
\end{align}

We have the following lemma analogous to \cite[Lemma 4.7]{MPasympcover}.
\begin{lem}
\label{lem:B-lambda-property}For any $h=(h_{a},h_{b},h_{c},h_{d})\in\U(n)^{4}$,
we have 
\[
\mathrm{tr}_{\W_{R}^{\lambda}}\left(B_{\lambda}\circ(h_{a},h_{b},h_{c},h_{d})\right)=\overline{s_{\lambda}\left([h_{a},h_{b}][h_{c},h_{d}]\right)}.
\]
\end{lem}

The purpose of Lemma \ref{lem:B-lambda-property} is that it turns
the integral in (\ref{eq:s_lambda_int}) into a \emph{projection operator.
}Indeed, let $Q$ denote the orthogonal projection in $W_{n}^{\lambda}$
onto the $\U(m)^{4}$-invariant vectors.
\begin{lem}
\label{lem:theta-as-trace-of-B}We have 
\begin{align*}
\int_{h_{1},h_{2}\in\U(m)^{4}}\overline{s_{\lambda}(R(h_{1}hh_{2}))}dh_{1}dh_{2} & =\mathrm{tr}_{\W_{R}^{\lambda}}\left(hQB_{\lambda}Q\right)
\end{align*}
.
\end{lem}

\begin{proof}
By Lemma \ref{lem:B-lambda-property}
\begin{align*}
\int_{h_{1},h_{2}\in\U(m)^{4}}\overline{s_{\lambda}(R(h_{1}hh_{2}))}dh_{1}dh_{2} & =\int_{h_{1},h_{2}\in\U(m)^{4}}\mathrm{tr}_{\W_{R}^{\lambda}}\left(B_{\lambda}\circ h_{1}hh_{2}\right)dh_{1}dh_{2}\\
 & =\mathrm{tr}_{\W_{R}^{\lambda}}\left(B_{\lambda}\circ\left(\int_{h_{1}\in\U(m)^{4}}h_{1}\right)h\left(\int_{h_{1}\in\U(m)^{4}}h_{2}\right)\right)\\
 & =\mathrm{tr}_{\W_{R}^{\lambda}}\left(B_{\lambda}QhQ\right)=\mathrm{tr}_{\W_{R}^{\lambda}}\left(hQB_{\lambda}Q\right).
\end{align*}
\end{proof}
Recalling the definition of $\EE_{R_{1},R_{2}}^{\lambda,\mu,m}$ from
(\ref{eq:E-def}), we are able to calculate $\mathrm{tr}_{\W_{R}^{\lambda}}\left(hQB_{\lambda}Q\right)$
as follows. 
\begin{prop}
\label{prop:trhQBQ}We have
\begin{align*}
\mathrm{tr}_{\W_{R}^{\lambda}}\left(hQB_{\lambda}Q\right) & =\sum_{\substack{\mu\subset^{\D}\lambda\\
\ell(\mu)\leq m
}
}\frac{1}{D_{\mu}(m)^{3}}\sum_{S_{1},S_{2},S_{3},S_{4},T_{2},T_{4},s_{1},s_{2}\in\SSTab_{[m+1,n]}(\lambda/\mu)}\\
 & \langle h[\EE_{s_{1},T_{2}}^{\lambda,\mu,m}\otimes\EE_{S_{1},s_{1}}^{\lambda,\mu,m}\otimes\EE_{s_{2},T_{4}}^{\lambda,\mu,m}\otimes\EE_{S_{3},s_{2}}^{\lambda,\mu,m}],\EE_{S_{1},S_{4}}^{\lambda,\mu,m}\otimes\EE_{S_{2},T_{2}}^{\lambda,\mu,m}\otimes\EE_{S_{3},S_{2}}^{\lambda,\mu,m}\otimes\EE_{S_{4},T_{4}}^{\lambda,\mu,m}\rangle.
\end{align*}
\end{prop}

\begin{proof}
An orthonormal basis for the $\U(m)^{4}$-invariant vectors in $\W_{R}^{\lambda}$
is given by
\[
\EE_{S_{1},T_{1}}^{\lambda,\mu_{1},m}\otimes\EE_{S_{2},T_{2}}^{\lambda,\mu_{2},m}\otimes\EE_{S_{3},T_{3}}^{\lambda,\mu_{3},m}\otimes\EE_{S_{4},T_{4}}^{\lambda,\mu_{4},m}
\]
where the $\mu_{i}$ range over all $\mu_{i}\subset^{\D}\lambda$
with $\ell(\mu_{i})\leq m$ and each $S_{i},T_{i}\in\SSTab_{[m+1,n]}(\lambda/\mu_{i}).$
This can be extended to a full orthonormal basis of $\W_{R}^{\lambda}$
and hence 
\begin{align}
\mathrm{tr}_{\W_{R}^{\lambda}}\left(hQB_{\lambda}Q\right)= & \sum_{\substack{\mu_{i}\subset^{\D}\lambda\\
\ell(\mu_{i})\leq m
}
}\sum_{S_{i},T_{i}\in\SSTab_{[m+1,n]}(\lambda/\mu_{i})}\langle hQB_{\lambda}\EE_{S_{1},T_{1}}^{\lambda,\mu_{1},m}\otimes\EE_{S_{2},T_{2}}^{\lambda,\mu_{2},m}\otimes\EE_{S_{3},T_{3}}^{\lambda,\mu_{3},m}\otimes\EE_{S_{4},T_{4}}^{\lambda,\mu_{4},m},\nonumber \\
 & \EE_{S_{1},T_{1}}^{\lambda,\mu_{1},m}\otimes\EE_{S_{2},T_{2}}^{\lambda,\mu_{2},m}\otimes\EE_{S_{3},T_{3}}^{\lambda,\mu_{3},m}\otimes\EE_{S_{4},T_{4}}^{\lambda,\mu_{4},m}\rangle.\label{eq:tr-proj-1}
\end{align}
The matrix coefficient
\[
\langle hQB_{\lambda}\EE_{S_{1},T_{1}}^{\lambda,\mu_{1},m}\otimes\EE_{S_{2},T_{2}}^{\lambda,\mu_{2},m}\otimes\EE_{S_{3},T_{3}}^{\lambda,\mu_{3},m}\otimes\EE_{S_{4},T_{4}}^{\lambda,\mu_{4},m},\EE_{S_{1},T_{1}}^{\lambda,\mu_{1},m}\otimes\EE_{S_{2},T_{2}}^{\lambda,\mu_{2},m}\otimes\EE_{S_{3},T_{3}}^{\lambda,\mu_{3},m}\otimes\EE_{S_{4},T_{4}}^{\lambda,\mu_{4},m}\rangle
\]
 will now be calculated in stages. Firstly
\begin{align*}
 & B_{\lambda}\EE_{S_{1},T_{1}}^{\lambda,\mu_{1},m}\otimes\EE_{S_{2},T_{2}}^{\lambda,\mu_{2},m}\otimes\EE_{S_{3},T_{3}}^{\lambda,\mu_{3},m}\otimes\EE_{S_{4},T_{4}}^{\lambda,\mu_{4},m}\\
= & \frac{1}{\sqrt{D_{\mu_{1}}(m)D_{\mu_{2}}(m)D_{\mu_{3}}(m)D_{\mu_{4}}(m)}}B_{\lambda}\\
 & \sum_{R_{i}\in\SSTab_{[m]}(\mu_{i})}w_{R_{1}\sqcup S_{1}}\otimes\check{w}_{R_{1}\sqcup T_{1}}\otimes w_{R_{2}\sqcup S_{2}}\otimes\check{w}_{R_{2}\sqcup T_{2}}\otimes w_{R_{3}\sqcup S_{3}}\otimes\check{w}_{R_{3}\sqcup T_{3}}\otimes w_{R_{4}\sqcup S_{4}}\otimes\check{w}_{R_{4}\sqcup T_{4}}\\
= & \frac{1}{\sqrt{D_{\mu_{1}}(m)D_{\mu_{2}}(m)D_{\mu_{3}}(m)D_{\mu_{4}}(m)}}\sum_{R_{i}\in\SSTab_{[m]}(\mu_{i})}\mathbf{1}\{R_{2}\sqcup S_{2}=R_{3}\sqcup T_{3},R_{1}\sqcup T_{1}=R_{4}\sqcup S_{4}\}\\
 & \sum_{t_{1},t_{2}\in\SSTab_{[n]}(\lambda)}w_{t_{1}}\otimes\check{w}_{R_{2}\sqcup T_{2}}\otimes w_{R_{1}\sqcup S_{1}}\otimes\check{w}_{t_{1}}\otimes w_{t_{2}}\otimes\check{w}_{R_{4}\sqcup T_{4}}\otimes w_{R_{3}\sqcup S_{3}}\otimes\check{w}_{t_{2}}.
\end{align*}
We have 
\begin{align*}
 & Q[w_{t_{1}}\otimes\check{w}_{R_{2}\sqcup T_{2}}\otimes w_{R_{1}\sqcup S_{1}}\otimes\check{w}_{t_{1}}\otimes w_{t_{2}}\otimes\check{w}_{R_{4}\sqcup T_{4}}\otimes w_{R_{3}\sqcup S_{3}}\otimes\check{w}_{t_{2}}]\\
= & \frac{\mathbf{1}\{t_{1}\lvert_{\leq m}=R_{2}=R_{1},t_{2}\lvert_{\leq m}=R_{4}=R_{3}\}}{\sqrt{D_{\mu_{1}}(m)D_{\mu_{2}}(m)D_{\mu_{3}}(m)D_{\mu_{4}}(m)}}\\
 & \EE_{t_{1}\lvert_{>m},T_{2}}^{\lambda,\mu_{2},m}\otimes\EE_{S_{1},t_{1}\lvert_{>m}}^{\lambda,\mu_{1},m}\otimes\EE_{t_{2}\lvert_{>m},T_{4}}^{\lambda,\mu_{4},m}\otimes\EE_{S_{3},t_{2}\lvert_{>m}}^{\lambda,\mu_{3},m}.
\end{align*}
Combining the previous two calculations we obtain
\begin{align*}
 & QB_{\lambda}[\EE_{S_{1},T_{1}}^{\lambda,\mu_{1},m}\otimes\EE_{S_{2},T_{2}}^{\lambda,\mu_{2},m}\otimes\EE_{S_{3},T_{3}}^{\lambda,\mu_{3},m}\otimes\EE_{S_{4},T_{4}}^{\lambda,\mu_{4},m}]\\
= & \frac{1}{D_{\mu_{1}}(m)D_{\mu_{2}}(m)D_{\mu_{3}}(m)D_{\mu_{4}}(m)}\sum_{R_{i}\in\SSTab_{[m]}(\mu_{i})}\mathbf{1}\{R_{2}\sqcup S_{2}=R_{3}\sqcup T_{3},R_{1}\sqcup T_{1}=R_{4}\sqcup S_{4}\}\\
 & \sum_{t_{1},t_{2}\in\SSTab_{[n]}(\lambda)}\mathbf{1}\{t_{1}\lvert_{\leq m}=R_{2}=R_{1},t_{2}\lvert_{\leq m}=R_{4}=R_{3}\}\EE_{t_{1}\lvert_{>m},T_{2}}^{\lambda,\mu_{2},m}\otimes\EE_{S_{1},t_{1}\lvert_{>m}}^{\lambda,\mu_{1},m}\otimes\EE_{t_{2}\lvert_{>m},T_{4}}^{\lambda,\mu_{4},m}\otimes\EE_{S_{3},t_{2}\lvert_{>m}}^{\lambda,\mu_{3},m}\\
= & \frac{\mathbf{1}\{S_{2}=T_{3},T_{1}=S_{4},\mu_{1}=\mu_{2}=\mu_{3}=\mu_{4}\}}{D_{\mu_{1}}(m)^{3}}\sum_{s_{1},s_{2}\in\SSTab_{[m+1,n]}(\lambda/\mu_{1})}\\
 & \EE_{s_{1},T_{2}}^{\lambda,\mu_{2},m}\otimes\EE_{S_{1},s_{1}}^{\lambda,\mu_{1},m}\otimes\EE_{s_{2},T_{4}}^{\lambda,\mu_{4},m}\otimes\EE_{S_{3},s_{2}}^{\lambda,\mu_{3},m}.
\end{align*}
Therefore from (\ref{eq:tr-proj-1}) 
\begin{align*}
\mathrm{tr}_{\W_{R}^{\lambda}}\left(hQB_{\lambda}Q\right)= & \sum_{\substack{\mu\subset^{\D}\lambda\\
\ell(\mu)\leq m
}
}\frac{1}{D_{\mu}(m)^{3}}\sum_{S_{i},T_{i}\in\SSTab_{[m+1,n]}(\lambda/\mu)}\mathbf{1}\{S_{2}=T_{3},T_{1}=S_{4}\}\sum_{s_{1},s_{2}\in\SSTab_{[m+1,n]}(\lambda/\mu)}\\
 & \langle h[\EE_{s_{1},T_{2}}^{\lambda,\mu,m}\otimes\EE_{S_{1},s_{1}}^{\lambda,\mu,m}\otimes\EE_{s_{2},T_{4}}^{\lambda,\mu,m}\otimes\EE_{S_{3},s_{2}}^{\lambda,\mu,m}],\EE_{S_{1},T_{1}}^{\lambda,\mu,m}\otimes\EE_{S_{2},T_{2}}^{\lambda,\mu,m}\otimes\EE_{S_{3},T_{3}}^{\lambda,\mu,m}\otimes\EE_{S_{4},T_{4}}^{\lambda,\mu,m}\rangle.
\end{align*}
\end{proof}
Now combining (\ref{eq:Istar-first-form}), Lemma \ref{lem:theta-as-trace-of-B},
and Proposition \ref{prop:trhQBQ}, we obtain:
\begin{prop}
\label{prop:first-integral-done}For $I\in[m+1,n]^{|w|}$
\begin{align*}
\mathcal{I}^{*}(w_{I},\lambda) & =\sum_{\substack{\mu\subset^{\D}\lambda\\
\ell(\mu)\leq m
}
}\frac{1}{D_{\mu}(m)^{3}}\sum_{S_{1},S_{2},S_{3},S_{4},T_{2},T_{4},s_{1},s_{2}\in\SSTab_{[m+1,n]}(\lambda/\mu)}\\
 & \int_{h\in\U(n)^{4}}w_{I}(h)\\
 & \langle h[\EE_{s_{1},T_{2}}^{\lambda,\mu,m}\otimes\EE_{S_{1},s_{1}}^{\lambda,\mu,m}\otimes\EE_{s_{2},T_{4}}^{\lambda,\mu,m}\otimes\EE_{S_{3},s_{2}}^{\lambda,\mu,m}],\EE_{S_{1},T_{1}}^{\lambda,\mu,m}\otimes\EE_{S_{2},T_{2}}^{\lambda,\mu,m}\otimes\EE_{S_{3},T_{3}}^{\lambda,\mu,m}\otimes\EE_{S_{4},T_{4}}^{\lambda,\mu,m}\rangle dh.
\end{align*}
\end{prop}

This is progress because now the integrand is $\U(m)^{4}$-bi-invariant
and hence the integral is now essentially over $\U(m)^{4}\backslash\U(n)^{4}/\U(m)^{4}$
rather than the full $\U(n)^{4}$. Moreover we were able to exploit
the structure of the relator $R$ of $\Gamma_{g}$ to get a lot of
diagonality in the sum: the fact that $\mu$ is the same in each tensor
factor of the equation stated in Proposition \ref{prop:first-integral-done}.
We show how to proceed further in the next section.

\subsection{Second integration: strategy\label{subsec:Second-integration:-overview}}

The only method we really have to proceed from Proposition \ref{prop:first-integral-done}
is to use the Weingarten calculus from $\S\S$\ref{subsec:The-Weingarten-calculus}.
The caveat is that this works well when integrating functions that
are finite products of matrix coefficients in the standard representation
of $\U(n)$. Here we are concerned with the integral
\begin{equation}
\int_{h\in\U(n)^{4}}w_{I}(h)\langle h[\EE_{s_{1},T_{2}}^{\lambda,\mu,m}\otimes\EE_{S_{1},s_{1}}^{\lambda,\mu,m}\otimes\EE_{s_{2},T_{4}}^{\lambda,\mu,m}\otimes\EE_{S_{3},s_{2}}^{\lambda,\mu,m}],\EE_{S_{1},T_{1}}^{\lambda,\mu,m}\otimes\EE_{S_{2},T_{2}}^{\lambda,\mu,m}\otimes\EE_{S_{3},T_{3}}^{\lambda,\mu,m}\otimes\EE_{S_{4},T_{4}}^{\lambda,\mu,m}\rangle dh\label{eq:w_I-hat-after-first-integral}
\end{equation}
appearing in Proposition \ref{prop:first-integral-done}. So to use
the Weingarten calculus we must write the integrand of (\ref{eq:w_I-hat-after-first-integral})
as a product of matrix coefficients in some $\End((\C^{n})^{\otimes K})$.

Since we assume $w\in[\F_{4},\F_{4}]$, by the observation about the
exponents of letters in the reduced word of $w$ from $\S\S$\ref{subsec:Free-groups-and-surface-groups},
(\ref{eq:w_I-hat-after-first-integral}) is equal to the product of
four independent integrals over $\U(n)$ of the form
\begin{equation}
\int_{\U(n)}h_{i_{1}j_{1}}\cdots h_{i_{p}j_{p}}\bar{h}_{i'_{1}j'_{1}}\cdots\bar{h}_{i'_{p}j'_{p}}\langle h[\EE_{s,t}^{\lambda,\mu,m}],\EE_{s',t'}^{\lambda,\mu,m}\rangle dh,\label{eq:mat-coef-int}
\end{equation}
one for each $a,b,c,d$, and where 
\begin{align}
i_{k},j_{k} & \in[m+1,n],\quad\forall k\in[p]\nonumber \\
s,s,t,t' & \in\SSTab_{[m+1,n]}(\lambda/\mu).\label{eq:int-assumptions}
\end{align}
So we will estimate (\ref{eq:mat-coef-int}) under the assumptions
(\ref{eq:int-assumptions}) and later take the product of the four
resulting bounds to estimate (\ref{eq:w_I-hat-after-first-integral}).

We now outline the strategy that we will follow to estimate the integral
(\ref{eq:mat-coef-int}).
\begin{enumerate}
\item We can write 
\[
h_{i_{1}j_{1}}\cdots h_{i_{p}j_{p}}\bar{h}_{i'_{1}j'_{1}}\cdots\bar{h}_{i'_{p}j'_{p}}=\langle\pi_{n}^{p,p}(h)X,X'\rangle
\]
 where $X$, $X'\in(\C^{n})^{\otimes p}\otimes(\check{\C^{n}})^{\otimes p}$
are pure tensors of the standard basis elements and dual standard
basis elements; moreover they only involve
\[
e_{k},\check{e}_{\ell},\quad k,l\in[m+1,n].
\]
 This means that 
\[
h_{i_{1}j_{1}}\cdots h_{i_{p}j_{p}}\bar{h}_{i'_{1}j'_{1}}\cdots\bar{h}_{i'_{p}j'_{p}}\langle h[\EE_{s,t}^{\lambda,\mu,m}],\EE_{s',t'}^{\lambda,\mu,m}\rangle=\langle h[\EE_{s,t}^{\lambda,\mu,m}\otimes X],\EE_{s',t'}^{\lambda,\mu,m}\otimes X'\rangle.
\]
The implication for this to our key integral (\ref{eq:mat-coef-int})
is that
\[
\eqref{eq:mat-coef-int}=\langle P[\EE_{s,t}^{\lambda,\mu,m}\otimes X],\EE_{s',t'}^{\lambda,\mu,m}\otimes X'\rangle=\langle P[\EE_{s,t}^{\lambda,\mu,m}\otimes X],P[\EE_{s',t'}^{\lambda,\mu,m}\otimes X']\rangle
\]
where $P$ is the selfadjoint orthogonal projection on $\End(V^{\lambda})\otimes(\C^{n})^{\otimes p}\otimes(\check{\C^{n}})^{\otimes p}$
onto the $\U(n)$-invariant vectors. Here, we accept some loss of
accuracy by using the Schwarz inequality to deduce
\begin{equation}
|\eqref{eq:mat-coef-int}|\leq\|P[\EE_{s,t}^{\lambda,\mu,m}\otimes X]\|\|P[\EE_{s',t'}^{\lambda,\mu,m}\otimes X']\|.\label{eq:schwarz}
\end{equation}
\item For any $s,t\in\SSTab_{[m+1,n]}(\lambda/\mu)$ we will construct a
vector $A_{s,t}^{\lambda,\mu}\in\End((\C^{n})^{\otimes B})$ in $\S\S$\ref{subsec:Construction-of-A}
such that under the Schur-Weyl intertwiner (cf. Proposition \ref{prop:Schur-weyl-duality})
we have
\[
(\FF\otimes\check{\FF})[A_{s,t}^{\lambda,\mu}]=\EE_{s,t}^{\lambda,\mu,m}\otimes z_{s,t}^{\lambda,\mu}\in\End(W^{\lambda})\otimes\End(V^{\lambda}),
\]
for some $z_{s,t}^{\lambda,\mu}\in\End(V^{\lambda})$. This property
is established in Proposition \ref{prop:pure-tensor-A-prop}. 
\item In Proposition \ref{prop:norm-calc} we calculate $\|A_{s,t}^{\lambda,\mu}\|$.
As $\EE_{s,t}^{\lambda,\mu,m}$ has unit norm, this is the same as
$\|z_{s,t}^{\lambda,\mu}\|$.
\item We now calculate $\hat{P}[A_{s,t}^{\lambda,\mu}\otimes X]$ where
$\hat{P}$ is the projection onto the $\U(n)$ invariant vectors in
$\End((\C^{n})^{\otimes B})\otimes(\C^{n})^{\otimes p}\otimes(\check{\C^{n}})^{\otimes p}$.
We are able to do this using the Weingarten calculus; this is the
reason for the construction of $A_{s,t}^{\lambda,\mu}$. This calculation
takes place in $\S\S$\ref{subsec:Projection-of-AtimesX}.
\item We estimate $\|\hat{P}[A_{s,t}^{\lambda,\mu}\otimes X]\|$ in Proposition
\ref{prop:norm-of-projection}.
\item On the other hand, after suitable identifications, we have 
\[
(\FF\otimes\check{\FF}\otimes\Id_{(\C^{n})^{\otimes p}\otimes(\check{\C^{n}})^{\otimes p}})[\hat{P}[A_{s,t}^{\lambda,\mu}\otimes X]]=P[\EE_{s,t}^{\lambda,\mu,m}\otimes X]\otimes z_{s,t}^{\lambda,\mu},
\]
and hence 
\begin{equation}
\|P[\EE_{s,t}^{\lambda,\mu,m}\otimes X]\|=\frac{\|\hat{P}[A_{s,t}^{\lambda,\mu}\otimes X]\|}{\|A_{s,t}^{\lambda,\mu}\|}.\label{eq:crit-ratio-strat-form}
\end{equation}
Since we have calculated the denominator on the right hand side above,
and bounded the numerator, we have bounded $\|P[\EE_{s,t}^{\lambda,\mu,m}\otimes X]\|$.
We can use exactly the same method to bound $\|P[\EE_{s',t'}^{\lambda,\mu,m}\otimes X']\|$
and putting these estimates into (\ref{eq:schwarz}) we obtain an
upper bound on $|\eqref{eq:mat-coef-int}$\textbar , as desired.
\end{enumerate}
\begin{rem}
It does not seem possible to exactly evaluate (\ref{eq:w_I-hat-after-first-integral})
using this method. Indeed one might hope that 
\[
\eqref{eq:w_I-hat-after-first-integral}=\frac{\langle\hat{P}[A_{s,t}^{\lambda,\mu}\otimes X],\hat{P}[A_{s',t'}^{\lambda,\mu}\otimes X']\rangle}{\|A_{s,t}^{\lambda,\mu}\|\|A_{s',t'}^{\lambda,\mu}\|}.
\]
The construction of $A_{s,t}^{\lambda,\mu},A_{s',t'}^{\lambda,\mu}$
we give here means that the numerator above does not in general coincide
with $\langle P[\EE_{s,t}^{\lambda,\mu,m}\otimes X],P[\EE_{s',t'}^{\lambda,\mu,m}\otimes X']\rangle$.
The use of the Schwarz inequality bypasses this problem, at the cost
of introducing an inequality.
\end{rem}

The components of this strategy will be carried out in the following
sections.

\subsection{Construction of $A_{s,t}^{\lambda,\mu}$\label{subsec:Construction-of-A}}

Now fix $m,\D,b,B\in\N_{0}$, $n=m+\D$, $b\leq B$. Also fix $\lambda\vdash B$
and $\mu\vdash b$ with $\ell(\lambda)\leq n$ and $\mu\subset^{\D}\lambda$
with $\ell(\mu)\leq m$. We also fix 
\[
s,t\in\SSTab_{[m+1,n]}(\lambda/\mu).
\]
From this data we will construct a special vector in $(\C^{n})^{\otimes B}\otimes(\check{\C^{n}})^{\otimes B}\cong\End((\C^{n})^{\otimes B})$.
This will be done in stages.

Recall the weight functions $\omega_{s},\omega_{t}:[m+1,n]\to\N_{0}$
from $\S$\ref{subsec:Young-diagrams-and}. For $z\in\C[S_{k}]$,
we will write $\rho_{n}^{k}(z)$ for the resulting element of $\End((\C^{n})^{\otimes k})$
according to the representation of $S_{k}$ on $(\C^{n})^{\otimes k}$.
We also have fixed linear embeddings for $M\leq n$
\[
\C^{M}=\langle e_{1},\ldots,e_{M}\rangle\subset\C^{n}=\langle e_{1},\ldots,e_{n}\rangle,
\]
similarly $\check{\C}^{M}\subset\check{\C}^{n}$, hence using the
type of isomorphism as in (\ref{eq:hom}) for $M_{1},M_{2}\leq k$
we obtain a fixed linear embedding
\[
\Hom((\C^{n})^{M_{1}},(\C^{n})^{M_{2}})\subset\End((\C^{n})^{\otimes k}).
\]
The skew tableaux $s$ yields a sequence of tableaux
\[
\mu=\nu_{0}(s)\subset^{1}\nu_{1}(s)\subset^{1}\cdots\subset^{1}\nu_{\D}(s)=\lambda
\]
where $\nu_{j}(s)$ is the YD given by $\mu$ along with the boxes
of $s$ containing numbers $\leq m+j$. Let $B_{j}(s)\eqdf|\nu_{j}(s)|$
and similarly define $\nu_{j}(t),B_{j}(t)$ for $0\leq j\leq\D$ using
the boxes of $t$. We have
\[
b=B_{0}(s)\leq\cdots\leq B_{\D}(s)=B.
\]
Recall the projection operators $\p_{\nu}$ from $\S\S$\ref{subsec:Representation-theory-of-sym-groups}.
We use the natural inclusions e.g. $S_{B_{i}(s)}\subset S_{B}$ to
view e.g. $\p_{\nu_{i}(s)}\subset\C[S_{B}]$. 

Let
\begin{align*}
A_{0} & \eqdf\rho_{m}^{b}(\p_{\mu})\in\End((\C^{m})^{\otimes b})\\
 & \cdots\\
A_{i} & \eqdf\rho_{m+i}^{B_{i}(t)}(\p_{\nu_{i}(t)})\left(A_{i-1}\otimes e_{m+i}^{\otimes\omega_{t}(m+i)}\otimes\check{e}_{m+i}^{\otimes\omega_{s}(m+i)}\right)\rho_{m+i}^{B_{i}(s)}(\p_{\nu_{i}(s)})\quad i\geq1\\
A_{i} & \in\Hom((\C^{m+i})^{\otimes B_{i}(s)},(\C^{m+i})^{\otimes B_{i}(t)})\\
 & \cdots\\
A_{s,t}^{\lambda,\mu} & \eqdf A_{\D}\in\Hom((\C^{n})^{\otimes B_{\D}(s)},(\C^{n})^{\otimes B_{\D}(t)})\cong\End((\C^{n})^{\otimes B}).
\end{align*}
We have the following proposition. Recalling the intertwiner $\FF$
from Schur-Weyl duality (Proposition \ref{prop:Schur-weyl-duality}),
we obtain a map
\begin{align}
\FF\otimes\check{\FF}:\End((\C^{n})^{\otimes B}) & \to\sum_{\substack{\lambda_{1},\lambda_{2}\vdash B\\
\ell(\lambda_{i})\leq n
}
}W_{n}^{\lambda_{1}}\otimes\check{W_{n}^{\lambda_{2}}}\otimes V^{\lambda_{1}}\otimes\check{V^{\lambda_{2}}}\nonumber \\
 & \cong\sum_{\substack{\lambda_{1},\lambda_{2}\vdash B\\
\ell(\lambda_{i})\leq n
}
}\Hom(W_{n}^{\lambda_{2}},W_{n}^{\lambda_{1}})\otimes\Hom(V^{\lambda_{2}},V^{\lambda_{1}});\label{eq:FotimesF}
\end{align}
this map is an intertwiner for $\U(n)\times\U(n)\times S_{B}\times S_{B}$.
The second isomorphism is the canonical one explained in $\S\S$\ref{subsec:General-representation-theory}.
\begin{prop}
\label{prop:pure-tensor-A-prop}Under $\FF\otimes\check{\FF}$, $A_{s,t}^{\lambda,\mu}$
maps to 
\[
(\FF\otimes\check{\FF})[A_{s,t}^{\lambda,\mu}]=\EE_{s,t}^{\lambda,\mu,m}\otimes z_{s,t}^{\lambda,\mu}\in\End(W_{n}^{\lambda})\otimes\End(V^{\lambda})
\]
 for some $z_{s,t}^{\lambda,\mu}\in\End(V^{\lambda})$.
\end{prop}

\begin{proof}
The proof relies on the fact that up to scalar multiplication, $\EE_{s,t}^{\lambda,\mu,m}\in\End(W_{n}^{\lambda})$
is uniquely characterized by the following properties:
\begin{description}
\item [{a.}] For $0\leq i\leq\D$, $\EE_{s,t}^{\lambda,\mu,m}$ is in the
$W_{m+i}^{\nu_{i}(t)}\otimes\check{W}_{m+i}^{\nu_{i}(s)}$-isotypic
component of $\End(W^{\lambda})$ for $\U(m+i)\times\U(m+i)$.
\item [{b.}] $\EE_{s,t}^{\lambda,\mu}$ commutes with $\U(m)$, or in other
words, is invariant under the diagonal subgroup $\Delta_{\U(m)}\leq\U(m)\times\U(m)$.
\end{description}
Consider the linear map by which $A_{i+1}$ is obtained from $A_{i}$
for $i\geq0$, that is, 
\[
f_{i}(A_{i})\eqdf\rho_{m+i+1}^{B_{i+1}(t)}(\p_{\nu_{i+1}(t)})\left(A_{i}\otimes e_{m+i+1}^{\otimes\omega_{t}(m+i+1)}\otimes\check{e}_{m+i+1}^{\otimes\omega_{s}(m+i+1)}\right)\rho_{m+i+1}^{B_{i+1}(s)}(\p_{\nu_{i+1}(s)}).
\]
This is a linear intertwiner for the group $\U(m+i)\times\U(m+i)\times S_{B_{i}(t)}\times S_{B_{i}(s)}$. 

\emph{Claim A. For all $0\leq i\leq\D$, $A_{s,t}^{\lambda,\mu}=A_{\D}$
is in the} $W_{m+i}^{\nu_{i}(t)}\otimes\check{W}_{m+i}^{\nu_{i}(s)}$-isotypic
subspace of $\End((\C^{n})^{\otimes B})$ for $\U(m+i)\times\U(m+i)$.

\emph{Proof of Claim A. }We prove by induction on $j$ the following
statement:
\begin{description}
\item [{S$(i,j)$:}] For $0\leq i\leq j\leq\D$, $A_{j}$ is in the $W_{n}^{\nu_{i}(t)}\otimes\check{W}_{n}^{\nu_{i}(s)}$-isotypic
subspace of\\
 $\Hom((\C^{m+j})^{\otimes B_{j}(s)},(\C^{m+j})^{\otimes B_{j}(t)})$
for the group $\U(m+i)\times\U(m+i)$.
\end{description}
Consider the base cases $j=i$. By Schur-Weyl duality, and in particular
(\ref{eq:FotimesF}), the fact that $A_{i}$ has the form
\[
\rho_{m+i}^{B_{i}(t)}(\p_{\nu_{i}(t)})Y\rho_{m+i}^{B_{i}(s)}(\p_{\nu_{i}(s)})
\]
 implies that it is in the $V^{\nu_{i}(t)}\otimes\check{V}^{\nu_{i}(s)}$-isotypic
component for $S_{B_{i}(t)}\times S_{B_{i}(s)}$, but this is the
same as the $W_{n}^{\nu_{i}(t)}\otimes\check{W}_{n}^{\nu_{i}(s)}$-isotypic
subspace for $\U(m+i)\times\U(m+i)$. For the inductive step, by the
intertwining properties of $f_{i}$ stated above, if \textbf{S$(i,j)$}
is true for some $i\leq j,$ it is true for all $(i,j')$ with $j'\geq j$.
Now taking $j=\D$ proves Claim A. \emph{The proof of Claim A is complete.}

\emph{Claim B. $A_{s,t}^{\lambda,\mu}$ commutes with $\U(m)$, or
in other words, is invariant under the diagonal subgroup $\Delta_{\U(m)}\leq\U(m)\times\U(m).$}

\emph{Proof of Claim B. }Using the intertwining properties of the
maps $f_{i}$ again we have for $u\in\U(m)$
\begin{align*}
\pi_{n}^{B}(u)A_{s,t}^{\lambda,\mu}\pi_{n}^{B}(u^{-1}) & =\pi_{n}^{B}(u)f_{\D-1}(f_{\D-2}(\cdots(f_{0}(\rho_{m}^{b}(\p_{\mu}))\cdots))\pi_{n}^{B}(u^{-1})\\
 & =f_{\D-1}(f_{\D-2}(\cdots(f_{0}(\pi_{n}^{b}(u)\rho_{m}^{b}(\p_{\mu})\pi_{n}^{b}(u^{-1}))\cdots))\\
 & =f_{\D-1}(f_{\D-2}(\cdots(f_{0}(\rho_{m}^{b}(\p_{\mu}))\cdots))\\
 & =A_{s,t}^{\lambda,\mu}.
\end{align*}
\emph{This completes the proof of Claim B.}

Now by Claim A with $i=\D$ we have that, in reference to (\ref{eq:FotimesF}),
\[
(\FF\otimes\check{\FF})[A_{s,t}^{\lambda,\mu}]\in\End(W_{n}^{\lambda})\otimes\End(V^{\lambda}),
\]
and the analogs of Claim A and Claim B hold for this element as $\FF\otimes\check{\FF}$
is a $\U(n)\times\U(n)$ intertwiner. Therefore, by the fact that
$\EE_{s,t}^{\lambda,\mu,m}$ is uniquely characterized by \textbf{a
}and \textbf{b},\textbf{ }we must have
\[
(\FF\otimes\check{\FF})[A_{s,t}^{\lambda,\mu}]\in\C\EE_{s,t}^{\lambda,\mu,m}\otimes\End(V^{\lambda}).
\]
This proves Proposition \ref{prop:pure-tensor-A-prop}.
\end{proof}

\subsection{Normalization of $A_{s,t}^{\lambda,\mu}$}

The goal of this section is to calculate $\|A_{s,t}^{\lambda,\mu}\|$,
where the norm is the standard norm on $\End((\C^{n})^{\otimes B})$. 
\begin{prop}
\label{prop:norm-calc}We have 
\[
\|A_{s,t}^{\lambda,\mu}\|^{2}=D_{\mu}(m)\frac{d_{\lambda}^{2}(b!)^{2}}{d_{\mu}(B!)^{2}}\prod_{i=1}^{\D}\omega_{t}(m+i)!\omega_{s}(m+i)!\neq0.
\]
\end{prop}

\begin{proof}
We do this calculation in stages. Firstly, the norm of $A_{0}$ coincides
with the Hilbert-Schmidt norm on $\End((\C^{n})^{\otimes b})$ so
is given by
\[
\|A_{0}\|^{2}=\tr(A_{0}A_{0}^{*})=\tr_{(\C^{m})^{\otimes b}}(\rho_{m}^{b}(\p_{\mu})\rho_{m}^{b}(\p_{\mu})^{*})=\tr_{(\C^{m})^{\otimes b}}(\rho_{m}^{b}(\p_{\mu})^{2})=\tr_{(\C^{m})^{\otimes b}}(\rho_{m}^{b}(\p_{\mu}))
\]
since $\rho_{m}^{b}(\p_{\mu})$ is self-adjoint and idempotent. Using
Schur-Weyl duality (Proposition \ref{prop:Schur-weyl-duality}), we
obtain
\[
\|A_{0}\|^{2}=\tr_{(\C^{m})^{\otimes b}}(\rho_{m}(\p_{\mu}))=D_{\mu}(m)d_{\mu}.
\]
We proceed to $A_{1}$. We have
\begin{align}
\|A_{1}\|^{2} & =\|\rho_{m+1}^{B_{1}(t)}(\p_{\nu_{1}(t)})\left(A_{0}\otimes e_{m+1}^{\otimes\omega_{t}(m+1)}\otimes\check{e}_{m+1}^{\otimes\omega_{s}(m+1)}\right)\rho_{m+1}^{B_{1}(s)}(\p_{\nu_{1}(s)})\|^{2}\nonumber \\
 & =\langle\rho_{m+1}^{B_{1}(t)}(\p_{\nu_{1}(t)})\left(A_{0}\otimes e_{m+1}^{\otimes\omega_{t}(m+1)}\otimes\check{e}_{m+1}^{\otimes\omega_{s}(m+1)}\right)\rho_{m+1}^{B_{1}(s)}(\p_{\nu_{1}(s)}),\left(A_{0}\otimes e_{m+1}^{\otimes\omega_{t}(m+1)}\otimes\check{e}_{m+1}^{\otimes\omega_{s}(m+1)}\right)\rangle\nonumber \\
 & =\frac{d_{\nu_{1}(t)}d_{\nu_{1}(s)}}{B_{1}(s)!B_{1}(t)!}\sum_{g_{t}\in S_{B_{1}(t)},g_{s}\in S_{B_{1}(s)}}\chi_{\nu_{1}(t)}(g_{t})\chi_{\nu_{1}(s)}(g_{s})\nonumber \\
 & \times\langle\rho_{m+1}^{B_{1}(t)}(g_{t})\left(A_{0}\otimes e_{m+1}^{\otimes\omega_{t}(m+1)}\otimes\check{e}_{m+1}^{\otimes\omega_{s}(m+1)}\right)\rho_{m+1}^{B_{1}(s)}(g_{s}),\left(A_{0}\otimes e_{m+1}^{\otimes\omega_{t}(m+1)}\otimes\check{e}_{m+1}^{\otimes\omega_{s}(m+1)}\right)\rangle\nonumber \\
 & =\frac{d_{\nu_{1}(t)}d_{\nu_{1}(s)}}{B_{1}(s)!B_{1}(t)!}\sum_{g_{t}\in S_{m}\times S_{\omega_{t}(m+1)},g_{s}\in S_{m}\times S_{\omega_{s}(m+1)}}\chi_{\nu_{1}(t)}(g_{t})\chi_{\nu_{1}(s)}(g_{s})\nonumber \\
 & \times\langle\rho_{m+1}^{B_{1}(t)}(g_{t})\left(A_{0}\otimes e_{m+1}^{\otimes\omega_{t}(m+1)}\otimes\check{e}_{m+1}^{\otimes\omega_{s}(m+1)}\right)\rho_{m+1}^{B_{1}(s)}(g_{s}),\left(A_{0}\otimes e_{m+1}^{\otimes\omega_{t}(m+1)}\otimes\check{e}_{m+1}^{\otimes\omega_{s}(m+1)}\right)\rangle\nonumber \\
 & =\frac{d_{\nu_{1}(t)}d_{\nu_{1}(s)}}{B_{1}(s)!B_{1}(t)!}\sum_{\substack{g_{t}\in S_{b}\times S_{\omega_{t}(m+1)},g_{s}\in S_{b}\times S_{\omega_{s}(m+1)}\\
g_{t}=(g_{t}^{1},g_{t}^{2}),g_{s}=(g_{s}^{1},g_{s}^{2})
}
}\chi_{\nu_{1}(t)}(g_{t})\chi_{\nu_{1}(s)}(g_{s})\langle\rho_{m}^{b}(g_{t}^{1})A_{0}\rho_{m}^{b}(g_{s}^{1}),A_{0}\rangle\label{eq:normalization-calc-1}
\end{align}
The second equality used that $\rho_{m+1}^{B_{1}(t)}(\p_{\nu_{1}(t)}),\rho_{m+1}^{B_{1}(s)}(\p_{\nu_{1}(s)})$
are self-adjoint projections. 

Now for each $g_{r}=(g_{r}^{1},g_{r}^{2})\in S_{b}\times S_{\omega_{r}(m+1)}$,
$r\in\{s,t\}$ write 
\begin{equation}
\chi_{\nu_{1}(r)}(g_{r})=\sum_{\tau_{1}(r)\vdash b,\tau_{2}(r)\vdash\omega_{r}(m+1)}\LR_{\tau_{1}(r),\tau_{2}(r)}^{\nu_{1}(r)}\chi_{\tau_{1}(r)}(g_{r}^{1})\chi_{\tau_{2}(r)}(g_{r}^{2}).\label{eq:LR-decomp}
\end{equation}
Note that
\[
\langle\rho_{m}^{b}(g_{t}^{1})A_{0}\rho_{m}^{b}(g_{s}^{1}),A_{0}\rangle
\]
is a matrix coefficient in $\End(V^{\mu})$ as a representation of
$S_{b}\times S_{b}$. Thus when we insert (\ref{eq:LR-decomp}) into
(\ref{eq:normalization-calc-1}), the only terms that survive, by
orthogonality of matrix coefficients and Schur orthogonality, have
\[
\tau_{1}(t)=\tau_{1}(s)=\mu,\quad,\tau_{2}(t)=(\omega_{t}(m+1)),\tau_{2}(s)=(\omega_{s}(m+1)).
\]
By Pieri's formula (Lemma \ref{lem:pieri-formula}) $\LR_{\mu,(\omega_{r}(m+1))}^{\nu_{1}(r)}=1$
as $\mu\subset^{1}\nu_{1}(r)$. So (\ref{eq:normalization-calc-1})
together with these observations gives
\begin{align*}
\|A_{1}\|^{2} & =\frac{d_{\nu_{1}(t)}d_{\nu_{1}(s)}}{B_{1}(t)!B_{1}(s)!}\sum_{\substack{g_{t}\in S_{b}\times S_{\omega_{t}(m+1)},g_{s}\in S_{b}\times S_{\omega_{s}(m+1)}\\
g_{t}=(g_{t}^{1},g_{t}^{2}),g_{s}=(g_{s}^{1},g_{s}^{2})
}
}\chi_{\mu}(g_{_{t}}^{1})\chi_{\mu}(g_{t}^{2})\langle\rho_{m}^{b}(g_{t}^{1})A_{0}\rho_{m}^{b}(g_{s}^{1}),A_{0}\rangle\\
 & =\frac{d_{\nu_{1}(t)}d_{\nu_{1}(s)}(b!)^{2}}{B_{1}(t)!B_{1}(s)!d_{\mu}^{2}}\omega_{t}(m+1)!\omega_{s}(m+1)!\langle\rho_{m}^{b}(\p_{\mu})A_{0}\rho_{m}^{b}(\p_{\mu}),A_{0}\rangle\\
 & =\frac{d_{\nu_{1}(t)}d_{\nu_{1}(s)}(b!)^{2}}{B_{1}(t)!B_{1}(s)!d_{\mu}^{2}}\omega_{t}(m+1)!\omega_{s}(m+1)!\|A_{0}\|^{2}.
\end{align*}
Repeating these arguments, mutatis mutandis, gives
\begin{align*}
\|A_{2}\|^{2} & =\frac{d_{\nu_{2}(t)}d_{\nu_{2}(s)}}{B_{2}(t)!B_{2}(s)!}\frac{\omega_{t}(m+2)!\omega_{s}(m+2)!B_{1}(t)!B_{1}(s)!}{d_{\nu_{1}(t)}d_{\nu_{1}(s)}}\|A_{1}\|\\
 & =\frac{d_{\nu_{2}(t)}d_{\nu_{2}(s)}(b!)^{2}}{d_{\mu}^{2}B_{2}(t)!B_{2}(s)!}\omega_{t}(m+2)!\omega_{s}(m+2)!\omega_{t}(m+1)!\omega_{s}(m+1)!\|A_{0}\|^{2}.
\end{align*}
Continuing to iterate this up to $A_{\D}=A_{s,t}^{\lambda,\mu}$ gives
the required
\begin{align*}
\|A_{s,t}^{\lambda,\mu}\|^{2} & =\|A_{0}\|^{2}\frac{d_{\lambda}^{2}(b!)^{2}}{d_{\mu}^{2}(B!)^{2}}\prod_{i=1}^{\D}\omega_{t}(m+i)!\omega_{s}(m+i)!\\
 & =D_{\mu}(m)\frac{d_{\lambda}^{2}(b!)^{2}}{d_{\mu}(B!)^{2}}\prod_{i=1}^{\D}\omega_{t}(m+i)!\omega_{s}(m+i)!.
\end{align*}
\end{proof}

\subsection{Projection of $A_{s,t}^{\lambda,\mu}\otimes X$ to invariant vectors\label{subsec:Projection-of-AtimesX}}

We now calculate $\hat{P}[A_{s,t}^{\lambda,\mu}\otimes X]$ where
$\hat{P}$ is the projection onto the $\U(n)$-invariant vectors in
$\End((\C^{n})^{\otimes B})\otimes\End((\C^{n})^{\otimes p})$, and
\[
X=e_{i_{1}}\otimes\cdots\otimes e_{i_{p}}\otimes\check{e}_{j_{1}}\otimes\cdots\otimes\check{e}_{j_{p}}\in\End(\left(\C^{n}\right)^{\otimes p})
\]
 with 
\[
i_{r},j_{r}\in[m+1,n],\quad r\in[p].
\]
We identify $\End((\C^{n})^{\otimes B})\otimes\End((\C^{n})^{\otimes p})\cong\End((\C^{n})^{B+p})$
via canonical isomorphisms as in (\ref{eq:hom}) and the map
\begin{align}
 & e_{i_{1}}\otimes\cdots\otimes e_{i_{B}}\otimes\check{e}_{j_{1}}\otimes\cdots\otimes\check{e}_{j_{B}}\otimes e_{i_{B+1}}\otimes\cdots\otimes e_{i_{B+p}}\otimes\check{e}_{j_{B+1}}\otimes\cdots\otimes\check{e}_{j_{B+p}}\nonumber \\
\mapsto\, & e_{i_{1}}\otimes\cdots\otimes e_{i_{B}}\otimes e_{i_{B+1}}\otimes\cdots\otimes e_{i_{B+p}}\otimes\check{e}_{j_{1}}\otimes\cdots\otimes\check{e}_{j_{B}}\otimes\check{e}_{j_{B+1}}\otimes\cdots\otimes\check{e}_{j_{B+p}}.\label{eq:EndBEndpEndBplusp}
\end{align}
Using this identification we view $\End((\C^{n})^{\otimes B})\otimes\End((\C^{n})^{\otimes p})$
as a unitary representation of $\U(n)\times S_{B+p}\times\U(n)\times S_{B+p}$;
the map $\hat{P}$ is a $\C[S_{B+p}]$-bimodule morphism. Hence 
\begin{align*}
\hat{P}[A_{s,t}^{\lambda,\mu}\otimes X] & =\hat{P}[\rho_{n}^{B}(\p_{\nu_{\D}(t)})\left(A_{\D-1}\otimes e_{m+\D}^{\otimes\omega_{t}(m+\D)}\otimes\check{e}_{m+\D}^{\otimes\omega_{s}(m+\D)}\right)\rho_{n}^{B}(\p_{\nu_{\D}(s)})\otimes X]\\
 & =\rho_{n}^{B}(\p_{\nu_{\D}(t)})\hat{P}[\left(A_{\D-1}\otimes e_{m+\D}^{\otimes\omega_{t}(m+\D)}\otimes\check{e}_{m+\D}^{\otimes\omega_{s}(m+\D)}\right)\otimes X]\rho_{n}^{B}(\p_{\nu_{\D}(s)})\\
 & =\rho_{n}^{B}(\p_{\nu_{\D}(t)})\hat{P}[A_{\D-1}\otimes X_{1}]\rho_{n}^{B}(\p_{\nu_{\D}(s)})
\end{align*}
where $X_{1}\eqdf e_{m+\D}^{\otimes\omega_{t}(m+\D)}\otimes\check{e}_{m+\D}^{\otimes\omega_{s}(m+\D)}\otimes X$.
Repeating this argument gives
\begin{align}
\hat{P}[A_{s,t}^{\lambda,\mu}\otimes X] & =\rho_{n}^{B}(\p_{\nu_{\D}(t)})\hat{P}[A_{\D-1}\otimes X_{1}]\rho_{n}(\p_{\nu_{\D}(s)})\nonumber \\
 & =\rho_{n}^{B}(\p_{\nu_{\D}(t)}\p_{\nu_{\D-1}(t)})\hat{P}[A_{\D-2}\otimes X_{2}]\rho_{n}^{B}(\p_{\nu_{\D-1}(s)}\p_{\nu_{\D}(s)})\nonumber \\
 & =\cdots\\
 & =\rho_{n}^{B}(\p_{\nu_{\D}(t)}\cdots\p_{\nu_{1}(t)})\hat{P}[A_{0}\otimes X_{\D}]\rho_{n}^{B}(\p_{\nu_{1}(s)}\cdots\p_{\nu_{\D}(s)})\nonumber \\
 & =\rho_{n}^{B}(\p_{\nu_{\D}(t)}\cdots\p_{\nu_{1}(t)})\hat{P}[\rho_{m}^{b}(\p_{\mu})\otimes X_{\D}]\rho_{n}^{B}(\p_{\nu_{1}(s)}\cdots\p_{\nu_{\D}(s)})\label{eq:zeroth-proj-formula}
\end{align}
 where 
\[
X_{\D}\eqdf e_{m+1}^{\otimes\omega_{t}(m+1)}\otimes\check{e}_{m+1}^{\otimes\omega_{s}(m+1)}\otimes\cdots\otimes e_{m+\D}^{\otimes\omega_{t}(m+\D)}\otimes\check{e}_{m+\D}^{\otimes\omega_{s}(m+\D)}\otimes X.
\]
So it remains to calculate $\hat{P}[\rho_{m}^{b}(\p_{\mu})\otimes X_{\D}]$.

The Weingarten calculus (Proposition \ref{prop:Weingarten}) yields
\begin{equation}
\hat{P}[\rho_{m}^{b}(\p_{\mu})\otimes X_{\D}]=\rho_{n}^{B+p}(\Phi[\rho_{m}^{b}(\p_{\mu})\otimes X_{\D}]\Wg_{n,B+p})\label{eq:first-proj-formula}
\end{equation}
 where 
\[
\Phi[\rho_{m}^{b}(\p_{\mu})\otimes X_{\D}]=\sum_{\sigma\in S_{B+p}}\tr[\rho_{n}^{B+p}(\sigma)^{-1}(\rho_{m}(\p_{\mu})\otimes X'_{\D})]\sigma
\]
and $X'_{\D}\in\End(\langle e_{m+1},\ldots,e_{n}\rangle^{\otimes B-b+p})$
is the element corresponding to $X_{\D}$ under the fixed isomorphism
\begin{align*}
 & \left(\C^{n}\right)^{\otimes\omega_{t}(m+1)}\otimes\check{\left(\C^{n}\right)}^{\otimes\omega_{s}(m+1)}\otimes\cdots\otimes\left(\C^{n}\right){}^{\otimes\omega_{t}(m+\D)}\otimes\check{\left(\C^{n}\right)}^{\otimes\omega_{s}(m+\D)}\otimes\End(\left(\C^{n}\right)^{\otimes p})\\
\cong & \End(\left(\C^{n}\right)^{\otimes B-b+p})
\end{align*}
that preserves the order of the factors in both 
\[
\left(\C^{n}\right)^{\otimes\omega_{t}(m+1)}\otimes\cdots\otimes\left(\C^{n}\right){}^{\otimes\omega_{t}(m+\D)}\otimes(\C^{n})^{\otimes p},\,\check{\left(\C^{n}\right)}^{\otimes\omega_{s}(m+1)}\otimes\cdots\otimes\check{\left(\C^{n}\right)}^{\otimes\omega_{s}(m+\D)}\otimes\left(\check{\C^{n}}\right){}^{\otimes p},
\]
so that 
\begin{equation}
X'_{\D}\eqdf e_{m+1}^{\otimes\omega_{t}(m+1)}\otimes\cdots\otimes e_{m+\D}^{\otimes\omega_{t}(m+\D)}\otimes e_{i_{1}}\otimes\cdots\otimes e_{i_{p}}\otimes\check{e}_{m+1}^{\otimes\omega_{s}(m+1)}\otimes\cdots\otimes\check{e}_{m+\D}^{\otimes\omega_{s}(m+\D)}\otimes\check{e}_{j_{1}}\otimes\cdots\otimes\check{e}_{j_{p}}.\label{eq:X-dash-def}
\end{equation}

We make an an observation that if $q_{r}$ denotes the number of $\ell\in[p]$
such that $j_{\ell}=m+r$, and $p_{r}$ denotes the number of $\ell\in[p]$
such that $i_{\ell}=m+r$, if we do \textbf{not} have 
\begin{align}
 & (\omega_{s}(m+1)+q_{1},\omega_{s}(m+2)+q_{2},\ldots,\omega_{s}(m+\D)+q_{\D})\label{eq:partition-equality}\\
= & (\omega_{t}(m+1)+p_{1},\omega_{t}(m+2)+p_{2},\ldots,\omega_{t}(m+\D)+p_{\D})\nonumber 
\end{align}
then $\tr[\rho_{n}^{B+p}(\sigma)^{-1}(\rho_{m}^{b}(\p_{\mu})\otimes X'_{\D})]=0$
for all $\sigma\in S_{B+p}$, hence $\Phi[\rho_{m}^{b}(\p_{\mu})\otimes X_{\D}]=0$,
$\hat{P}[\rho_{m}^{b}(\p_{\mu})\otimes X_{\D}]=0$, and $\hat{P}[A_{s,t}^{\lambda,\mu}\otimes X]=0$.

\emph{So now assume (\ref{eq:partition-equality}) holds.}

To calculate $\Phi[\rho_{m}^{b}(\p_{\mu})\otimes X_{\D}]$, we note
that 
\[
\tr[\rho_{n}^{B+p}(\sigma)^{-1}(\rho_{m}^{b}(\p_{\mu})\otimes X'_{\D})]=0
\]
 unless $\sigma=(\sigma_{1},\sigma_{2})\in S_{b}\times S_{B-b+p}$,
and in this case, by Schur-Weyl duality (Proposition \ref{prop:Schur-weyl-duality})
\begin{align*}
\tr[\rho_{n}^{B+p}(\sigma)^{-1}(\rho_{m}^{b}(\p_{\mu})\otimes X'_{\D})] & =\tr[\rho_{m}^{b}(\sigma_{1}^{-1}\p_{\mu})\otimes(\rho_{n}^{B-b+p}(\sigma_{2}^{-1})X'_{\D})]\\
 & =\tr[\rho_{n}^{B-b+p}(\sigma_{2}^{-1})X'_{\D})]\tr[\rho_{m}^{b}(\sigma_{1}^{-1}\p_{\mu})]\\
 & =\tr[\rho_{n}^{B-b+p}(\sigma_{2}^{-1})X'_{\D})]\sum_{\substack{\mu'\vdash b\\
\ell(\mu')\leq m
}
}D_{\mu'}(m)\chi_{\mu'}(\sigma_{1}^{-1}\p_{\mu})\\
 & =\tr[\rho_{n}^{B-b+p}(\sigma_{2}^{-1})X'_{\D})]D_{\mu}(m)\chi_{\mu}(\sigma_{1}^{-1}).
\end{align*}
Therefore 
\begin{align}
\Phi[\rho_{m}^{b}(\p_{\mu})\otimes X_{\D}] & =\sum_{(\sigma_{1},\sigma_{2})\in S_{b}\times S_{B-b+p}}\tr[\rho_{n}^{B+p}(\sigma)^{-1}(\rho_{m}^{b}(\p_{\mu})\otimes X'_{\D})]\sigma\nonumber \\
 & =D_{\mu}(m)\sum_{(\sigma_{1},\sigma_{2})\in S_{b}\times S_{B-b+p}}\tr[\rho_{n}^{B-b+p}(\sigma_{2}^{-1})X'_{\D}]\chi_{\mu}(\sigma_{1}^{-1})(\sigma_{1},\sigma_{2}).\label{eq:Phi-calc-1}
\end{align}
Here we write $(\sigma_{1},\sigma_{2})$ for the element of $S_{B+p}$
corresponding to $(\sigma_{1},\sigma_{2})\in S_{b}\times S_{B-b+p}$.

Given (\ref{eq:X-dash-def}), the value of $\tr[\rho_{n}^{B-b+p}(\sigma_{2}^{-1})X'_{\D}]$
is either $1$ or $0$. The values of $\sigma_{2}$ for which $\tr[\rho_{n}^{B-b+p}(\sigma_{2}^{-1})X'_{\D}]=1$
are a right coset 
\[
\pi_{0}S_{\Pi}
\]
where $\pi_{0}\in S_{B-b+p}$ and $S_{\Pi}\leq S_{B-b+p}$ is the
subgroup of elements respecting a certain partition $\Pi$ of $[B-b+p]$
dictated by the indices of $X'_{\D}$. The blocks of $\Pi$ have sizes
given by either side of (\ref{eq:partition-equality}). Note $\sum_{i=1}^{\D}p_{i}=\sum_{i=1}^{\D}q_{i}=p$.
Moreover we have 
\[
S_{s}\eqdf S_{\omega_{s}(m+1)}\times S_{\omega_{s}(m+2)}\times\cdots\times S_{\omega_{s}(m+\D-1)}\times S_{\omega_{s}(n)}\leq S_{\Pi}.
\]
From (\ref{eq:Phi-calc-1}) we now have 
\begin{align}
\Phi[\rho_{m}^{b}(\p_{\mu})\otimes X_{\D}] & =D_{\mu}(m)\sum_{(\sigma_{1},\sigma_{2})\in S_{b}\times S_{B-b+p}}\tr[\rho_{n}^{B-b+p}(\sigma_{2}^{-1})X'_{\D}]\chi_{\mu}(\sigma_{1}^{-1})(\sigma_{1},\sigma_{2})\nonumber \\
 & =D_{\mu}(m)\sum_{(\sigma_{1},\sigma_{2})\in S_{b}\times S_{\Pi}}\chi_{\mu}(\sigma_{1}^{-1})(\sigma_{1},\pi_{0}\sigma_{2})\nonumber \\
 & =D_{\mu}(m)\sum_{[\tau]\in S_{\Pi}/S_{s}}\sum_{(\sigma_{1},\sigma_{2})\in S_{b}\times S_{s}}\chi_{\mu}(\sigma_{1}^{-1})(\sigma_{1},\pi_{0}\tau\sigma_{2})\nonumber \\
 & =D_{\mu}(m)\sum_{[\tau]\in S_{\Pi}/S_{s}}(\id_{S_{b}},\pi_{0}\tau)\sum_{(\sigma_{1},\sigma_{2})\in S_{b}\times S_{s}}\chi_{\mu}(\sigma_{1}^{-1})(\sigma_{1},\sigma_{2})\nonumber \\
 & =\frac{b!}{d_{\mu}}|S_{s}|D_{\mu}(m)\sum_{[\tau]\in S_{\Pi}/S_{s}}(\id_{S_{b}},\pi_{0}\tau)\q_{s}\label{eq:Phi-calc}
\end{align}
where 
\[
\q_{s}\eqdf\frac{d_{\mu}}{b!}\frac{1}{|S_{s}|}\sum_{(\sigma_{1},\sigma_{2})\in S_{b}\times S_{s}}\chi_{\mu}(\sigma_{1}^{-1})(\sigma_{1},\sigma_{2})\in\C[S_{B}]\subset\C[S_{B+p}]
\]
 is an element which in any unitary representation of $S_{B+p}$ gives
the orthogonal projection onto the $V^{\mu}\otimes\triv_{S_{s}}$-isotypic
subspace for $S_{b}\times S_{s}$.

Combining (\ref{eq:first-proj-formula}) and (\ref{eq:Phi-calc})
and recalling the definition of $\Wg_{n,B+p}$ from (\ref{eq:Wg-def})
gives
\begin{align*}
\hat{P}[\rho_{m}^{b}(\p_{\mu})\otimes X_{\D}] & =\rho_{n}^{B+p}(\Phi[\rho_{m}^{b}(\p_{\mu})\otimes X_{\D}]\Wg_{n,B+p})\\
 & =\frac{b!}{d_{\mu}}|S_{s}|D_{\mu}(m)\sum_{[\tau]\in S_{\Pi}/S_{s}}\rho_{n}^{B+p}(\id_{S_{b}},\pi_{0}\tau)\rho_{n}^{B+p}(\q_{s})\rho_{n}^{B+p}(\Wg_{n,B+p})\\
 & =\frac{b!}{d_{\mu}}\frac{|S_{s}|}{(B+p)!}D_{\mu}(m)\sum_{\substack{\lambda'\vdash B+p\\
\ell(\lambda')\leq n
}
}\frac{d_{\lambda'}}{D_{\lambda'}(n)}\sum_{[\tau]\in S_{\Pi}/S_{s}}\rho_{n}^{B+p}(\id_{S_{b}},\pi_{0}\tau)\rho_{n}(\q_{s}\p_{\lambda'}).
\end{align*}
Combined with (\ref{eq:zeroth-proj-formula}) this gives
\begin{align}
\hat{P}[A_{s,t}^{\lambda,\mu}\otimes X] & =\frac{b!}{d_{\mu}}\frac{|S_{s}|}{(B+p)!}D_{\mu}(m)\nonumber \\
 & \sum_{\substack{\lambda'\vdash B+p\\
\ell(\lambda')\leq n
}
}\frac{d_{\lambda'}}{D_{\lambda'}(n)}\rho_{n}^{B+p}\left(\p_{\nu_{\D}(t)}\cdots\p_{\nu_{1}(t)}\sum_{[\tau]\in S_{\Pi}/S_{s}}(\id_{S_{b}},\pi_{0}\tau)\q_{s}\p_{\lambda'}\p_{\nu_{1}(s)}\cdots\p_{\nu_{\D}(s)}\right)\label{eq:main-proj-formula}
\end{align}
that is the main result of this section.

\subsection{The norm of the projection of $A_{s,t}^{\lambda,\mu}\otimes X$}

We now turn to the calculation of $\|\hat{P}[A_{s,t}^{\lambda,\mu}\otimes X]\|$
where the norm is the natural one induced the standard inner product
on $\C^{n}$. We identify $\End((\C^{n})^{\otimes B+p})\cong\End((\C^{n})^{\otimes B})\otimes\End(\left(\C^{n}\right)^{\otimes p})$
using the isomorphism (\ref{eq:EndBEndpEndBplusp}); after this identification
the norm on $\End((\C^{n})^{\otimes B})\otimes\End(\left(\C^{n}\right)^{\otimes p})$
coincides with the Hilbert-Schmidt norm on $\End((\C^{n})^{\otimes B+p})$.
The main result of this section is:
\begin{prop}
\label{prop:norm-of-projection}If (\ref{eq:partition-equality})
holds we have 
\begin{align*}
\|\hat{P}[A_{s,t}^{\lambda,\mu}\otimes X]\|^{2} & \leq\frac{(b!)^{2}|S_{\Pi}|^{2}D_{\mu}(m)^{2}}{d_{\mu}((B+p)!)^{2}}\sum_{\substack{\lambda'\vdash B+p\\
\ell(\lambda')\leq n
}
}\frac{d_{\lambda'}^{2}d_{\lambda'/\lambda}}{D_{\lambda'}(n)}.
\end{align*}
Otherwise, $\|\hat{P}[A_{s,t}^{\lambda,\mu}\otimes X]\|=0$.
\end{prop}

We view $(\C^{n})^{\otimes B+p}$ with its given inner product as
a unitary representation of $\U(n)\times S_{B+p}$ in the usual way
involved in Schur-Weyl duality. For $\lambda'\vdash m+p$ with $\ell(\lambda')\leq n$
we let $Z^{\lambda'}$ denote the $W_{n}^{\lambda'}\times V^{\lambda'}$-isotypic
subspace of $(\C^{n})^{\otimes B+p}$ for $\U(n)\times S_{B+p}$.
By Schur-Weyl duality (Proposition \ref{prop:Schur-weyl-duality})
this space is itself an irreducible representation of $\U(n)\times S_{B+p}$
isomorphic to $W_{n}^{\lambda'}\times V^{\lambda'}$. Since $\hat{P}[A_{s,t}^{\lambda,\mu}\otimes X]\in\rho_{n}^{B+p}(\C[S_{B+p}])$,
it preserves each $Z^{\lambda'}$. Therefore we have
\begin{equation}
\|\hat{P}[A_{s,t}^{\lambda,\mu}\otimes X]\|^{2}=\sum_{\substack{\lambda'\vdash B+p\\
\ell(\lambda')\leq n
}
}\|\hat{P}[A_{s,t}^{\lambda,\mu}\otimes X]\|_{\lambda'}^{2}\label{eq:norm-splitting-formula}
\end{equation}
where $\|\hat{P}[A_{s,t}^{\lambda,\mu}\otimes X]\|_{\lambda'}$ is
the Hilbert-Schmidt norm of $\hat{P}[A_{s,t}^{\lambda,\mu}\otimes X]$
acting on $Z^{\lambda'}$.

We may assume (\ref{eq:partition-equality}) holds, otherwise $\hat{P}[A_{s,t}^{\lambda,\mu}\otimes X]=0$
and Proposition \ref{prop:norm-of-projection} is proved. In this
case, using that $\p_{\lambda'}$ is central in $\C[S_{B+p}]$, inspection
of (\ref{eq:main-proj-formula}) gives that 
\begin{equation}
\|\hat{P}[A_{s,t}^{\lambda,\mu}\otimes X]\|_{\lambda'}=\frac{b!}{d_{\mu}}\frac{|S_{s}|}{(B+p)!}D_{\mu}(m)\frac{d_{\lambda'}}{\sqrt{D_{\lambda'}(n)}}\|\p_{\nu_{\D}(t)}\cdots\p_{\nu_{1}(t)}\sum_{[\tau]\in S_{\Pi}/S_{s}}(\id_{S_{b}},\pi_{0}\tau)\q_{s}\p_{\nu_{1}(s)}\cdots\p_{\nu_{\D}(s)}\|_{V^{\lambda'}}\label{eq:first-proj-norm-formula}
\end{equation}
where for $z\in\C[S_{B+p}]$ we write $\|z\|_{V^{\lambda'}}$ for
the Hilbert-Schmidt norm of $z$ acting on $V^{\lambda'}$. Notice
that we obtained a factor $\sqrt{D_{\lambda'}(n)}$ from the multiplicity
of $V^{\lambda'}$ in $Z^{\lambda'}$. Due to the presence of $\p_{\nu_{\D}(s)}=\p_{\lambda}$
in (\ref{eq:first-proj-norm-formula}), the right hand side of \ref{eq:first-proj-norm-formula}
is zero unless $\lambda\subset_{p}\lambda'$ and hence
\begin{equation}
\|\hat{P}[A_{s,t}^{\lambda,\mu}\otimes X]\|^{2}=\sum_{\substack{\lambda\subset_{p}\lambda'\\
\ell(\lambda')\leq n
}
}\|\hat{P}[A_{s,t}^{\lambda,\mu}\otimes X]\|_{\lambda'}^{2}.\label{eq:norm-splitting-formula2}
\end{equation}
To calculate the Hilbert-Schmidt norm in \ref{eq:first-proj-norm-formula}
above we study in detail the operator
\[
Q\eqdf\q_{s}\p_{\nu_{1}(s)}\cdots\p_{\nu_{\D}(s)}.
\]
Let $\pi_{\lambda'}$ denote the representation of $S_{B+p}$ on $V^{\lambda'}$.
\begin{lem}
\label{lem:sa-idempotent}With notation as above, $\pi_{\lambda'}(Q)$
is a self-adjoint idempotent element of $\End(V^{\lambda'})$. Similarly,
$\pi_{\lambda'}(\p_{\nu_{1}(t)}\cdots\p_{\nu_{\D}(t)})$ is a self-adjoint
idempotent. 
\end{lem}

\begin{proof}
Recall that $\q_{s}$ is projection onto the $V^{\mu}\otimes\triv_{S_{s}}$-isotypic
subspace for 
\begin{align*}
S_{b}\times S_{s} & =S_{b}\times S_{\omega_{s}(m+1)}\times S_{\omega_{s}(m+2)}\times\cdots\times S_{\omega_{s}(m+\D-1)}\times S_{\omega_{s}(n)}\leq S_{B}\leq S_{B+p}.
\end{align*}
One see that this commutes with all $\p_{\nu_{1}(s)},\ldots,\p_{\nu_{\D}(s)}$
in $\C[S_{B+p}]$, because for any $i\in[\D]$, we can write 
\[
\q_{s}=Q_{1}Q_{2}
\]
 where $Q_{1},Q_{2}$ are commuting projections supported respectively
on the two factors of $S_{B}=S_{B_{i}(s)}\times S_{B-B_{i}(s)}$.
But $\p_{\nu_{i}(s)}$ commutes with all elements of $S_{B}$ in the
$S_{B-B_{i}(s)}$ factor, so commutes with $Q_{2}$. It also commutes
with $Q_{1}$ since $\p_{\nu_{i}(s)}$ is a central element of $\C[S_{B_{i}(s)}]$.
The $\p_{\nu_{1}(s)},\ldots,\p_{\nu_{\D}(s)}$ are easily seen to
commute with each other: $\p_{\nu_{i}(s)}$ is central in $\C[S_{B_{i}(s)}]$
and hence commutes with all elements $\p_{\nu_{j}(s)}$ with $j\leq i$.
Thus $\pi_{\lambda'}(Q)$ is a product of commuting self-adjoint idempotents
and so is itself a self-adjoint idempotent. The same arguments show
that $\pi_{\lambda'}(\p_{\nu_{1}(t)}\cdots\p_{\nu_{\D}(t)})$ is a
self-adjoint idempotent.
\end{proof}
We now calculate the dimension of $\pi_{\lambda'}(Q)V^{\lambda'}$.
This is the eigenspace of $\pi_{\lambda'}(Q)$ with eigenvalue 1 and
also the orthogonal complement to the kernel of $\pi_{\lambda'}(Q)$. 
\begin{lem}
With notation as above, if $\lambda\subset_{p}\lambda'$, $\dim\left(\pi_{\lambda'}(Q)V^{\lambda'}\right)=d_{\lambda'/\lambda}d_{\mu}.$
\end{lem}

\begin{proof}
Firstly, as $Q\in\C[S_{B}]$ and contains the factor
\[
\p_{\nu_{\D}(s)}=\p_{\lambda},
\]
we have 
\begin{equation}
\dim\left(\pi_{\lambda'}(Q)V^{\lambda'}\right)=\dim\Hom_{\C[S_{B}]}(V^{\lambda},V^{\lambda'})\dim\left(\pi_{\lambda}(Q)V^{\lambda}\right)=d_{\lambda'/\lambda}\dim\left(\pi_{\lambda}(Q)V^{\lambda}\right).\label{eq:dim-temp-1}
\end{equation}
Given this we now consider $Q$ acting on $V^{\lambda}$. This is
the same as 
\[
\pi_{\lambda}(Q_{1}),\quad Q_{1}\eqdf\q_{s}\p_{\nu_{1}(s)}\cdots\p_{\nu_{\D-1}(s)}.
\]
We have $\pi_{\lambda}(Q_{1})$ is a self-adjoint idempotent by the
same proof as that of Lemma \ref{lem:sa-idempotent} . Then $Q_{1}$
is supported on, and preserves, the $V^{\nu_{\D-1}(s)}\otimes\triv$-isotypic
component for $S_{B_{\D-1}}\times S_{\omega_{s}(n)}\leq S_{B}$. By
Pieri's formula (Lemma \ref{lem:pieri-formula}), as $\nu_{\D-1}\subset^{1}\nu_{\D}=\lambda$,
this isotypic component consists of a unique isomorphic copy of $V^{\nu_{\D-1}(s)}\otimes\triv$
in $V^{\lambda}$. Moreover, when identifying this space with $V^{\nu_{\D-1}(s)}$,
the action of $\pi_{\lambda}(Q_{1})$ is given by
\[
\pi_{\nu_{\D-1}(s)}(Q_{2}),\quad Q_{2}\eqdf\q_{2}\p_{\nu_{1}(s)}\cdots\p_{\nu_{\D-2}(s)}
\]
where $\q_{2}$ is the projection onto the $V^{\mu}\otimes\triv$-isotypic
subspace for 
\[
S_{b}\times S_{\omega_{s}(m+1)}\times S_{\omega_{s}(m+2)}\times\cdots\times S_{\omega_{s}(m+\D-1)}.
\]
This gives
\[
\dim\left(\pi_{\lambda}(Q)V^{\lambda}\right)=\dim(\pi_{\nu_{\D-1}(s)}(Q_{2})V^{\nu_{\D-1}(s)}).
\]
Iterating this argument, using Pieri's formula each time, gives eventually
that 
\[
\dim\left(\pi_{\lambda}(Q)V^{\lambda}\right)=\dim(\pi_{\mu}(\p_{\mu})V^{\mu})=d_{\mu}.
\]
Thus in total, going back to (\ref{eq:dim-temp-1}), we obtain the
required
\[
\dim\left(\pi_{\lambda'}(Q)V^{\lambda'}\right)=d_{\lambda'/\lambda}d_{\mu}.
\]
\end{proof}
\begin{proof}[Proof of Proposition \ref{prop:norm-of-projection}]
Pick an orthonormal basis $\{v_{1},\ldots,v_{D}\}$ for $\pi_{\lambda'}(Q)V^{\lambda'}$
with $D=d_{\lambda'/\lambda}d_{\mu}$ and extend this to an orthonormal
basis of $V^{\lambda'}$. In regards to (\ref{eq:first-proj-norm-formula})
we have 
\begin{align}
 & \|\p_{\nu_{\D}(t)}\cdots\p_{\nu_{1}(t)}\sum_{[\tau]\in S_{\Pi}/S_{s}}(\id_{S_{b}},\pi_{0}\tau)\q_{s}\p_{\nu_{1}(s)}\cdots\p_{\nu_{\D}(s)}\|_{V^{\lambda'}}^{2}\nonumber \\
= & \|\p_{\nu_{\D}(t)}\cdots\p_{\nu_{1}(t)}\sum_{[\tau]\in S_{\Pi}/S_{s}}(\id_{S_{b}},\pi_{0}\tau)Q\|_{V^{\lambda'}}^{2}\nonumber \\
= & \sum_{i=1}^{D}\|\pi_{\lambda'}(\p_{\nu_{\D}(t)}\cdots\p_{\nu_{1}(t)})\sum_{[\tau]\in S_{\Pi}/S_{s}}\pi_{\lambda'}(\id_{S_{b}},\pi_{0}\tau)\pi_{\lambda'}(Q)v_{i}\|^{2}\nonumber \\
= & \sum_{i=1}^{D}\|\pi_{\lambda'}(\p_{\nu_{\D}(t)}\cdots\p_{\nu_{1}(t)})\sum_{[\tau]\in S_{\Pi}/S_{s}}\pi_{\lambda'}(\id_{S_{b}},\pi_{0}\tau)v_{i}\|^{2}.\label{eq:nomr-proj-temp}
\end{align}
Now, each $\pi_{\lambda'}(\id_{S_{b}},\pi_{0}\tau)$ is unitary and
by the second statement of Lemma \ref{lem:sa-idempotent}, $\pi_{\lambda'}(\p_{\nu_{\D}(t)}\cdots\p_{\nu_{1}(t)})$
is a self-adjoint idempotent and hence has operator norm bounded by
one. Hence for $i\in[D]$
\[
\|\pi_{\lambda'}(\p_{\nu_{\D}(t)}\cdots\p_{\nu_{1}(t)})\sum_{[\tau]\in S_{\Pi}/S_{s}}\pi_{\lambda'}(\id_{S_{b}},\pi_{0}\tau)v_{i}\|\leq[S_{\Pi}:S_{s}].
\]
Combining this with (\ref{eq:nomr-proj-temp}) we obtain
\begin{align*}
\|\p_{\nu_{\D}(t)}\cdots\p_{\nu_{1}(t)}\sum_{[\tau]\in S_{\Pi}/S_{s}}(\id_{S_{b}},\pi_{0}\tau)\q_{s}\p_{\nu_{1}(s)}\cdots\p_{\nu_{\D}(s)}\|_{V^{\lambda'}} & \leq\sqrt{D}[S_{\Pi}:S_{s}]\\
 & =\sqrt{d_{\lambda'/\lambda}d_{\mu}}[S_{\Pi}:S_{s}].
\end{align*}
Using this in (\ref{eq:first-proj-norm-formula}) gives
\begin{align*}
\|\hat{P}[A_{s,t}^{\lambda,\mu}\otimes X]\|_{\lambda'} & \leq\frac{b!}{d_{\mu}}\frac{|S_{s}|}{(B+p)!}D_{\mu}(m)\frac{d_{\lambda'}}{\sqrt{D_{\lambda'}(n)}}\sqrt{d_{\lambda'/\lambda}d_{\mu}}[S_{\Pi}:S_{s}]\\
 & =\frac{b!}{\sqrt{d_{\mu}}}\frac{|S_{\Pi}|}{(B+p)!}D_{\mu}(m)\frac{d_{\lambda'}}{\sqrt{D_{\lambda'}(n)}}\sqrt{d_{\lambda'/\lambda}}.
\end{align*}
Finally incorporating this estimate into (\ref{eq:norm-splitting-formula})
gives
\[
\|\hat{P}[A_{s,t}^{\lambda,\mu}\otimes X]\|^{2}\leq\frac{(b!)^{2}|S_{\Pi}|^{2}D_{\mu}(m)^{2}}{d_{\mu}((B+p)!)^{2}}\sum_{\substack{\lambda'\vdash B+p\\
\ell(\lambda')\leq n
}
}\frac{d_{\lambda'}^{2}d_{\lambda'/\lambda}}{D_{\lambda'}(n)}.
\]
\end{proof}

\subsection{Completing the outlined strategy}
\begin{prop}
\label{prop:strategy-complete-norm-bound}With notation as in the
previous sections,
\[
\|P[\EE_{s,t}^{\lambda,\mu,m}\otimes X]\|^{2}\leq(|\lambda/\mu|+p)^{2p}\frac{(B!)^{2}}{((B+p)!)^{2}}D_{\mu}(m)\sum_{\substack{\lambda\subset_{p}\lambda'\\
\ell(\lambda')\leq n
}
}\frac{d_{\lambda'}^{2}d_{\lambda'/\lambda}}{d_{\lambda}^{2}D_{\lambda'}(n)}.
\]
\end{prop}

\begin{proof}
From the argument of our strategy in $\S\S$\ref{subsec:Second-integration:-overview},
we have 
\begin{equation}
\|P[\EE_{s,t}^{\lambda,\mu,m}\otimes X]\|^{2}=\frac{\|\hat{P}[A_{s,t}^{\lambda,\mu}\otimes X]\|^{2}}{\|A_{s,t}^{\lambda,\mu}\|^{2}}.\label{eq:proj-norm-ratio}
\end{equation}

If (\ref{eq:partition-equality}) fails to hold, then Propositions
\ref{prop:norm-calc} and \ref{prop:norm-of-projection} give $\|P[\EE_{s,t}^{\lambda,\mu,m}\otimes X]\|=0$,
proving Proposition \ref{prop:strategy-complete-norm-bound} in this
instance.

Now suppose (\ref{eq:partition-equality}) holds. Then Propositions
\ref{prop:norm-calc} and \ref{prop:norm-of-projection} combined
with (\ref{eq:proj-norm-ratio}) give
\begin{align}
\|P[\EE_{s,t}^{\lambda,\mu,m}\otimes X]\|^{2} & \leq\left(D_{\mu}(m)\frac{d_{\lambda}^{2}(b!)^{2}}{d_{\mu}(B!)^{2}}\prod_{i=1}^{\D}\omega_{t}(m+i)!\omega_{s}(m+i)!\right)^{-1}\frac{(b!)^{2}|S_{\Pi}|^{2}D_{\mu}(m)^{2}}{d_{\mu}((B+p)!)^{2}}\sum_{\substack{\lambda\subset_{p}\lambda'\\
\ell(\lambda')\leq n
}
}\frac{d_{\lambda'}^{2}d_{\lambda'/\lambda}}{D_{\lambda'}(n)}\nonumber \\
 & =\frac{(B!)^{2}}{((B+p)!)^{2}}\frac{|S_{\Pi}|^{2}}{\prod_{i=1}^{\D}\omega_{t}(m+i)!\omega_{s}(m+i)!}D_{\mu}(m)\sum_{\substack{\lambda\subset_{p}\lambda'\\
\ell(\lambda')\leq n
}
}\frac{d_{\lambda'}^{2}d_{\lambda'/\lambda}}{d_{\lambda}^{2}D_{\lambda'}(n)}.\label{eq:temp3}
\end{align}
Recalling that the block sizes of $\Pi$ are given by either side
of $\eqref{eq:partition-equality}$ we have
\begin{align*}
\frac{|S_{\Pi}|^{2}}{\prod_{i=1}^{\D}\omega_{t}(m+i)!\omega_{s}(m+i)!} & =\prod_{i=1}^{\D}\frac{(\omega_{t}(m+i)+p_{i})!}{\omega_{t}(m+i)!}\frac{(\omega_{s}(m+i)+q_{i})!}{\omega_{s}(m+i)!}\\
 & \leq\prod_{i=1}^{\D}(\omega_{t}(m+i)+p_{i})^{p_{i}}(\omega_{s}(m+i)+q_{i})^{q_{i}}\\
 & \leq\prod_{i=1}^{\D}(|\lambda/\mu|+p)^{p_{i}}(|\lambda/\mu|+p)^{q_{i}}\\
 & \leq(|\lambda/\mu|+p)^{2p}
\end{align*}
using that $\sum_{i=1}^{\D}p_{i}=\sum_{i=1}^{\D}q_{i}=p$. Inserting
this bound in (\ref{eq:temp3}) gives the result stated in the proposition.
\end{proof}

\subsection{Proof of Theorem \ref{thm:single-dimension-bound}}

Following our strategy, applying Proposition \ref{prop:strategy-complete-norm-bound}
to both $P[\EE_{s,t}^{\lambda,\mu,m}\otimes X]$ and $P[\EE_{s',t'}^{\lambda,\mu,m}\otimes X']$
and using the Schwarz inequality we obtain\textbar{} the following
\[
|\eqref{eq:mat-coef-int}|\leq(|\lambda/\mu|+p)^{2p}\frac{(B!)^{2}}{((B+p)!)^{2}}D_{\mu}(m)\sum_{\substack{\lambda\subset_{p}\lambda'\\
\ell(\lambda')\leq n
}
}\frac{d_{\lambda'}^{2}d_{\lambda'/\lambda}}{d_{\lambda}^{2}D_{\lambda'}(n)}.
\]

Therefore for any $I\subset[m+1,n]^{|w|}$, using the argument from
the beginning of $\S\S\ref{subsec:Second-integration:-overview}$,
writing $f$ for an element of $\{a,b,c,d\}$, we have 
\begin{align}
 & |\int_{h\in\U(n)^{4}}w_{I}(h)\langle h[\EE_{s_{1},T_{2}}^{\lambda,\mu,m}\otimes\EE_{S_{1},s_{1}}^{\lambda,\mu,m}\otimes\EE_{s_{2},T_{4}}^{\lambda,\mu,m}\otimes\EE_{\mu,S_{3},s_{2}}^{\lambda,\mu,m}],\EE_{S_{1},S_{4}}^{\lambda,\mu,m}\otimes\EE_{S_{2},T_{2}}^{\lambda,\mu,m}\otimes\EE_{S_{3},S_{2}}^{\lambda,\mu,m}\otimes\EE_{S_{4},T_{4}}^{\lambda,\mu,m}\rangle dh|\nonumber \\
\leq & D_{\mu}(m)^{4}\prod_{f\in\{a,b,c,d\}}\left((|\lambda/\mu|+p_{f})^{2p_{f}}\frac{(B!)^{2}}{((B+p_{f})!)^{2}}\sum_{\substack{\lambda\subset_{p_{f}}\lambda'_{f}\\
\ell(\lambda'_{f})\leq n
}
}\frac{d_{\lambda'_{f}}^{2}d_{\lambda'_{f}/\lambda}}{d_{\lambda}^{2}D_{\lambda'_{f}}(n)}\right)\label{eq:pre-cleanup}
\end{align}
where $p_{f}$ is the number of occurrences of $f^{+1}$ in the word
(\ref{eq:combinatorial-word}). This can be tidied up a little. Indeed
consider, by using twice both the hook-length formula (\ref{eq:hook-length-formula})
and the hook content formula (\ref{eq:hook-content-formula}) we obtain
\begin{align}
\frac{(B!)^{2}}{((B+p_{f})!)^{2}}\sum_{\substack{\lambda\subset_{p_{f}}\lambda'_{f}\\
\ell(\lambda'_{f})\leq n
}
}\frac{d_{\lambda'_{f}}^{2}d_{\lambda'_{f}/\lambda}}{d_{\lambda}^{2}D_{\lambda'_{f}}(n)} & =\frac{(B!)}{(B+p_{f})!}\frac{1}{d_{\lambda}D_{\lambda}(n)}\sum_{\substack{\lambda\subset_{p_{f}}\lambda'_{f}\\
\ell(\lambda'_{f})\leq n
}
}\frac{\prod_{\square\in\lambda}(n+c(\square)}{\prod_{\square\in\lambda'_{f}}(n+c(\square)}d_{\lambda'_{f}}d_{\lambda'_{f}/\lambda}\nonumber \\
 & \leq\frac{B!}{(B+p_{f})!}\frac{1}{d_{\lambda}D_{\lambda}(n)}\sum_{\substack{\lambda\subset_{p_{f}}\lambda'_{f}\\
\ell(\lambda'_{f})\leq n
}
}d_{\lambda'_{f}}d_{\lambda'_{f}/\lambda}\nonumber \\
 & \leq\frac{1}{D_{\lambda}(n)},\label{eq:cleaningup1}
\end{align}
where the last inequality used (\ref{eq:induced-rep-formula}). We
are also satisfied to use the bound
\begin{equation}
(|\lambda/\mu|+p_{f})^{2p_{f}}\leq\left(|\lambda/\mu|+|w|\right)^{2|w|}\label{eq:cleaningup2}
\end{equation}
 for each $f\in\{a,b,c,d\}$. Inserting (\ref{eq:cleaningup1}) and
(\ref{eq:cleaningup2}) into (\ref{eq:pre-cleanup}) gives
\begin{align*}
 & |\int_{h\in\U(n)^{4}}w_{I}(h)\langle h[\EE_{s_{1},T_{2}}^{\lambda,\mu,m}\otimes\EE_{S_{1},s_{1}}^{\lambda,\mu,m}\otimes\EE_{s_{2},T_{4}}^{\lambda,\mu,m}\otimes\EE_{S_{3},s_{2}}^{\lambda,\mu,m}],\EE_{S_{1},S_{4}}^{\lambda,\mu,m}\otimes\EE_{S_{2},T_{2}}^{\lambda,\mu,m}\otimes\EE_{S_{3},S_{2}}^{\lambda,\mu,m}\otimes\EE_{S_{4},T_{4}}^{\lambda,\mu,m}\rangle dh|\\
\leq & \frac{D_{\mu}(m)^{4}}{D_{\lambda}(n)^{4}}\left(|\lambda/\mu|+|w|\right)^{8|w|}.
\end{align*}
(For general $g$ the exponent $8|w|$ is $4g|w|$.)

Next using Proposition \ref{prop:first-integral-done} we obtain for
$m(I)=n-\D(I)$
\begin{align*}
\I^{*}(w_{I},\lambda)\leq & \sum_{\substack{\mu\subset^{\D(I)}\lambda\\
\ell(\mu)\leq m(I)
}
}\frac{1}{D_{\mu}(m(I))^{3}}\sum_{S_{1},S_{2},S_{3},S_{4},T_{2},T_{4},s_{1},s_{2}\in\SSTab_{[m(I)+1,n]}(\lambda/\mu)}\frac{D_{\mu}(m(I))^{4}}{D_{\lambda}(n)^{4}}\left(|\lambda/\mu|+|w|\right)^{8|w|}\\
= & \frac{1}{D_{\lambda}(n)^{4}}\sum_{\substack{\mu\subset^{\D(I)}\lambda\\
\ell(\mu)\leq m(I)
}
}D_{\mu}(m)\left(|\lambda/\mu|+|w|\right)^{8|w|}|\SSTab_{[m(I)+1,n]}(\lambda/\mu)|^{8}.
\end{align*}
(For general $g$ $|\SSTab_{[m(I)+1,n]}(\lambda/\mu)|^{8}$ is replaced
by $|\SSTab_{[m(I)+1,n]}(\lambda/\mu)|^{4g}$.) Hence from (\ref{eq:index-splitting-two})
we obtain Theorem \ref{thm:single-dimension-bound}. $\square$

\section{The total contribution from large dimensional families\label{sec:The-total-contribution}}

\subsection{Statement of main sectional results}

The main task of this section is to estimate the sum
\[
\sum_{\substack{\text{\text{(\ensuremath{\rho,W)\in\widehat{\SU(n)}}\ensuremath{\backslash}\ensuremath{\Omega(B;n)}}}}
}(\dim W)\mathcal{I}(w,\rho)
\]
appearing in the left hand side of (\ref{eq:tail-bound}). Let $\Lambda(B;n)$
denote the collection of Young diagrams of length at most $n-1$ such
that 
\[
\lambda\in\Lambda(B;n)\mapsto(\rho_{n}^{\lambda},W_{n}^{\lambda})\in\widehat{\SU(n)}\ensuremath{\backslash}\ensuremath{\Omega(B;n)}
\]
is a one-to-one parametrization. The goal of this section is to prove
Theorem \ref{thm:high-dim-sum}, which amounts to establishing a bound
for 
\[
\Sigma_{2}(w,B,n)\eqdf\sum_{\lambda\in\Lambda(B;n)}D_{\lambda}(n)\I(w,\lambda)
\]
and proving its absolute convergence. 

Theorem \ref{thm:high-dim-sum} has the following corollary that will
be useful later.
\begin{cor}
\label{cor:tail-control-witten}Let $g\geq2$ and $B\in\N$ be fixed.
We have 
\[
\zeta(2g-2;n)=\sum_{(\rho,W)\in\Omega(B;n)}\frac{1}{(\dim W)^{2g-2}}+O_{g,B}\left(\frac{1}{n}\left(n^{-2\log B}\right)\right)
\]
as $n\to\infty$.
\end{cor}

\begin{proof}[Proof of Corollary \ref{cor:tail-control-witten} given Theorem \ref{thm:high-dim-sum}.]
Let $w=\id_{\F_{2g}}.$ We use the result, for any $(\rho,W)\in\widehat{\SU(n)}$
\begin{equation}
\int\overline{\tr(\rho((R_{g}(x))))}d\mu_{\SU(n)^{2g}}^{\mathrm{Haar}}(x)=\frac{1}{(\dim W)^{2g-1}}\label{eq:R_g_integral}
\end{equation}
which is due to Frobenius \cite{frobenius1896gruppencharaktere} when
$g=1$; see \cite[\S 2]{parzanchevski2014fourier} or \cite[eq. (2.2)]{CMP-Aut}
for the (easy) extension to general $g\geq2.$ Since $\tr(\id_{\F_{2g}}(x))=n$
for all $x\in\SU(n)^{2g}$ we obtain from (\ref{eq:R_g_integral})
that for any $(\rho,W)\in\widehat{\SU(n)}$, $\I(w,\rho)=\frac{n}{(\dim W)^{2g-1}}$,
hence 
\begin{align}
\Sigma_{2}(\id_{\F_{2g}},B,n) & =n\sum_{(\rho,W)\in\widehat{\SU(n)}\backslash\Omega(B;n)}\frac{1}{(\dim W)^{2g-2}}.\label{eq:heuristic-sigma-2}
\end{align}
Now the corollary follows directly\footnote{Admittedly, this is overkill. The full arguments needed to prove Theorem
\ref{thm:high-dim-sum} are not required for Corollary \ref{cor:tail-control-witten}:
the entirety of $\S$\ref{sec:The-contribution-from-single-large-dim}
is bypassed by (\ref{eq:heuristic-sigma-2}).} from Theorem \ref{thm:high-dim-sum}.
\end{proof}

\subsection{Preliminary estimates}

Even in the simple case of $\gamma=\id$ (\ref{eq:heuristic-sigma-2})
shows that the problem of estimating $\Sigma_{2}$ is related to the
large-$n$ convergence of the Witten zeta function as in Theorem \ref{thm:GLM}.
The techniques used in \cite{GLM} can be adapted to deal with (\ref{eq:heuristic-sigma-2}),
and indeed, are also useful in dealing with general $\gamma$ (or
$w$). The key estimate we take from \emph{(ibid.) }is the following.
Given a YD $\lambda$ with $\ell(\lambda)<n$, we define $x_{i}(\lambda)\eqdf\lambda_{i}-\lambda_{i+1}$,
setting $\lambda_{i}=0$ for $i>\ell(\lambda)$. These are the coefficients
of the highest weight of $(\rho_{n}^{\lambda},W_{n}^{\lambda})$ with
respect to a system of fundamental weights for $\SU(n)$.
\begin{lem}[{\cite[eq. (1), Lemma 8]{GLM}}]
\label{lem:GLM}For a YD $\lambda$ with $\ell(\lambda)<n$, we have
\[
D_{\lambda}(n)\geq\prod_{i=1}^{n-1}(1+x_{i}(\lambda))^{v_{i}}
\]
where $v_{i}$ are positive real numbers satisfying for $0\leq j\leq\frac{n}{2}$
\[
v_{j}=v_{n-j}\geq j\max(1,\log(n-1)-\log j).
\]
\end{lem}

For the reader's convenience, Lemma \ref{lem:GLM} follows from applying
the AM-GM inequality to the Weyl dimension formula. It will turn out,
because of Lemma \ref{lem:GLM}, to be useful to work with the coordinates
$\x(\lambda)\in\N_{0}^{n-1}$ from now on, instead of $\lambda$.
On the other hand, Theorem \ref{thm:single-dimension-bound} involves
quantities $|\lambda/\mu|$ and $|\SSTab_{[m+1,n]}(\lambda/\mu)|$
for $\mu\subset^{\D}\lambda$, $m+\D=n$. We now estimate these quantities
in terms of $\x(\lambda)$.
\begin{lem}
\label{lem:lambdaminusmu-bound}If $\D\in\N_{0}$, $\lambda$ and
$\mu$ are YDs with $\ell(\lambda)\leq n-1$ and $\mu\subset^{\D}\lambda$
then
\[
|\lambda/\mu|\leq\D\sum_{j=1}^{n-1}x_{j}(\lambda)\leq\D\prod_{j=1}^{n-1}(1+x_{j}(\lambda)).
\]
\end{lem}

\begin{proof}
The condition that $\mu\subset^{\D}\lambda$ means that there is a
chain of YDs $\mu=\mu^{\D}\subset^{1}\mu^{\D-1}\subset^{1}\ldots\subset\mu^{1}\subset^{1}\mu_{0}=\lambda$.
First note that $|\lambda/\mu^{1}|\leq\lambda_{1}$ since to obtain
$\mu^{1}$ from $\lambda$, one can delete at most one box from each
column and there are $\lambda_{1}$ non-empty columns of $\lambda$.
One also has $\mu_{1}^{1}\leq\lambda_{1}$ and so repeating the argument
gives $|\mu^{1}/\mu^{2}|\leq\mu_{1}\leq\lambda_{1}$ and iterating
further gives $|\mu^{i}/\mu^{i+1}|\leq\lambda_{1}$ for all $0\leq i\leq\D-1$.
Hence
\[
|\lambda/\mu|=\sum_{i=0}^{\D-1}|\mu^{i}/\mu^{i+1}|\leq\D\lambda_{1}=\D\sum_{j=1}^{n-1}x_{i}(\lambda).
\]
\end{proof}
\begin{lem}
\label{lem:size-of-SST-bound}If $\D\in\N_{0}$, $n\geq\D$, $\lambda$
and $\mu$ are YDs with $\ell(\lambda)\leq n-1$ and $\mu\subset^{\D}\lambda$,
and $m=n-\D$, then
\[
|\SSTab_{[m+1,n]}(\lambda/\mu)|\leq\left(\prod_{i=1}^{n-1}(x_{i}(\lambda)+1)\right)^{2^{\D}}.
\]
\end{lem}

\begin{proof}
Choosing an element of $\SSTab_{[m+1,n]}(\lambda/\mu)$ is the same
as choosing a sequence $\mu=\mu^{\D}\subset^{1}\mu^{\D-1}\subset^{1}\ldots\subset\mu^{1}\subset^{1}\mu_{0}=\lambda$.
We regard all $\x(\mu^{i})\in\N_{0}^{n-1}$ by extending by zeros.
The number of choices of $\mu^{1}$ is $\prod_{i=1}^{n-1}(x_{i}(\lambda)+1)$
since one removes from the $i$\textsuperscript{th} row of $\lambda$
between $0$ and $x_{i}(\lambda)$ boxes inclusive. Regardless of
how $\mu^{1}$ is chosen we have
\[
x_{i}(\mu^{1})=\mu_{i}^{1}-\mu_{i+1}^{1}\leq\lambda_{i}-\lambda_{i+2}=x_{i}(\lambda)+x_{i+1}(\lambda)\leq(1+x_{i}(\lambda))(1+x_{i+1}(\lambda))
\]
so 
\[
\prod_{i=1}^{n-1}(x_{i}(\mu)+1)\leq\left(\prod_{i=1}^{n-1}(x_{i}(\lambda)+1)\right)^{2}.
\]
Thus repeating the previous argument gives that the number of choices
of $\mu^{2}$, given $\mu^{1}$, is at most $\left(\prod_{i=1}^{n-1}(x_{i}(\lambda)+1)\right)^{2}$.
Iterating further, the number of choices of the chain $\mu^{1},\mu^{2},\ldots,\mu^{\D-1}$
is at most
\begin{align*}
 & \prod_{i=1}^{n-1}(x_{i}(\lambda)+1)\left(\prod_{i=1}^{n-1}(x_{i}(\lambda)+1)\right)^{2}\left(\prod_{i=1}^{n-1}(x_{i}(\lambda)+1)\right)^{4}\cdots\left(\prod_{i=1}^{n-1}(x_{i}(\lambda)+1)\right)^{2^{\D-2}}\\
= & \left(\prod_{i=1}^{n-1}(x_{i}(\lambda)+1)\right)^{\sum_{k=0}^{\D-2}2^{k}}\leq\left(\prod_{i=1}^{n-1}(x_{i}(\lambda)+1)\right)^{2^{\D}}.
\end{align*}
\end{proof}
We can use Lemmas \ref{lem:lambdaminusmu-bound} and Lemma \ref{lem:size-of-SST-bound}
to tidy up Theorem \ref{thm:single-dimension-bound}.
\begin{prop}
\label{prop:tidied-up-single-lambda-contrib}For $w\in[\F_{2g},\F_{2g}]$,
we have 
\[
|\I(w,\lambda)|\ll_{w,g}n^{|w|}\frac{\left(\prod_{j=1}^{n-1}(1+x_{j}(\lambda))\right)}{D_{\lambda}(n)^{2g-1}}^{C(w,g)}
\]
as $n\to\infty$, where
\[
C(w,g)\eqdf4g(|w|^{2}+2^{|w|}).
\]
\end{prop}

\begin{proof}
We begin by using Lemmas \ref{lem:lambdaminusmu-bound} and \ref{lem:size-of-SST-bound}
in Theorem \ref{thm:single-dimension-bound} to obtain
\begin{align*}
|\I(w,\lambda)| & \leq\sum_{[I]\in S_{_{n}}\backslash[n]^{|w|}}(n)_{\D(I)}\frac{1}{D_{\lambda}(n)^{2g}}\sum_{\substack{\mu\subset^{\D(I)}\lambda\\
\ell(\mu)\leq n-\D(I)
}
}D_{\mu}(n-\D(I))\left(|\lambda/\mu|+|w|\right)^{4g|w|}|\SSTab_{[n-\D(I)+1,n]}(\lambda/\mu)|^{4gw}\\
 & \ll_{w,g}\sum_{[I]\in S_{_{n}}\backslash[n]^{|w|}}n^{\D(I)}\frac{\left(\prod_{j=1}^{n-1}(1+x_{j}(\lambda))\right)^{4g|w|\D(I)+4g\cdot2^{\D(I)}}}{D_{\lambda}(n)^{2g}}\sum_{\substack{\mu\subset^{\D(I)}\lambda\\
\ell(\mu)\leq n-\D(I)
}
}D_{\mu}(n-\D(I)).
\end{align*}
Here the notation $\ll_{w}$ is with respect to $n\to\infty$. Since
given $\lambda$ with $\ell(\lambda)\leq n-1$, for every $\mu\subset^{\D(I)}\lambda$
with $\ell(\mu)\leq n-\D(I)$, $\dim\Hom_{\U(n-\D(I))}(W_{n-\D(I)}^{\mu},W_{n}^{\lambda})\geq1$
by the branching rules, we have 
\[
\sum_{\substack{\mu\subset^{\D(I)}\lambda\\
\ell(\mu)\leq n-\D(I)
}
}D_{\mu}(n-\D(I))\leq D_{\lambda}(n),
\]
hence
\[
|\I(w,\lambda)|\ll_{w,g}\sum_{[I]\in S_{_{n}}\backslash[n]^{|w|}}n^{\D(I)}\frac{\left(\prod_{j=1}^{n-1}(1+x_{j}(\lambda))\right)^{8|w|\D(I)+8\cdot2^{\D(I)}}}{D_{\lambda}(n)^{2g-1}}.
\]
Recalling that $\D(I)$ is the number of distinct indices of $I$,
we have $\D(I)\leq|w|$. Moreover,
\[
|S_{_{n}}\backslash[n]^{|w|}|\leq|w|^{|w|}.
\]
So we obtain 
\begin{equation}
|\I(w,\lambda)|\ll_{w,g}n^{|w|}\frac{\left(\prod_{j=1}^{n-1}(1+x_{j}(\lambda))\right)}{D_{\lambda}(n)^{2g-1}}^{4g|w|^{2}+4g\cdot2^{|w|}}=n^{|w|}\frac{\left(\prod_{j=1}^{n-1}(1+x_{j}(\lambda))\right)}{D_{\lambda}(n)^{2g-1}}^{C(w,g)}\label{eq:single-lambda-bound}
\end{equation}
as required.
\end{proof}

\subsection{Proof of Theorem \ref{thm:high-dim-sum}\label{subsec:Proof-of-Theorem-high-dim-sum}}

We begin by considering the sum $\Sigma_{2}(w,B,n)=\sum_{\lambda\in\Lambda(B;n)}D_{\lambda}(n)\I(w,\lambda)$.
Fix $g\geq2$. We use Proposition \ref{prop:tidied-up-single-lambda-contrib}
in this sum to obtain
\[
\Sigma_{2}(w,B,n)\ll_{w,g}n^{|w|}\sum_{\lambda\in\Lambda(B;n)}\frac{\left(\prod_{j=1}^{n-1}(1+x_{j}(\lambda))\right)}{D_{\lambda}(n)^{2g-2}}^{C(w,g)}
\]
where $C(w,g)>0$. Let $C=C(w,g)$. We need a key observation relating
the weight coefficients $\x(\lambda)$ to the set $\Omega(B;n)$ and
hence $\Lambda(B;n)$. It is not hard to check that if $\lambda\in\Lambda(B;n)$
then either 
\begin{itemize}
\item $x_{i}(\lambda)>B$ for some $i\leq B$ or $i\geq n-B$ \textbf{or}
\item $x_{i}(\lambda)>0$ for some $B<i<n-B$, 
\end{itemize}
and given $\x\in\Z_{\geq0}^{n-1}$ satisfying these conditions there
is at most one corresponding $\lambda\in\Lambda(B;n)$. (Obtaining
these conditions was the reason for the original choice of $\Omega(B,n)$.)

Now by Lemma \ref{lem:GLM}
\begin{align}
\Sigma_{2}(w,B,n) & \ll_{w,g}n^{|w|}\sum_{\lambda\in\Lambda(B;n)}\frac{1}{\prod_{j=1}^{n-1}(1+x_{j}(\lambda))^{(2g-2)v_{j}-C}}\nonumber \\
 & \leq n^{|w|}\sum_{\lambda\in\Lambda(B;n)}\frac{1}{\prod_{j=1}^{n-1}(1+x_{j}(\lambda))^{2v_{j}-C}}\\
 & \leq n^{|w|}\sum_{\substack{\x\in\N_{0}^{n-1}:\\
x_{j}>0\text{ for some }B<j<n-B\text{ \textbf{or}}\\
x_{j}>B\text{ for some }j\in[B]\cup[n-B,n-1]
}
}\frac{1}{\prod_{j=1}^{n-1}(1+x_{j}(\lambda))^{2v_{j}-C}}.\label{eq:passage-to-x}
\end{align}
Here $v_{j}$ are the constants from Lemma \ref{lem:GLM} satisfying
for $0\leq j\leq\frac{n}{2}$
\[
v_{j}=v_{n-j}\geq j\max(1,\log(n-1)-\log j).
\]
Notice that if $j\in[\frac{n-1}{e},n-\frac{n-1}{e}]$ we have $v_{j}\geq\frac{n-1}{e}$.
The function $x\mapsto x\log\left(\frac{n-1}{x}\right)$ has non-negative
derivative on $[1,\frac{n-1}{e}]$, so the minimum value of $v_{j}$
for $j\in[1,\frac{n-1}{e}]\cup[n-1-\frac{n-1}{e},n-1]$ is $v_{1}\geq\log(n-1)$.
We have $\frac{n-1}{e}\geq\log(n-1)$ for $n\gg1$ so for $n\gg1$
we have 
\begin{equation}
v_{j}\geq\log(n-1),\quad j\in[n-1].\label{eq:v_j-abs-lower-bound}
\end{equation}
The sum in (\ref{eq:passage-to-x}) can be crudely estimated by disregarding
the constraints on $\x$ to obtain for $n\gg1$
\begin{equation}
\Sigma_{2}(w,B,n)\ll_{w}n^{|w|}\prod_{j=1}^{n-1}\zeta(2v_{j}-C)\label{eq:crude-est-to-zeta}
\end{equation}
where for $\Re(s)>1$
\[
\zeta(s)\eqdf\sum_{k=1}^{\infty}\frac{1}{n^{s}}
\]
 is an absolutely convergent sum defining \emph{Riemann's zeta function.
}Notice that for $n\gg_{w,g}1$, $2v_{j}-C\geq2\log(n-1)-C>1$, so
(\ref{eq:crude-est-to-zeta}) shows that $\Sigma_{2}(w,B,n)$ is defined
by an absolutely convergent sum. \emph{This proves the first statement
of Theorem \ref{thm:high-dim-sum}.}

We now turn to finer estimates for $\Sigma_{2}$. Assume $n\geq2B\max(2B,C+1)$
and $2\log(n-1)-C>2$, so that $2v_{j}-C\geq3$ . 

Incorporating the constraints on $\x$ in (\ref{eq:passage-to-x})
gives the improved estimate
\begin{align*}
\frac{\Sigma_{2}(w,B,n)}{n^{|w|}} & \ll_{w,g}\sum_{j=1}^{B}\zeta^{(B+1)}(2v_{j}-C)\prod_{\substack{i\in[n-1],\,i\neq j}
}\zeta(2v_{i}-C)\\
 & +\sum_{j=B+1}^{n-B-1}\zeta^{(2)}(2v_{j}-C)\prod_{\substack{i\in[n-1],\,i\neq j}
}\zeta(2v_{i}-C)\\
 & +\sum_{j=n-B}^{n-1}\zeta^{(B+1)}(2v_{j}-C)\prod_{\substack{i\in[n-1],\,i\neq j}
}\zeta(2v_{i}-C)\\
 & \ll\left(\sum_{j=1}^{B}\zeta^{(B+1)}(2v_{j}-C)+\sum_{j=B+1}^{\lfloor\frac{n}{2}\rfloor}\zeta^{(1)}(2v_{j}-C)\right)\prod_{\substack{i\in[n-1]}
}\zeta(2v_{i}-C)
\end{align*}
where the last estimate used $v_{j}=v_{n-j}$ and $\zeta(2-C)\geq1$,
and
\[
\zeta^{(p)}(s)\eqdf\sum_{k=p}^{\infty}\frac{1}{n^{s}}.
\]

Moreover, for $s\geq3$ the simple bound \cite[pg. 1826]{GLM} $\zeta(s)\leq1+2\cdot2^{-s}$
and our assumptions on $n$ imply
\begin{align*}
\log(\prod_{\substack{i\in[n-1]}
}\zeta(2v_{i}-C)) & \leq2\sum_{j=1}^{\lfloor\frac{n}{2}\rfloor}\log(1+2\cdot2^{-2v_{j}+C})\\
 & \leq4\sum_{j=1}^{\lfloor\frac{n}{2}\rfloor}2^{-2v_{j}+C}\\
 & \leq2^{C+2}\left(\sum_{1\leq j\leq\lfloor\frac{n}{2}\rfloor}2^{-2j}\right)\ll_{w}1.
\end{align*}
Hence
\begin{equation}
\frac{\Sigma_{2}(w,B,n)}{n^{|w|}}\ll_{w}\sum_{j=1}^{B}\zeta^{(B+1)}(2v_{j}-C)+\sum_{j=B+1}^{\lfloor\frac{n}{2}\rfloor}\zeta^{(2)}(2v_{j}-C).\label{eq:2-part-bound}
\end{equation}
We next need bounds for the $\zeta^{(p)}$. By comparison with an
integral we have for $s>1$ and $p\geq2$
\[
\zeta^{(p)}(s)\leq\frac{(p-1)^{1-s}}{s-1}.
\]
This gives
\begin{align}
\sum_{j=1}^{B}\zeta^{(B+1)}(2v_{j}-C) & \leq\sum_{j=1}^{B}\frac{B^{1+C-2v_{j}}}{2v_{j}-C}\leq\frac{1}{2}B^{1+C}\sum_{j=1}^{B}B^{-2v_{j}}\nonumber \\
 & \ll_{B,w}\sum_{j=1}^{B}B^{-2j\log\left(\frac{n-1}{j}\right)}\leq\sum_{j=1}^{B}B^{-2j\log\left(\frac{n-1}{B}\right)}\nonumber \\
 & \ll_{B}B^{-2\log\left(\frac{n-1}{B}\right)}=\left(\frac{n-1}{B}\right)^{-2\log B}\ll_{B}(n-1)^{-2\log B}.\label{eq:frist-part-bound}
\end{align}

To deal with the second part of (\ref{eq:2-part-bound}) we use a
different bound. One has the bound \cite[proof of Lemma 6]{GLM}
\[
\zeta^{(2)}(s)\leq2\cdot2^{-s}
\]
for $s\geq3$. Under our current assumptions on $n$ this gives
\begin{align}
\sum_{j=B+1}^{\lfloor\frac{n}{2}\rfloor}\zeta^{(2)}(2v_{j}-C) & \leq2\sum_{j=B+1}^{\lfloor\frac{n}{2}\rfloor}2^{C-2v_{j}}\ll_{w,g}\sum_{j=B+1}^{\lfloor\frac{n}{2}\rfloor}2^{-2v_{j}}\nonumber \\
 & \leq\sum_{j=B+1}^{\lceil\sqrt{n-1}\rceil}2^{-2j\log\left(\frac{n-1}{j}\right)}+\sum_{j=\lceil\sqrt{n-1}\rceil+1}^{\lfloor\frac{n}{2}\rfloor}2^{-2j}\nonumber \\
 & \ll\sum_{j=B+1}^{\lceil\sqrt{n-1}\rceil}2^{-j\log\left(n-1\right)}+2^{-2\sqrt{n-1}}\ll(n-1)^{(B+1)\log2}+2^{-2\sqrt{n-1}}.\label{eq:second-part-bound}
\end{align}
Thus in total by combining (\ref{eq:2-part-bound}), (\ref{eq:frist-part-bound}),
and (\ref{eq:second-part-bound}) we achieve
\begin{align*}
\Sigma_{2}(w,B,n) & \ll_{B,w,g}n^{|w|}\left((n-1)^{-2\log B}+(n-1)^{(B+1)\log2}+2^{-2\sqrt{n-1}}\right)\\
 & \ll_{B,w,g}n^{|w|-2\log B}
\end{align*}
as $n\to\infty$\emph{. This proves the second statement of Theorem
\ref{thm:high-dim-sum}}. $\square$

\subsection{Proof of Theorem \ref{thm:Main-theorem}\label{subsec:Proof-of-Theorem-main}}
\begin{proof}[Proof of Theorem \ref{thm:Main-theorem}]

We are given $M\in\N$ and we choose $B\in\N$ such that 
\[
2\log B\geq M+|w|\geq M.
\]
We first take care of the term $\zeta(2g-2;n)^{-1}$ appearing in
Corollary \ref{cor:Fourier-expansion-of-main-integral}. Corollary
\ref{cor:tail-control-witten} shows that 
\begin{align*}
\zeta(2g-2;n) & =\sum_{(\rho,W)\in\Omega(B;n)}\frac{1}{(\dim W)^{2g-2}}+O_{g,B}\left(\frac{1}{n}\left(n^{-2\log B}\right)\right)\\
 & =\sum_{\substack{\mu,\nu\\
\ell(\mu),\ell(\nu)\leq B,\mu_{1},\nu_{1}\leq B^{2}
}
}\frac{1}{D_{[\mu,\nu]}(n)^{2g-2}}+O_{g,B}\left(n^{-M-1}\right).
\end{align*}
Now as the sum is over a fixed finite set, Corollary \ref{cor:dim-of-rational-representation}
implies that for $n>2B^{3}$ we have
\begin{equation}
\sum_{(\mu,\nu)\in\Omega(B)}\frac{1}{D_{[\mu,\nu]}(n)^{2g-2}}=F_{g,B}(n)\label{eq:Fdef}
\end{equation}
for $F_{g,B}\in\Q(t)$ with only possible poles in $[-2B^{3},2B^{3}]$
and no zeros there. From Theorem \ref{thm:GLM} $\lim_{n\to\infty}F_{g,B}(n)=1$;
this can also be seen directly from (\ref{eq:Fdef}). Therefore we
have 
\begin{equation}
\zeta(2g-2;n)^{-1}=\frac{1}{F_{g,B}(n)}\left(1+O_{g,B}\left(F_{g,B}(n)^{-1}n^{-M-1}\right)\right)^{-1}=\frac{1}{F_{g,B}(n)}+O_{g}\left(n^{-M-1}\right)\label{eq:rat1}
\end{equation}
as $n\to\infty$.

By Theorem \ref{thm:rational-small-dim}, Theorem \ref{thm:high-dim-sum},
and Corollary \ref{cor:Fourier-expansion-of-main-integral} we have
\begin{align}
\E_{g,n}[\tr_{\gamma}] & =\frac{1}{\zeta(2g-2;n)}\left(Q_{B,w}(n)+O_{w,g}\left(n^{|w|-2\log B}\right)\right)\nonumber \\
 & \frac{1}{\zeta(2g-2;n)}\left(Q_{B,w}(n)+O_{w,g}\left(n^{-M}\right)\right)\label{eq:full-formula}
\end{align}
as $n\to\infty$, where $Q_{w,B}\in\Q(t)$. Combining (\ref{eq:rat1})
with (\ref{eq:full-formula}) gives as $n\to\infty$
\begin{align}
\E_{g,n}[\tr_{\gamma}] & =\left(\frac{1}{F_{g,B}(n)}+O_{g}(\left(n^{-M-1}\right)\right)\left(Q_{w,B}(n)+O_{g,w}\left(n^{-M}\right)\right)\nonumber \\
 & =\frac{Q_{w,B}(n)}{F_{g,B}(n)}+O_{g}\left(Q_{w,B}(n)n^{-M-1}\right)+O_{g,w}\left(n^{-M}\right).\label{eq:pre-boot}
\end{align}
Using $F_{g,B}(n)\to1$ as $n\to\infty$ again, and $\E_{g,n}[\tr_{\gamma}]\leq n$,
we obtain $Q_{w,B}(n)=O_{g,w}(n)$ as $n\to\infty$ and hence we bootstrap
(\ref{eq:pre-boot}) to
\[
\E_{g,n}[\tr_{\gamma}]=\frac{Q_{w,B}(n)}{F_{g,B}(n)}+O_{g,w}\left(n^{-M}\right).
\]
This completes the proof; $\frac{Q_{w,B}(n)}{F_{g,B}(n)}$ is $O(n)$
and can be expanded as a Laurent polynomial as in (\ref{eq:main-theorem-eq})
up to order $O\left(n^{-M}\right)$. Moreover it is clear that the
Laurent polynomials arising from different $M$ must be coherent,
i.e., arise from a fixed infinite sequence $a_{-1}(\gamma),a_{0}(\gamma),a_{1}(\gamma)\ldots$
of rational numbers.
\end{proof}
\bibliographystyle{alpha}
\bibliography{database}

\noindent Michael Magee, 

\noindent Department of Mathematical Sciences, 

\noindent Durham University, Lower Mountjoy, DH1 3LE Durham, United
Kingdom

\noindent \texttt{michael.r.magee@durham.ac.uk}
\end{document}